\documentclass[11pt,reqno]{article}
\usepackage{amsmath, latexsym, amsfonts, amssymb,amsthm, amscd,epsfig,enumerate}
\usepackage{subfigure,color}

\advance\hoffset -.75cm

\usepackage[usesnames]{xcolor}

\definecolor{refkey}{gray}{.75}
\definecolor{labelkey}{gray}{.75}

\newcommand{\C}{\mathbb C}
\newcommand{\R}{\mathbb R}
\newcommand{\Z}{\mathbb Z}
\newcommand{\N}{\mathbb N}
\newcommand{\Prob}{\mathbb P}
\newcommand{\E}{\mathbb E}

\newcommand{\diff}{\mathrm{d}}

\newcommand{\eps}{\varepsilon}

\newcommand{\cF}{\mathcal{F}}

\newcommand{\cH}{\mathcal{H}}

\newcommand{\pr}{\mathbb P}
\newcommand{\T}{\mathbb T}

\newcommand{\lra}{\longrightarrow}
\newcommand{\ident}{{\mathchoice {\rm 1\mskip-4mu l} {\rm 1\mskip-4mu l}
{\rm 1\mskip-4.5mu l} {\rm 1\mskip-5mu l}}}
\newcommand{\undertilde}[1]{\underset{\widetilde{}}{#1}}

\renewcommand{\thesubfigure}{\arabic{subfigure}}
\makeatletter
\renewcommand{\@thesubfigure}{\tiny Figure \thesubfigure: \space}
\renewcommand{\p@subfigure}{}
\makeatother

\newtheorem{teo}{Theorem}[section]
\newtheorem{lem}[teo]{Lemma}
\newtheorem{cor}[teo]{Corollary}
\newtheorem{rem}[teo]{Remark}
\newtheorem{pro}[teo]{Proposition}
\newtheorem{defn}[teo]{Definition}
\newtheorem{exmp}[teo]{Example}

\newtheorem{assump}[teo]{Assumption}

\title
{\textbf{Recent results on branching random walks}}

\author
{Daniela Bertacchi\\
Dipartimento di Matematica e Applicazioni\\
Universit\`a di Milano--Bicocca\\
via Cozzi 53, 20125 Milano, Italy\\
daniela.bertacchi\@@unimib.it
\and
Fabio Zucca \\
Dipartimento di Matematica \\
Politecnico di Milano\\
piazza Leonardo da Vinci 32, 20133 Milano, Italy\\
fabio.zucca\@@polimi.it}

\date{}

\begin{document}

\maketitle

%

\begin{abstract}
This paper is a collection of recent results on discrete-time and
continuous-time branching random walks.
Some results are new and others are known. Many aspects of this
theory are considered: local, global and strong local survival,
the existence of a pure global survival phase and the approximation of
branching random walks by means of multitype contact processes
or spatially confined branching random walks.
Most results are obtained using a generating function approach: the probabilities
of extinction are seen as fixed points of an infinite dimensional power
series. Throughout this paper we provide many nontrivial examples and
counterexamples.
\end{abstract}



\tableofcontents

\section{Introduction}
\label{sec:intro}

Branching random walks can be considered as processes which simultaneously generalize
the concepts of branching process and of random walk.
A branching process is a very simple population model (introduced in \cite{cf:GW1875})
where particles breed and die (independently of each other) according to some random law.
At any time, this process is completely characterized by the total number of particles alive.
Branching random walks (in short, BRWs) add space to this picture: particles live in a spatially structured environment
and the reproduction law, which may depend on the location, not only tells how many children the particle
has, but also where it places them. The state of the process, at any time, is thus described by the
collection of the numbers of particles alive at $x$, where $x$ varies among the possible sites.
Although particles do not actually move, an observer would see a random movement of the population.
Moreover if we identify every particle with one of its children (if there are any), then we may view
the BRW as a system of random walkers which may disappear (i.e.~the corresponding
particle has no children) or split into two or more independent walkers (i.e.~the corresponding
particle has two or more children).

The basic question that one answers studying the branching process is whether it survives (i.e.~with
positive probability at any time there is someone alive); while the classical question for random walks
is whether the walker returns (with positive probability or, equivalently, with probability one)
infinitely many times to some fixed site. Transposed into BRW theory, the first question
asks whether there is global survival, that is, with
positive probability at any time there is someone alive \textit{somewhere}); while the second question
deals with local survival, that is, with
positive probability the process returns infinitely many times to some fixed site
(this event, in contrast with the situation for random walks, may have probability one only in trivial
examples).

In the literature one can find BRWs both in continuous and discrete time.
The continuous-time setting has been studied by many authors
(see \cite{cf:HuLalley, cf:Ligg1, cf:Ligg2, cf:MadrasSchi, cf:MountSchi} just to name a few).
As we see in Section~\ref{subsec:continuous},
in this case one studies a family of BRWs which depends on the choice
of a parameter $\lambda$.
There are two (possibly coinciding) parameters of interest:
 $\lambda_w\le\lambda_s$. If $\lambda<\lambda_w$ there is almost sure extinction,
if $\lambda_w<\lambda\le \lambda_s$ there is global but not local survival and
if $\lambda_s<\lambda$ there is local and global survival
 (see for instance \cite{cf:PemStac1, cf:Stacey03, cf:BZ, cf:BZ2}).
The discrete-time case has been initially considered as a natural
generalization of branching processes
(see \cite{cf:AthNey, cf:Big1977, cf:Big1978, cf:BigKypr97, cf:BigRah05, cf:Harris63}),
but, since every continuous-time BRW admits a discrete-time counterpart which has the same
behavior, results in this setting naturally extend to continuous time.

In recent years, there has been a growing interest about BRWs in random environment
(see for instance \cite{cf:CMP98, cf:GMPV09, cf:DHMP99, cf:MP00, cf:MP03, cf:M08}).
This is an interesting subject that we do not discuss in this chapter.

Being at the crossroad between branching process and random walk theories, BRW theory
benefits of the techniques of both fields (to be honest, there is a third road at this intersection,
since BRWs can also be seen as interacting particle systems --  although particles do not interact).
Indeed, recalling that the probability of extinction of a branching process is the fixed point of
a generating function associated to the offspring distribution, one can associate to the offspring
distribution of the BRW a (possibly infinite-dimensional) generating function $G$ (this is what we do in
Section~\ref{subsec:genfun}). Moreover it is possible to prove that the extinction probabilities
of the BRW are fixed points of this generating function. This is a fundamental tool that we use in Section~\ref{sec:survival}.
On the other hand some tools borrowed from random walk theory, such as generating functions
of first return probabilities and superharmonic functions, are particularly useful in the no death case
(that is, the case where every particle has at least one child).

The chapter is a collection of recent results on BRWs: some of them are already known,
 some are new and their
proofs can be found in Section~\ref{sec:proofs}. A brief outline of the chapter
 and of main results follows.
The chapter is  divided into five main sections.
Section~\ref{sec:basic} is a short technical
introduction to the subject. There is a description of discrete-time and continuous-time BRWs (Sections~\ref{subsec:discrete}
and \ref{subsec:continuous} respectively). Classical processes as edge-breeding and site-breeding continuous-time BRWs
are discussed and it is shown that, from the point of view of survival vs.~extinction, the class of discrete-time
BRWs extends the class of continuous-time BRWs.
In Section~\ref{subsec:otherdynamics}, other models are presented along
with their relation with BRWs. In Sections~\ref{subsec:trails} and \ref{subsec:genfun} two important tools are discussed:
trails and generating functions. While the first one is more important from a theoretical point of view, the second one
is repeatedly used to study the behavior of a BRW. In particular a \textit{maximum principle} for solutions of
certain inequalities involving these generating functions is proved.

Section~\ref{sec:specialprocesses} presents two particular families of BRWs: $\cF$-BRWs (Section~\ref{subsec:FBRWs})
and BRWs with no death (Section~\ref{subsec:nodeath}).
The first class (which has been introduced in
\cite{cf:Z1}), is a natural generalization of the classes of continuous-time BRWs on weighted $\cF$-graphs (see
\cite[Proposition 4.5]{cf:BZ2}) and on $\cF$-multigraphs (see \cite[Definition 3.1]{cf:BZ}). This class contains
\textit{quasi-transitive BRWs} (for instance, edge-breeding continuous-time BRWs on quasi-transitive graphs) and
BRWs which are \textit{locally isomorphic to branching processes} (for instance,
site-breeding continuous-time BRWs on regular graphs); nevertheless, as
 Example~\ref{rem:pureFBRW} shows, the class of $\mathcal F$-BRWs is strictly larger than the class of quasi-transitive
BRWs even for edge-breeding continuous-time BRWs (see also Example~\ref{exmp:pureglobalsurvival}).
As for BRWs with no death,
which are a natural generalization of random walks, we note that
 even though they represent a limited subclass
of BRWs, many results can be extended immediately to the general class of BRWs using a comparison introduced
by Harris for branching processes (see \cite[Chapter I.12]{cf:AthNey} and Section~\ref{subsec:nodeath}).

Section~\ref{sec:survival} is devoted to the study of the behavior of BRWs (survival vs.~extinction).
In Section~\ref{subsec:survivalprob} the probabilities of survival are
viewed as fixed
points of an infinite-dimensional generating function.
Local survival is completely characterized through the knowledge of
the first-moment matrix $M$ (Theorem~\ref{th:equiv1local}).
 For global survival we give
an equivalent condition in terms of the generating function $G$ (Theorem~\ref{th:equiv1global}(1)).
In terms of $M$ we can only provide an equivalent condition for $\cF$-BRWs and a necessary
one in the general case (Theorem~\ref{th:equiv1global3}).
Example~\ref{exmp:0} shows that many conjectures
about sufficient conditions for global survival are false.
In continuous time,
Corollary~\ref{cor:pemantleimproved} identifies $\lambda_s$ and states almost sure
local extinction at $\lambda=\lambda_s$; a characterization of $\lambda_w$ through
the existence of solutions of certain inequalities is given in Theorem~\ref{th:equiv1global2}.
In the case of $\cF$-BRWs, Corollary~\ref{cor:globalsurvivalcontinuous} provides a more
explicit expression for $\lambda_w$  (this expression is a lower bound in the general case), and
states global extinction at $\lambda=\lambda_w$.
Example~\ref{exm:4} shows that in the general case global survival is possible at $\lambda_w$.
Clearly local survival implies global survival and the converse is false.
When the two events coincide (and have positive probability) we say that there is
strong local survival. Proposition~\ref{pro:qtransitive} claims that quasi-transitivity
and local survival imply strong local survival; Theorem~\ref{th:MenshikovVolkov}
characterizes strong local survival, generalizing \cite[Theorem 3.1]{cf:MenshikovVolkov}
which was stated for the no death case.
Examples~\ref{ex:nonstronglocalFBRW} and \ref{ex:nonstronglocalFBRW2} show that even when the
BRW is locally isomorphic
to a branching process (i.e.~the reproduction law does not depend on the site)
non-strong local survival is possible. Moreover
Example~\ref{rem:nonstrongandstrong} is an edge-breeding continuous-time BRW on a homogeneous
tree with a loop where, for small and large values of
$\lambda$ there is strong local survival while for intermediate values we have non-strong local
survival.
This shows that, unlike local and global survival, for strong local survival there is no
monotonicity in $\lambda$.
Irreducibility guarantees that if there is local survival at some $y$ (or global survival)
starting from some $x_0$, then there is local survival at any $w$ (or global survival) starting from any $x$.
Clearly the probabilities of global or local survival may depend on the staring point even in the irreducible case
(see Example~\ref{exm:BP1-2}).
Example~\ref{exm:reducible} shows that in the reducible case there can be local extinction at $x$
starting from $x$ for
all $x \in X$, but local survival at some $y$ starting from some $x \not = y$;
in addition, this example shows that in the reducible case it there might be global extinction
starting from $x$ but global survival starting from some $y\neq x$.
Example~\ref{exm:irreduciblenotstrongeverywhere} proves that, even in the irreducible case,
if there are vertices where particles have at least one child almost surely, then
it might happen that there is strong local survival starting from some vertex and
non-strong local survival starting from other vertices.
In Example~\ref{exm:4.5} we find a BRW which survives globally even if the
law at each site gives a branching process which dies out.
The main tool that we use in many examples is the discussion in
Remark~\ref{rem:strongconditioned} which relates the probability of visiting a subset
$A \subseteq X$, the probability of local survival at $A$ and the probability of global survival.
Section~\ref{subsec:pureweak} is devoted to pure global survival, that is when
the process survives globally but not locally.
For $\cF$-BRWs the existence of a pure global survival phase
is equivalent to nonamenability (Corollary~\ref{cor:nonam}).
Theorem~\ref{th:nonam} gives an equivalent condition on the first moment matrix $M$,
in the general case, for nonamenability. Unfortunately, the existence of a pure
global survival phase is not equivalent, in general, to nonamenability:
Example~\ref{exm:amenable} shows that there exists an amenable edge-breeding, continuous-time BRW
without pure global survival and, conversely, according to
Example~\ref{exm:nonamenable}, there exists
a nonamenable edge-breeding, continuous-time BRW with pure global survival.
Theorem~\ref{th:subexponentialgrowth} gives a sufficient condition for no pure global
survival of a continuous-time BRW. In Section~\ref{subsec:locisomBP} we treat the special case
where the reproduction law is independent of the site.

In Section~\ref{sec:approx} the question of the approximation of a BRW is studied.
In particular in Section~\ref{subsec:spatial} we obtain an approximation
of a general BRW, which is not necessarily irreducible,
by means of a sequence of spatially confined BRWs (Theorem~\ref{th:spatial}).
This results is a corollary of a generalization of a theorem due to Sarymshakov and Seneta
(see \cite[Theorem 6.8]{cf:Sen}) which deals with nonnegative matrices
 (Theorem~\ref{th:genseneta}).
Here we obtain, as a particular case, that
if we have a surviving process, then by confining it
to a sufficiently large (possibly finite and not necessarily connected) proper subgraph
the resulting BRW survives as well;
this result was already known for irreducible BRWs confined to connected subgraphs.
In Section~\ref{subsec:truncated} we study the
approximation of the BRW with a sequence
of truncated BRWs (which are, in fact, multitype contact processes).
The key to obtain such a result
is the comparison of our process with a suitable oriented percolation
(as explained in \cite{cf:BZ3, cf:Z1}).
The strategy is then applied to some classes of regular BRWs in discrete-time and continuous-time.

In Section~\ref{sec:proofs} all the proofs of new results can be found, along with some technical lemmas.

\section{Basic definitions and preliminaries}
\label{sec:basic}

\subsection{Discrete-time Branching Random Walks}
\label{subsec:discrete}

We start with the construction of a generic discrete-time BRW $\{\eta_n\}_{n \in \N}$
(see also \cite{cf:BZ2} where it is
called \textit{infinite-type branching process}) on a set $X$ which is
at most countable; $\eta_n(x)$ represents the number of particles alive
at $x$ at time $n$. To this aim we
consider a family $\mu=\{\mu_x\}_{x \in X}$
of probability measures 
on the (countable) measurable space $(S_X,2^{S_X})$
where $S_X:=\{f:X \to \N:\sum_yf(y)<\infty\}$.
To obtain generation $n+1$ from generation $n$ we proceed as follows:
a particle at site $x\in X$ lives one unit of time,
then a function $f \in S_X$ is chosen at random according to the law $\mu_x$
and the original particle is replaced by $f(y)$ particles at
$y$, for all $y \in X$; this is done independently for all particles of
generation $n$.
Note that the choice of $f$ assigns simultaneously the total number of children
and the location where they will live.
We denote the BRW by the couple $(X,\mu)$.

Equivalently we could introduce the BRW by choosing the number of children and then their location.
Indeed define $\cH:S_X \rightarrow \N$ as $\cH(f):=\sum_{y \in X} f(y)$
which represents the total number of children associated to $f$.
Denote by $\rho_x$ the measure on $\N$ defined by
$\rho_x(\cdot):=\mu_x(\cH^{-1}(\cdot))$; this is the law of the random number of children
of a particle living at $x$. For each particle, independently, we pick a number $n$ at random,
according to the law $\rho_x$, and then we choose a function $f \in \cH^{-1}(n)$ with probability
$\mu_x(f)/\rho_x(n)\equiv\mu_x(f)/\sum_{g \in \cH^{-1}(n)}\mu_x(g)$ and
we replace the particle at $x$ with $f(y)$ particles at $y$ (for all $y \in X$).

To be precise, to construct the process,
pick a family $\{ f_{i,n,x}\}_{i,n \in \N , x \in X}$
of independent $S_X$-valued random variables such that, for every $x \in X$, $\{f_{i,n,x}\}_{i,n \in \N}$
have the common law $\mu_x$, and an initial state $\eta_0$ such that $\sum_{x \in X} \eta_0(x) <+\infty$.
The discrete-time BRW $\{\eta_n\}_{n \in \N}$ is defined
iteratively as follows
\begin{equation}\label{eq:evolBRW}
\eta_{n+1}(x)=\sum_{y \in X} \sum_{i=1}^{\eta_n(y)} f_{i,n,y}(x)=
\sum_{y \in X} \sum_{j=0}^\infty \ident_{\{\eta_n(y)=j\}} \sum_{i=1}^{j} f_{i,n,y}(x).
\end{equation}
Even though in this chapter the initial state  will always be deterministic,
considering a random initial distribution $\sum_{x \in X} \eta_0(x) <+\infty$ a.s., would not
be a significant generalization.


While in random walk theory a fundamental role is played by the transition matrix,
in BRW theory a similar role is played by the \textit{first-moment matrix}
$M=(m_{xy})_{x,y \in X}$,
where
$m_{xy}:=\sum_{f\in S_X} f(y)\mu_x(f)$ is
 the expected number of particles from $x$ to $y$
(that is, the expected number of children that a particle living
at $x$ sends to $y$).
 We
suppose that $\sup_{x \in X} \sum_{y \in X} m_{xy}<+\infty$; most of the results
of this chapter still hold without this hypothesis, nevertheless
it allows us to avoid dealing with an infinite expected number of offsprings.
Note that the expected number of children generated by a particle living at $x$
is $\sum_{y \in X} m_{xy} = \sum_{n \ge 0} n \rho_x(n)=:\bar \rho_x$.
Given a function $f$ defined on $X$ we denote by $Mf$ the function $Mf(x):=\sum_{y \in X} m_{xy}f(y)$
whenever the RHS makes sense.
We denote by
 $m^{(n)}_{xy}$
the entries of the $n$th power matrix $M^n$ and we define
\begin{equation}\label{eq:generalgeomparam}
M_s(x,y) := \limsup_{n \to \infty} \sqrt[n]{m_{xy}^{(n)}}, \quad
M_w(x) := \liminf_{n \to \infty} \sqrt[n]{\sum_{y \in X} m_{xy}^{(n)}}, \qquad \forall x,y \in X.
\end{equation}

From equation~\eqref{eq:evolBRW},
it is straightforward to prove that the expected number of particles,
starting from an initial state $\eta_0$,
satisfies the recurrence equation $\E^{\eta_0}(\eta_{n+1}(x))
=\sum_{y \in X}m_{yx} \E^{\eta_0}(\eta_n(y))$
hence
\[ 
\E^{\eta_0}(\eta_n(x)) = \sum_{y \in X} m^{(n)}_{yx} \eta_0(y).
\] 

\begin{rem}\label{rem:BRWasRW}
Note that a BRW can be seen as a random walk on $\N^X$ (to be precise, on $S_X \subseteq \N^X$),
where $\mathbf{0}$ is an absorbing state.
If $\rho_x(0)>0$ for all $x \in X$ then every state in $S_X \setminus \{\mathbf{0}\}$ is transient. Basically this is
due to the fact that the probability of going into the state $\mathbf{0}$ starting from a state
$\eta \in S_X$ in one step is $\prod_{x \in X}\rho_x(0)^{\eta(x)}>0$; hence the probability of visiting
infinitely often the state $\eta$ without ending in the trap state $\mathbf{0}$ is $0$
(for a formal proof in the case of a branching process see \cite[Theorem 6.2]{cf:Harris63}).
\end{rem}

We introduce here some terminology borrowed from random walk and graph theory.
In general our definitions extend the classical ones which apply to graphs in
the following way: a discrete-time counterpart of an edge-breeding continuous-time BRW
(see Section~\ref{subsec:continuous}) has the 
property
$\mathcal{P}$ if and only if the underlying graph has the usual
property $\mathcal{P}$.
The BRW $(X,\mu)$ is called \textit{non-oriented} or \textit{symmetric} if $m_{xy}=m_{yx}$ for every $x,y \in X$.
$(X,\mu)$ is called \textit{nonamenable} if and only if
\begin{equation}\label{eq:amenable}
\inf
\left \{
\frac{\sum_{x \in S, y \in S^\complement} m_{xy}}{|S|}: S \subseteq X, |S| < \infty
\right \}=:\iota_{(X,\mu)}>0,
\end{equation}
and it is called \textit{amenable} otherwise.
The value $\iota_{(X,\mu)}$ is called \textit{isoperimetric constant} since in the
case of a continuous-time, edge-breeding BRW (see the end of Section~\ref{subsec:continuous})
this is the usual isoperimetric constant of the underlying multigraph and, in that case,
the nonamenability of the BRW is equivalent to the nonamenability of the multigraph
(see \cite[Section 3.3]{cf:BZ} for the definition).

For a generic discrete-time BRW, the set $X$ is not \textit{a priori} a graph;
nevertheless, the family of probability measures, $\{\mu_x\}_x$ induces in a natural way
a graph structure on $X$ that
we denote by $(X,E_\mu)$ where $E_\mu:=\{(x,y):m_{xy}>0\}\equiv\{(x,y): \exists f \in S_X, \mu_x(f)>0, f(y)>0\}$.
Roughly speaking, $(x,y)$ is and edge if and only if a particle living at $x$ can
send a child at $y$ with positive probability (from now on \textit{wpp}).
We say that there is a path from $x$ to $y$, and we write $x \to y$, if it is
possible to find a finite sequence $\{x_i\}_{i=0}^n$ (where $n \in \N$)
such that $x_0=x$, $x_n=y$ and $(x_i,x_{i+1}) \in E_\mu$
for all $i=0, \ldots, n-1$. If $x \to y$ and $y \to x$ we write $x \rightleftharpoons y$.
Observe that there is always a path of length $0$ from $x$ to itself.

We call the matrix $M=(m_{xy})_{x,y \in X}$ \textit{irreducible} if and only if
the graph $(X,E_\mu)$ is \textit{connected}, otherwise we call it \textit{reducible}. We denote by $\mathrm{deg}(x)$
the degree of a vertex $x$, that
is, the cardinality of the set $ \mathcal{N}_x:=\{y\in X: (x,y) \in E_\mu
\}$.
Note that if $(X,\mu)$ is non-oriented then the graph $(X,E_\mu)$ is non-oriented (that is, $(x,y) \in E_\mu$
if and only if $(y,x) \in E_\mu$).

\begin{defn}\label{def:survival}
The colony can survive in different ways: we say that the
colony \textsl{survives locally wpp} at $y \in X$ starting from $x \in X$
if
\[
\pr^{\delta_x}(\limsup_{n \to \infty} \eta_n(y)>0)>0;
\]
we say that it \textsl{survives globally wpp} starting from $x$ if
\[
\pr^{\delta_x} \Big (\sum_{w \in X} \eta_n(w)>0, \forall n \in N \Big )>0
\]
(or, equivalently, $\pr^{\delta_x}(\limsup_{n \to \infty} \sum_{w \in X}\eta_n(w)>0$).

\noindent Let us define the probabilities of extinction
$q(x,y):=1-\pr^{\delta_x}(\limsup_{n \to \infty} \eta_n(y)>0)$ and
$\bar q(x):=1-\pr^{\delta_x} \Big (\sum_{w \in X} \eta_n(w)>0, \forall n \in N \Big )$.
Following \cite{cf:GMPV09}, we say that the
there is \textsl{strong local survival wpp} at $y \in X$ starting from $x \in X$
if
\[
q(x,y)=\bar q(x)<1.
\]
Finally we say that the BRW is in a \textit{pure global survival phase} starting from $x$ if
\[
\bar q(x)<q(x,x)=1.
\]
From now on when we talk about survival, ``wpp'' will be tacitly understood.
Often we will say simply that local survival occurs ``starting from $x$'' or ``at $x$'':
in this case we mean that $x=y$.
\end{defn}

Roughly speaking, there is global survival if there are particles alive somewhere at all times wpp and
there is local survival at $y$ if there are particles alive at $y$ at arbitrarily large times wpp.
Strong local survival at $y$ starting from $x$ requires that the probability of local survival at $y$
equals the probability of global survival starting from $x$ and that they are both positive.
Equivalently, there is strong survival at $y$ starting from $x$ if and only if the probability
of local survival at $y$ starting from $x$ conditioned on global survival starting from $x$ is $1$.
One can show
that strong local survival implies that for almost all realizations the process either survives locally
and globally or it goes extinct. The typical case (but not the only one)
where there is global but no strong local survival is being in a pure global survival phase.

Clearly local survival at some $y$ starting from $x$ implies global survival starting from $x$,
since $\bar q(x) \le q(x,y)$ for all $x,y \in X$. It is easy to construct examples where
$\bar q(x)<1=q(x,y)$ (see Section~\ref{subsec:pureweak}). One may wonder whether it is
possible to find examples of BRWs where $\bar q(x) < q(x,y) < 1$;
according to Examples~\ref{rem:nonstrongandstrong} and \ref{ex:nonstronglocalFBRW} the answer is positive.

We observe that
if $x \to y$ 
then local survival at $x$ implies local survival
at $y$ starting from any $w$ such that $w \to x$  ($q(w,y) \le q(w,x)$ for all $w\in X$). Analogously,
if $x \to y$ then
global survival starting from $y$ implies global survival starting from $x$
(indeed $\bar q(x) \le \bar q(y)$).
Moreover if $x \to y$, $w \to w^\prime$ and $q(w^\prime,x)<1$ then $q(w,y)<1$.
In particular if $x \rightleftharpoons y$ then local (resp.~global) survival
starting from $x$ is equivalent to local (resp.~global) survival starting from $y$.
As a consequence, if $M$ is irreducible then the process survives locally (resp.~globally) at one vertex if and
only if it survives locally (resp.~globally) at every vertex.
In this case $q(x,y)=q(x,x)$ for all $x,y \in X$ (see Section~\ref{subsec:survivalprob} for details).
Note that even if in the irreducible case one cannot guarantee that $q(x,x)=q(y,y)$ or $\bar q(x)=\bar q(y)$ when $x \not = y$
(see for instance Example~\ref{exm:BP1-2}).

\begin{assump}\label{assump:1}
We assume henceforth that for all $x \in X$ there is a vertex $y \rightleftharpoons x$ such that
$\mu_y(f: \sum_{w:w \rightleftharpoons y} f(w)=1)<1$,
 that is, in every equivalence class (with respect to $\rightleftharpoons$)
there is at least one vertex where a particle
can have inside the class a number of children different from one wpp.
\end{assump}

\begin{rem}\label{rem:assump1}
The previous assumption guarantees that the restriction of the BRW
to an equivalence class is \textit{nonsingular} (see \cite[Definition II.6.2]{cf:Harris63}).
%
There is a technical reason behind the previous assumption.
The classical Galton--Watson branching process is a particular BRW where $X:=\{x\}$ is a singleton
and $S_X$ and $\mu_x$ can be identified with $\N$ and a probability measure on $\mathbb N$
respectively. It is well-known that
\begin{itemize}
 \item if $\mu_x(1)=1$ then $m_{xx}=1$ and there is survival with probability $1$;

\item if $\mu_x(1)<1$ then there is survival wpp if and
only if $m_{xx}>1$,
\end{itemize}
(see also Example~\ref{ex:branchingprocess}).
Hence the condition  $m_{xx}>1$ is equivalent to survival under
Assumption~\ref{assump:1}.
\end{rem}


For a generic BRW, we call \textit{diffusion matrix} the matrix $P$ with entries $p(x,y)=m_{xy}/\bar \rho_x$.
Note that $P$ is a stochastic matrix which defines a random walk on $X$,
but it is not true in general that
the offsprings are dispersed independently according to $P$. This last updating rule
characterizes a particular, but meaningful, subclass of discrete-time processes that we call
\textit{BRWs with independent diffusion}:
a particle at site $x$ lives one unit of time
and is replaced by
a random number of children (with law $\rho_x$) which
are dispersed independently on $X$, according to a stochastic matrix $P$.
This rule is a particular case of the
general one, since here one simply chooses 
\begin{equation}\label{eq:particular1}
\mu_x(f)=\rho_x \left (\sum_y f(y) \right )\frac{(\sum_y f(y))!}{\prod_y f(y)!} \prod_y p(x,y)^{f(y)},
\quad \forall f \in S_X.
\end{equation}
\smallskip
Clearly in this case
 the expected number of children at $y$ of a particle
living at $x$ is
\begin{equation}\label{eq:meanparticular}
 m_{xy}=p(x,y) \bar \rho_x.
\end{equation}

\subsection{Continuous-time Branching Random Walks}
\label{subsec:continuous}

Continuous-time BRWs have been studied extensively by many authors;
in this section we make use of a natural correspondence between continuous-time BRWs and
discrete-time BRWs which preserves both local and global behaviors.

In continuous time each particle has an exponentially distributed
random lifetime with parameter 1. The breeding mechanisms can
be regulated by means of a nonnegative matrix $K=(k_{xy})_{x,y \in X}$ in such a way that
for each particle at $x$,
there is a clock with $Exp(\lambda k_{xy})$-distributed intervals (where $\lambda>0$),
each time the clock
rings the particle breeds in $y$. We say that the BRW has a death rate 1 and a
reproduction rate $\lambda k_{xy}$ from $x$ to $y$.
We observe (see Remark~\ref{rem:deathrate}) that the assumption
of a nonconstant death rate does not represent a
significative generalization.
We denote by  $(X,K)$ a family of continuous-time BRWs (depending on the parameter $\lambda>0$),
while we use the notation $(X,\mu)$ for
a discrete-time BRW.

Equivalently, one can associate to each particle at $x$
a clock with $Exp(\lambda k(x))$-distributed intervals ($k(x)=\sum_y k_{xy}$):
each time the clock rings the particle breeds and the offspring is placed
at random according to a stochastic matrix $P$
(where $p(x,y)=k_{xy}/k(x)$).

To a continuous-time BRW
one can associate a discrete-time counterpart;
here is the construction.
The initial particles represent the generation $0$ of the discrete-time BRW;
the generation $n+1$ (for all $n \ge 0$) is obtained by
considering the children of all particles of generation $n$
(along with their positions).
Clearly the progenies of the original continuous-time BRW and of its discrete-time counterpart
are both finite (or both infinite) at the same time.
Moreover, almost surely, the two processes have the same local and global behavior.
In this sense the theory of continuous-time BRWs, as long as we
are interested in the probability of survival (local, strong local and global),
is a particular case of the theory of discrete-time BRWs.

Elementary calculations show that each particle living at $x$, before dying,
has a random number of offsprings given by equation~\eqref{eq:particular1} where
\begin{equation}\label{eq:counterpart}
\rho_x(i)=\frac{1}{1+\lambda k(x)} \left ( \frac{\lambda k(x)}{1+\lambda k(x)} \right )^i, \qquad
p(x,y)=\frac{k_{xy}}{k(x)},
\end{equation}
and this is the law of the discrete-time counterpart; note that the discrete-time
counterpart of a continuous-time BRW is a BRW with independent diffusion.
Note that $\rho_x$ depends only on $\lambda k(x)$.
Using equation~\eqref{eq:meanparticular},
it is straightforward to show that $m_{xy}=\lambda k_{xy}$ and $\bar \rho_x=k(x)$.
Note that,
for a continuous-time BRW the first-moment matrix $M$ equals $\lambda K$.
From equation~\eqref{eq:counterpart} we have that, for any $\lambda>0$, the discrete-time
counterpart satisfies Assumption~\ref{assump:1}.


\begin{rem}\label{rem:deathrate}
The same construction applies to continuous-time BRWs with a death rate $d(x)>0$ dependent on $x \in X$. In this
case the discrete-time counterpart satisfies equation~\eqref{eq:particular1} where
\[ 
\rho_x(i)=\frac{d(x)}{d(x)+\lambda k(x)} \left ( \frac{\lambda k(x)}{d(x)+\lambda k(x)} \right )^i, \qquad
p(x,y)=\frac{k_{xy}}{k(x)}.
\] 
Hence, from the point of view of local and global survival, this process is equivalent to a continuous-time BRW with death rate
$1$ and reproduction rate $\lambda k_{xy}/d(x)$ from $x$ to $y$.
\end{rem}

All the definitions given in the previous section extend to the continuous-time case:
a continuous-time BRW has a certain property if and only if its discrete-time counterpart
has it. In particular, we observe that, for a continuous-time BRW, the isoperimetric
constant defined in equation~\eqref{eq:amenable} equals to $\lambda \iota_{(X,K)}$
where $\iota_{(X,K)}$ is the constant when $\lambda=1$. Hence either the BRW
is nonamenable for all $\lambda >0$ or it is amenable for all $\lambda >0$.

\smallskip

Given $x_0 \in X$, two critical parameters are associated to the
continuous-time BRW: the \textit{global} (or \textit{weak})
\textit{survival critical parameter} $\lambda_w(x_0)$ and the  \textit{local} (or  \textit{strong})
 \textit{survival critical parameter} $\lambda_s(x_0)$.
They are defined as
\[ 
\begin{split}
\lambda_w(x_0)&:=\inf \Big \{\lambda>0:\,
\pr^{\delta_{x_0}}\Big (\sum_{w \in X} \eta_t(w)>0, \forall t\Big) >0 \Big \},\\
\lambda_s(x_0)&:=
\inf\{\lambda>0:\,
\pr^{\delta_{x_0}} \big(\limsup_{t \to \infty} \eta_t(x_0)>0 \big) >0
\},
  \end{split}
\] 
where $\mathbf{0}$ is the configuration with no particles at all
sites and $\pr^{\delta_{x_0}}$ is the law of the process which
starts with one individual in $x_0$. The process is called
\textit{globally supercritical}, \textit{critical} or \textit{subcritical}
if $\lambda>\lambda_w$, $\lambda=\lambda_w$ or $\lambda<\lambda_w$;
an analogous definition is given for the local behavior using $\lambda_s$ instead of $\lambda_w$.

If the graph $(X, E_{\mu})$
is connected (that is, the BRW is irreducible) then these values do not depend on the initial
configuration, provided that this configuration is finite (that
is, it has only a finite number of individuals), nor on the choice
of $x_0$.
If we have $(X,K)$ and $(Y,\bar K)$ such that $Y\subseteq X$
and $k_{xy}\ge \bar k_{xy}$ for all $x,y\in Y$ then for all $x\in Y$
we have $\lambda_s^X(x)\le \lambda_s^Y(x)$ and  $\lambda_w^X(x)\le \lambda_w^Y(x)$.
In particular we say that there exists a \textit{pure global survival phase} starting from $x$ if the
interval $(\lambda_w(x),\lambda_s(x))$ is not empty; clearly, if $\lambda  \in (\lambda_w(x),\lambda_s(x))$
then the BRW is in a pure global survival phase according to Definition~\ref{def:survival}.

Given a continuous-time BRW $(X,K)$ we define the following two
families of parameters
\begin{equation}\label{eq:geomparam}
K_s(x,y) := \limsup_{n \to \infty} \sqrt[n]{k_{xy}^{(n)}}, \quad
K_w(x) := \liminf_{n \to \infty} \sqrt[n]{\sum_{y \in X} k_{xy}^{(n)}}, \qquad \forall x,y \in X,
\end{equation}
introduced in \cite{cf:BZ, cf:BZ2} where they are called $M_s$ and $M_w$.
Note that supermultiplicative arguments imply that
$K_s(x,x) =\lim_{n} (k^{(d(x)n)}_{xx})^{1/d(x)n}$ where
$d(x):=\textrm{gcd}\{n >0 : m^{(n)}_{xx}>0\}$ is the period of $x \in X$
(see \cite[Definition 2.19]{cf:Woess09};
hence,
for all $x \in X$, we have that
$K_s(x,x) \le K_w(x)$.
If the BRW is irreducible then $K_s(x,y)$ and $K_w(x)$ do not depend on $x,y \in X$.

Two special cases are particularly interesting in the continuous-time setting:
\textit{site-breeding BRWs} and \textit{edge-breeding BRWs}.
We say that a BRW is \textit{site-breeding} if $k(x)$ does not depend on $x \in X$
(cfr.~Definition~\ref{def:locallyisomorphic}); on one hand the number of children of a particle
is independent of the site, on the other hand, clearly, the diffusion matrix $P=(p(x,y))_{x,y \in X}$
can be inhomogeneous.
We say that a BRW is \textit{edge-breeding} if $X$ has a multigraph structure
(see \cite[Section 2.1]{cf:BZ} for a formal definition)
and $k_{xy}$
is the number of edges from $x$ to $y$; in this case to each edge there corresponds a constant
reproduction rate $\lambda$.
Thus the total reproduction rate for a particle living at $x$ is $\lambda \cdot \mathrm{deg}(x)$
where $\mathrm{deg}(x)$ is the number of edges \textit{from} $x$. The diffusion matrix $P$ in this
case is the transition matrix of the simple random walk on the underlying multigraph.
Note that if the multigraph is \textit{regular} (i.e.~$\mathrm{deg}(x)$ does not depend on $x$) then
the edge-breeding BRW is site breeding.

\subsection{Other dynamics}
\label{subsec:otherdynamics}

We describe some slightly different dynamics which, in fact,
have a natural discrete-time counterpart which has the
same local and global behavior.

Multitype BRWs has been studied by some authors (see for instance \cite{cf:MachadoMenshikovPopov}).
In these processes there is a metapopulation which consists of individuals
carrying a certain characteristic chosen among a family $I$ of types.
A particle of type $i\in I$ living
at $x \in X$ generates a random number of children of any type which are randomly
placed in $X$. More precisely, this can be seen as a discrete-time BRW where
the space is $X \times I$ and the family of probability measures $\{\mu_{x,i}\}_{x \in X, i \in I}$
are defined on $S_{X \times I}$. Thus, a particle of type $i$ living
at $x$ at the end of its lifetime dies and generates $f(y,j)$ children of
type $j$ at $y$ (for all $y \in X$ and $j \in I$) with probability $\mu_{x,i}(f)$.
For the multitype BRW the global survival of the metapopulation (resp.~the local
survival of a fixed type) is equivalent to the global (resp.~local) survival
of the single-type BRW on $X \times I$. On the other hand, the global survival
of a fixed type $i_0$ is equivalent to the survival in the subset $X \times \{i_0\}$,
while local survival at $x$ of the metapopulation is equivalent to the
survival on $\{x\} \times I$.

Time-dependent BRWs can be described similarly, if we have a time-dependent
family of laws $\{\bar \mu_{x,n}\}_{x \in X, n \in \N}$ we can construct a BRW
on $X\times \N$ where, at the end of its lifetime, a particle
at $(x, n)$ dies and generates $f(y)$ children at $(y, n+1)$, for all $y \in X$, with probability
$\mu_{x,n}(f)$.

We could define a continuous-time BRW where a particle living at $x$
is endowed with a Poisson clock of parameter $\lambda k(x)$; when this
clock rings it picks a function $f \in S_X$ with probability $\mu_{x}(f)$
and reproduces accordingly. All particles reproduce a random number of times
during their exponentially distributed (with mean $1$) lifetime.
This process has a discrete-time counterpart which can be constructed as in
Section~\ref{subsec:continuous}. In this case the discrete-time counterpart
in general does not satisfy equation~\eqref{eq:particular1}.

Another definition of BRW, which is used by some authors (see for instance \cite{cf:KestenSidorav} or
\cite{cf:Zaehle}) is
the following. Each particle moves according to a random walk on $X$. After a random
number of steps it dies and is replaced by a random number of children, whose law may depend
on the final position. Again this process has a natural counterpart:
if in the original BRW a particle starts its life at $x$ and dies at $y$ giving
birth to $n$ children then in the discrete-time counterpart this particle does not move
and, after one unit of time, it generates $n$ children at $y$.   Note that
in terms of local and global survival these two processes are equivalent.

\subsection{Reproduction trails}\label{subsec:trails}

A fundamental tool which allows us to give an alternative construction of the BRW is
the reproduction trail (see \cite{cf:PemStac1}).
We fix an injective map $\phi:X\times X \times \Z \times \N \rightarrow \N$.
Let the family $\{ f_{i,n,x}\}_{i \in \Z,n \ge 0 , x \in X}$ be as in Section~\ref{subsec:discrete} and let
$\eta_0$ be the initial value. For any fixed realization of the process we call \textit{reproduction trail}
to $(x,n) \in X \times \N$ a sequence
\begin{equation}\label{eq:trail}
(x_0,i_0,1),(x_1,i_1,j_1), \ldots ,(x_{n},i_{n}, j_{n})
\end{equation}
such that $x=x_n$, $-\eta_0(x_0) \le i_0 <0$, $0 < j_l \le f_{i_{l-1}, l-1, x_{l-1}}(x_l)$
and $\phi(x_{l-1}, x_l, i_{l-1}, j_l)=i_l$, where $0 <l \le n$.
The interpretation  is the following:
$i_n$ is the identification number of the particle, which
lives at $x_n$ at time $n$ and is the $j_n$-th offspring
of its parent. The sequence $\{x_0, x_1, \ldots, x_n\}$ is
the path induced by the trail (sometimes, we say
that the trail is based on this path).
Given any element $(x_l, i_l, j_l)$ of the trail~\eqref{eq:trail}, we say that the particle identified
by $i_n$ is a descendant of generation $n-l$ of the particle identified by $i_l$ and the trail
joining them is $(x_l,i_l,j_l), \ldots ,(x_{n},i_{n}, j_{n})$. We also say that the trail
of the particle $i_n$ is a prolongation of the trail of the particle $i_l$.

Roughly speaking the trail represents the past history
of each single particle back to its original ancestor, that is, the one living at time $0$; we note that from the couple
$(n, i_n)$, since the map $\phi$ is injective, we can trace back the entire genealogy of the particle.
The random variable $\eta_n(x)$ can be alternatively defined as the number of reproduction trails to
$(x,n)$.
This construction does not coincide with the one
induced by the equation~\eqref{eq:evolBRW} but the resulting processes have the same laws.

\subsection{Generating functions}\label{subsec:genfun}

Later on we will need some generating functions, both
1-dimensional and
infinite dimensional.
Define $T^n_x:=\sum_{y \in X}m^{(n)}_{xy}$ and
$\varphi^{(n)}_{xy}:=\sum_{x_1,\ldots,x_{n-1} \in X \setminus\{y\}} m_{x x_1} m_{x_1 x_2} \cdots m_{x_{n-1} y}$
(by definition $\varphi^{(0)}_{xy}:=0$ for all $x,y \in X$).
$T^n_x$ is the expected number of particles alive at time $n$ when the initial state is a single
particle at $x$.
Roughly speaking, $\varphi^{(n)}_{xy}$ is
the expected number of particles alive at $y$ at time $n$
when the initial state is just one particle at $x$ and the
process behaves like a BRW except that every particle reaching
$y$ at any time $i <n$ is immediately killed (before breeding).
In other words $\varphi^{(n)}_{xy}$ is the expected number of particles alive at $y$ at time $n$
whose trail did not hit any $(y,i,k)$ with $k <n$.

Let us consider the following family of 1-dimensional generating functions
(depending on $x,y \in X$), where $\lambda \in \C$:
\[
\begin{split}
\Gamma(x,y|\lambda)&:=\sum_{n =0}^\infty m^{(n)}_{xy} \lambda^n,
\qquad \Phi(x,y|\lambda):=\sum_{n =1}^\infty \varphi_{xy}^{(n)} \lambda^n.
\end{split}
\]
To compare with random walk theory, $\Gamma$ is the analog of the
\textit{Green function} (cfr.~\cite[Section 1.C]{cf:Woess09}) and
$\Phi$ is the analog of the generating function of
the first-return probabilities (cfr.~the function $U$ of \cite[Section 1.C]{cf:Woess09}).
%
It is easy to prove that $\Gamma(x,x|\lambda)= \sum_{i \in \N} \Phi(x,x|\lambda)^i$ for all $\lambda>0$, hence
\[ 
\Gamma(x,x|\lambda)= \frac{1}{1-\Phi(x,x|\lambda)},
\qquad \forall \lambda \in \C: |\lambda|< 
M_s(x,x)^{-1},
\] 
and we have that $M_s(x,x)
^{-1}=\max\{ \lambda 
\in {\mathbb R}:\Phi(x,x|\lambda)\leq 1\}$
for all $x \in X$. 
In particular $\Phi(x,x|1) \le 1$ if and only if
$
M_s(x,x) \le 1$.
The interpretation of $\Gamma(x,y|1)$ is the expected value of the total number of descendants at $y$ of
a common ancestor living at $x$. On the other hand, $\Phi(x,y|1)$ is the expected number of descendants at $y$, of
a common ancestor living at $x$, whose trails start from $x$ and arrive at $y$ for the first time.

The classical approach to branching processes (see for instance \cite{cf:Harris63}) makes use of the one-dimensional
generating function of the offspring distribution $\rho$: $\widetilde G(z)=\sum_{n \in \N} \rho(n) z^n$, whose
minimal fixed point is the probability of extinction (see also Example~\ref{ex:branchingprocess}). Inspired by
this approach, we associate a generating function $G:[0,1]^X \to [0,1]^X$
to the family $\{\mu_x\}_{x \in X}$
which can be considered as an infinite dimensional power series (see also \cite[Section 3]{cf:BZ2}). More precisely,
for all $z \in [0,1]^X$ the function $G(z) \in [0,1]^X$ is defined as the following weighted sum of (finite) products
\begin{equation}
\label{eq:genfun}
G(z|x):= \sum_{f \in S_X} \mu_x(f) \prod_{y \in X} z(y)^{f(y)}.
\end{equation}
Note that $G$ is continuous with respect to the \textit{pointwise convergence topology} of $[0,1]^X$  and nondecreasing
with respect to the usual partial order of $[0,1]^X$ (see \cite[Sections 2 and 3]{cf:BZ2} for further details).
Moreover, $G$ represents the 1-step reproductions; we denote by $G^{(n)}$ the generating function
associated to the $n$-step reproductions, which is inductively defined as $G^{(n+1)}(z)=G^{(n)}(G(z))$.

The generating function $G$ can be explicitly computed in some cases: for instance,
for a BRW with independent diffusion (i.e.~if equation~\eqref{eq:particular1} holds).
Indeed in this case it is straightforward
to show that
$G(z|x)=F_x(Pz(x))$
where $F_x(y)=\sum_{n=0}^\infty \rho_x(n)y^n$ is the generating function of the number of children
and $Pz(x)=\sum_{y \in X} p(x,y)z(y)$ is the transition operator of the corresponding random walk.
In particular if $\rho_x(n)=\frac{1}{1+\bar \rho_x} (\frac{\bar \rho_x}{1+\bar \rho_x} )^n$
(for instance if we are dealing with the discrete-time counterpart of a continuous-time BRW, see
equation~\eqref{eq:counterpart}), we have
$G(z|x)=\frac{1}{1+\bar \rho_x(1-Pz(x))}$ (see \cite[Section 3.1]{cf:BZ2}),
that is,
\begin{equation}\label{eq:Gcontinuous}
G(z)= \frac{\mathbf{1}}{\mathbf{1}+M(\mathbf{1}-z)}
\end{equation}
where $\mathbf{1}(x):=1$ for all $x \in X$, the ratio is to be intended as coordinatewise and
$Mv$ was defined in Section~\ref{subsec:discrete}; in this case
$m_{xy}$ is given by equation~\eqref{eq:meanparticular}, thus $M=\lambda K$.

The following proposition is a sort of \textit{maximum principle} for
the function $(z-\bar q)/(\mathbf{1}-\bar q)$ where $G(z) \ge z$
(see Section~\ref{subsec:nodeath}).

\begin{pro}\label{pro:maximumprinciple}
 Let  $z \in [0,1]^X$, $z \ge \bar q$ be a solution of the inequality $G(z) \ge z$.
If $\bar q < \mathbf{1}$ and we define
$\widehat z:=(z-\bar q)/(\mathbf{1}-\bar q)$ (by definition $\widehat z(x):=1$ for all
$x$ such that $\bar q(x)=1$) 
then for all $x \in X$ such that the set $\mathcal{N}_x=\{y:(x,y) \in E_\mu\}$
is not empty,
either $\widehat z(y) = \widehat z(x)$ for all $y \in \mathcal{N}_x$ or there exists
$y \in  \mathcal{N}_x$ such that $\widehat z(y)> \widehat z(x)$. In particular
if $\widehat z(x)=1$ then for all $y \in \mathcal{N}_x$ we have $\widehat z(y)=1$.
The same results hold if we take the set $\{y \in X : x \to y\}$ instead of $\mathcal{N}_x$.
\end{pro}

The proof, which makes use of some arguments of Section~\ref{subsec:nodeath},
can be found in Section~\ref{sec:proofs}.
We observe that in a finite, final irreducible class
(for instance if the BRW is irreducible and the set $X$ is finite) then
$\widehat z$ is constant if $z \ge \bar q$ is a solution of  $G(z) \ge z$.
Since the probabilities of extinction
$\{q(\cdot,A)\}_{A \subseteq X}$ (see Section~\ref{subsec:survivalprob} for the definition)
are solutions of $G(z)=z$,
Proposition~\ref{pro:maximumprinciple} applies with $z(\cdot)=q(\cdot,A)$ for all $A \subseteq X$.
In this case $\widehat z(x)$ can be interpreted as the probability of
local extinction in $A$ conditioned on global survival (starting from $x$).
Note that if $\mu_x(\mathbf{0})=0$ for 
all $x \in X$ (see Section~\ref{subsec:nodeath})
then $\bar q= \mathbf{0}$ and $\widehat z=z$.

As an application,
if we have an irreducible BRW on $\N$ where $m_{xy}>0$ implies $|x-y| \le 1$ we get the following behavior of
the extinction probabilities: for all $A \subseteq X$,
there exists $x_0 \in \N \cup \{+\infty\}$ such that $q(x,A)=q(0,A)$ for all $x \le x_0$ and
$q(x,A)<q(x+1,A)$ for all $x \ge x_0$.

Finally it is worth mentioning that for a BRW with independent
diffusion the following generating function is very
useful:
\[
 L(z|x):=z(x)\frac{Mz(x)}{\mathbf{1}+ \bar \rho_x}+\frac{\mathbf{1}}{\mathbf{1}+\bar \rho_x}.
\]
While $G$ is obtained by conditioning on the number
of total children of a particle at the end of its life, in continuous time $L$
can be obtained by conditioning on the first event (birth or death of
a particle).  By algebraic manipulation we have that,
if $G$ satisfies equation~\eqref{eq:Gcontinuous}, then
$z=G(z)$ if and only if $z=L(z)$.
In particular, using the same arguments as in Section~\ref{subsec:survivalprob},
we see that by means of $L$ one can compute iteratively the probability of extinction
before the $n$th event (either birth or death).

\section{Special processes}\label{sec:specialprocesses}

\subsection{$\mathcal F$-BRWs}\label{subsec:FBRWs}

Some results can be achieved if the BRW has some regularity; to this aim
we introduce the concept of $\mathcal F$-BRW (see also \cite[Definition 4.2]{cf:Z1}).

\begin{defn}
\label{def:locallyisomorphic}
We say that a BRW $(X, \mu)$ is locally isomorphic to a BRW $(Y,\nu)$ if there exists a
surjective map $g:X\to Y$ such that
\begin{equation}
\label{eq:fgraph}
\nu_{g(x)}(\cdot)=
\mu_x\left(\pi_g^{-1}(\cdot)\right)
\end{equation}
where $\pi_g:S_X \rightarrow S_Y$ is defined as $\pi_g(f)(y)=\sum_{z\in g^{-1}(y)}f(z)$ for all $f\in S_X$, $y \in Y$.
We say that $(X, \mu)$ is a $\mathcal F$-BRW if it is locally isomorphic to some
BRW $(Y,\nu)$ on a finite set $Y$.
\end{defn}
Clearly, if $(X,\mu)$ is locally isomorphic to $(Y, \nu)$ then
\begin{equation} \label{eq:Gfunctions}
G_X(z \circ g|x)=G_Y(z|g(x))
\end{equation}
for all $z \in [0,1]^Y$ and $x \in X$. Indeed $\pi_g$ is surjective and
\[
\begin{split}
 G_X(z \circ g|x)&=\sum_{f \in S_X} \mu_x(f) \prod_{w \in X} z(g(w))^{f(w)}
=  \sum_{h \in S_Y}  \sum_{f \in \pi_g^{-1}(h) } \mu_x(f) \prod_{w \in X} z(g(w))^{f(w)}\\
&=  \sum_{h \in S_Y}  \sum_{f \in \pi_g^{-1}(h) } \mu_x(f) \prod_{v \in Y} \prod_{w \in g^{-1}(v)} z(v)^{f(w)}\\
&=  \sum_{h \in S_Y}  \mu_x(\pi_g^{-1}(h)) \prod_{v \in Y} z(v)^{h(v)} \\
&= \sum_{h \in S_Y}  \nu_{g(x)}(h) \prod_{v  \in y} z(v)^{h(v)}=G_Y (z | g(x)).
\end{split}
\]
By induction on $n$
\[
G^{(n+1)}_X(z \circ g)=G_X(G_X^{(n)}(z \circ g))=G_X(G^{(n)}_Y(z)\circ g)=G_Y(G^{(n)}_Y(z))\circ g=
G^{(n+1)}_Y(z)\circ g,
\]
whence
\begin{equation}\label{eq:G-FBRWs}
 G^{(n)}_X(z \circ g)=G^{(n)}_Y(z)\circ g
\end{equation}
for all $n \in \N$.
%
We note that, since $\mu$ is uniquely determined by $G$,
 equation~\eqref{eq:Gfunctions} holds  if and only if $(X,\mu)$ is locally
isomorphic to $(Y,\nu)$ and $g$ is the map in Definition~\ref{def:locallyisomorphic}.
To see the ``only if'' part, define $\widehat \nu$ by using equation~\eqref{eq:fgraph} (substitute $\nu$ with $\widehat \nu$),
then equation~\eqref{eq:Gfunctions} holds with
$\widehat G$ instead of $G_Y$; thus $\widehat G=G_Y$ and this implies that equation~\eqref{eq:fgraph}
holds for $\nu$.

Using equation~\eqref{eq:G-FBRWs} and the fact that $\bar q= \lim_{n \to \infty} G^{(n)}(\mathbf 0)$
(see equation~\eqref{eq:extprobab} with $A=X$), it is possible to prove that
there is global survival for $(X,\mu)$ starting from $x$ if and only if
there is global survival for $(Y,\nu)$ starting from $g(x)$  (see \cite[Theorem 4.3]{cf:Z1}).

In particular note that the total offspring generating functions $G(t \mathbf 1|x)$ satisfy
$G_X(t\mathbf 1_X|x)= G_X(t\mathbf 1_Y \circ g|x)=G_Y(t\mathbf 1_Y | g(x))$, hence the offspring distribution
of a particle at $x$ behaving according to $(X,\mu)$ is the same of the offspring distribution of a particle
at $g(x)$ behaving according to $(Y,\nu)$.

If $(X,\mu)$ is locally isomorphic to a BRW $(Y, \nu)$ where $Y$ is a singleton (equivalently,
if the law of the total number of children $\rho_x$ does
not depend on $x \in X$) then we say that the BRW is \textit{locally isomorphic to a branching process}
(this case is discussed in details in Section~\ref{subsec:locisomBP}).

It is easy to prove that $\sum_{w \in X} m^X_{xw}v(x)= \frac{\diff}{\diff t} G_X(\mathbf 1 -(1-t)v|x)|_{t=1}$ for all $x \in X$ and
$v \in [0,1]^X$, hence, using equation~\eqref{eq:G-FBRWs}, we have
  $\sum_{w \in X} m^X_{xw}z(g(x))=
\sum_{y \in Y} m^Y_{g(x)y}z(y)$,
for all $x \in X$ and $z \in [0,1]^Y$. This, in turn, implies that
$\sum_{w \in X} m^{X,(n)}_{xw}z(g(w))=\sum_{y \in Y} m^{Y,(n)}_{g(x)y}z(y)$ for all $n \in \N$.
In particular, when $n=1$ and $z = \mathbf 1$ we have $\bar \rho^X_{x}=\bar \rho^Y_{g(x)}$.

In continuous time (see \cite{cf:BZ2})
we say that $(X,K)$ is \textit{locally isomorphic} to $(Y,\widetilde K)$ if and only if
there exists a surjective map $g:X \to Y$ such that
$\sum_{z \in g^{-1}(y)} k_{xz}=\widetilde k_{g(x)y}$
for all
$x \in X$ and $y \in Y$. We observe that
$(X,K)$ is locally isomorphic to $(Y,\widetilde K)$
if and only if the discrete-time counterparts satisfy
Definition~\ref{def:locallyisomorphic}; this can easily be checked
by proving that the equation~\eqref{eq:Gfunctions} holds when
$G_X$ and $G_Y$ satisfy equation~\eqref{eq:Gcontinuous}.
Note that a continuous-time BRW is site-breeding if and only if it
is locally isomorphic to a branching process.
On the other hand a continuous-time, edge-breeding BRW is locally isomorphic
to another edge-breeding BRW if and only if the underlying multigraphs
satisfy \cite[Definition 3.1]{cf:BZ}.

\begin{defn}\label{def:quasitransitive}
Let $\gamma:X \to X$ be an injective map. 
We say that $\mu=\{\mu_x\}_{x \in X}$ is $\gamma$-invariant if
for all $x,y \in X$ and $f \in S_X$ we have
$\mu_x(f)=\mu_{\gamma(x)}(f \circ \gamma^{-1})$.

Moreover $(X,\mu)$ is \textit{quasi transitive} if and only if there exists
a finite subset $X_0 \subseteq X$ such that for all $x \in X$ there exists a bijective
map $\gamma:X\to X$ and $x_0 \in X_0$ satisfying $\gamma(x_0)=x$ and $\mu$ is $\gamma$-invariant.
\end{defn}

The previous definition generalizes the usual one (which applies to graphs) in the following way:
a discrete-time counterpart of an edge-breeding continuous-time BRW is quasi transitive if and only
if the underlying graph is quasi transitive (that is, the action of the group of automorphisms
has only finitely many orbits).

We note that every quasi-transitive BRW is an $\cF$-BRW. Indeed, consider
the equivalence relation $x \sim y$ if and only if there exists a bijective map
$\gamma:X \to X$ such that $\gamma(x)=y$. Clearly if $Y:=X/_\sim$ then $\# Y \le \# X_0$.
Let $g$ be the usual projection from $X$ onto $X/_\sim$ and $\nu_{g(x)}(\cdot):=\mu_x(\pi_g^{-1}(\cdot))$.
We have to show that the last definition is well-posed. Note that if $\mu$ is $\gamma$ invariant
then $g=g \circ \gamma$ which implies $\pi_g(f)=\pi_{g \circ \gamma}(f)=\pi_g(f \circ \gamma^{-1})$; indeed
\[
 \pi_g(f)(y)=\sum_{z \in g^{-1}(y)} f(z)= \sum_{z \in (g \circ \gamma)^{-1}(y)} f(z)=
\sum_{z \in g^{-1}(y)} f(\gamma^{-1}(z))
\]
whence
for all $h \in S_Y$
\[
\begin{split}
 \mu_x(\pi_g^{-1}(h))&=\mu_{\gamma(x)}(\{f \circ \gamma^{-1}: \pi_g(f)=h\})\\
&=
\mu_{y}(\{f \circ \gamma^{-1}: \pi_g(f \circ \gamma^{-1})=h\})=
\mu_y(\pi_g^{-1}(h)).
\end{split}
\]
The class of $\mathcal{F}$-BRWs is strictly larger than the class of quasi-transitive BRWs. An example is given by
the BRW described in Example~\ref{ex:nonstronglocalFBRW}. Another one is the following
(a further example can be found in \cite[Example 3.2]{cf:BZ}).

\begin{exmp}\textbf{\cite[Example 3.1]{cf:BZ}}\label{rem:pureFBRW}
Take a square and attach to every vertex a branch of a homogeneous tree of degree $3$, obtaining a regular graph (of degree $3$)
which is not quasi transitive (see Figure~\ref{fig:square}).
If we attach now to each vertex a new edge with a new endpoint we obtain a
non-oriented, nonamenable
$\cF$-graph $(X,E(X))$ (see \cite[Definition 3.1]{cf:BZ}) which is neither regular nor quasi transitive. It is easily seen
(see \cite[Lemma 3.2]{cf:BZ}) to be locally isomorphic to a multigraph with
adjacency matrix
\[
N=
\begin{pmatrix}
3 & 1\\
1 & 0 \\
\end{pmatrix}.
\]
The corresponding continuous-time, edge-breeding BRW is an $\mathcal{F}$-BRW which is defined on a
(nonamenable) multigraph which is neither regular nor quasi transitive (thus, the process is not quasi transitive).
\end{exmp}

\begin{figure}
    \begin{center}
    \subfigure[\tiny The regular graph of Example~\ref{rem:pureFBRW}.]{\includegraphics[height=5cm]{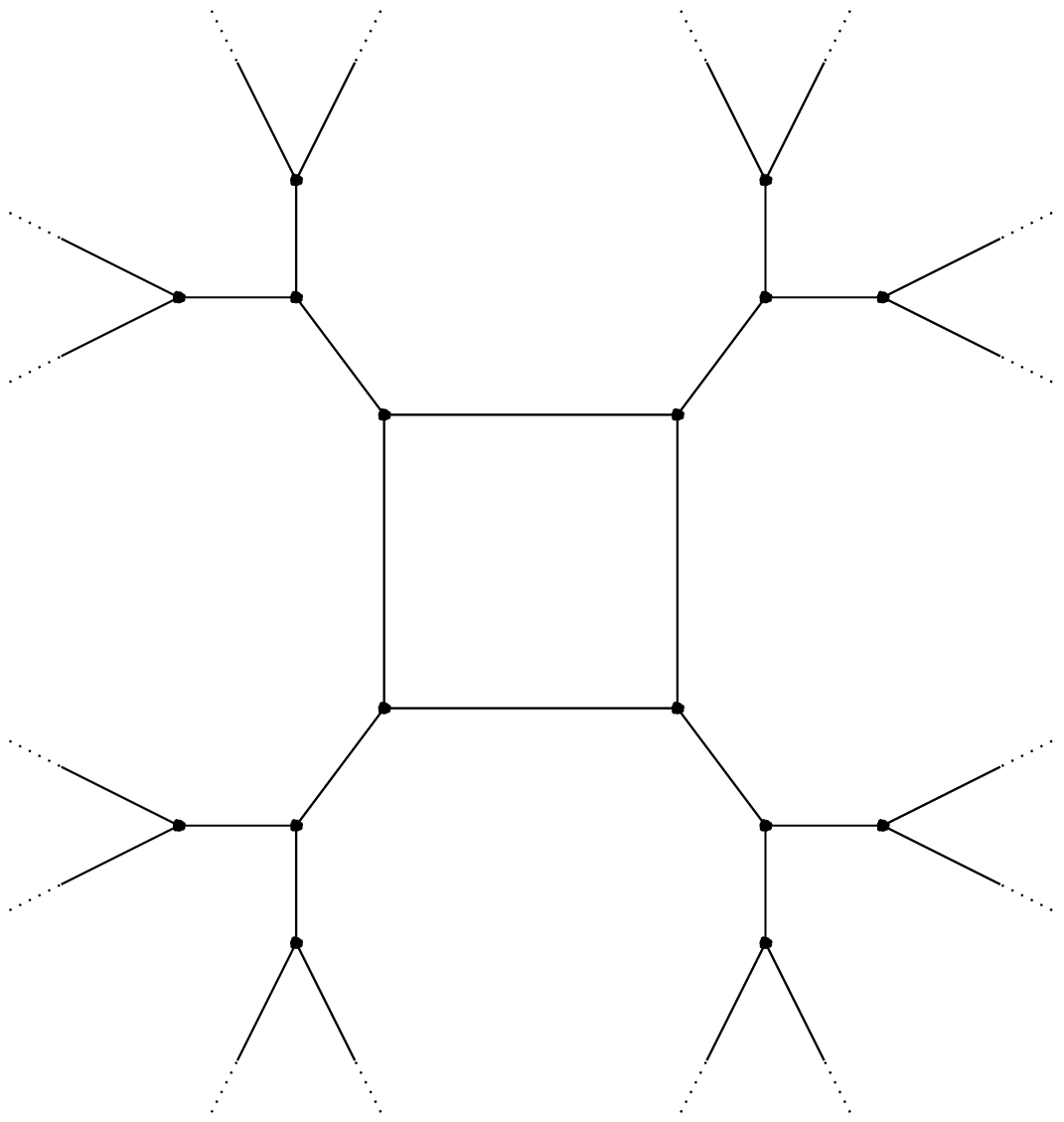}
    \label{fig:square}}
   \hspace{0.5cm}
    \subfigure[\tiny The graph $X$ of Example~\ref{exm:amenable}.]{\includegraphics[height=3.5cm]{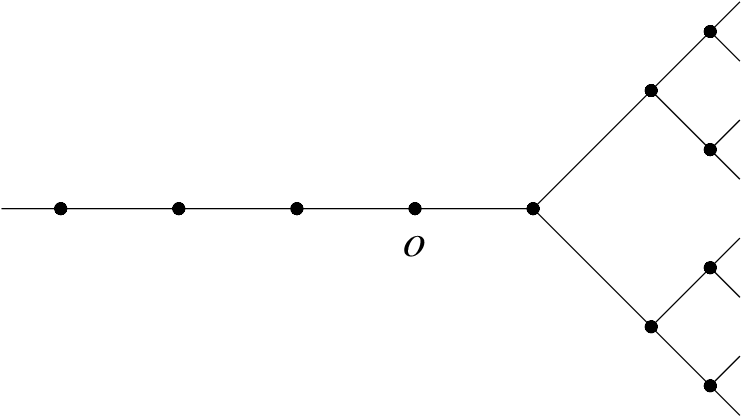}
    \label{fig:amenable}}
    \end{center}
\caption{}
\end{figure}

\subsection{BRWs with and without death: a comparison}
\label{subsec:nodeath}

Some authors (see for instance \cite{cf:CMP98, cf:GMPV09, cf:MachadoMenshikovPopov, cf:M08, cf:Muller08-2} and some
results of \cite{cf:MenshikovVolkov}) have extensively studied the special case
where $\mu_x(\mathbf 0)=0$ for all $x \in X$
or, that is the same, $\rho_x(0)=0$ for all $x \in X$.  We call this kind of process a
\textit{BRW with no death}; to be honest, in the usual interpretation each particle still dies but
it has at least one descendant almost surely.
On the other hand one can think that particles never die (in this case, $\rho_x(n)$ must be interpreted
as the probability of having $n-1$ children). We stick with the first interpretation.

Being the global survival trivial in this case,
it is interesting to explore the local behavior of the BRW. This situation is
the closest one to the random walk theory.
In particular,
the process can become extinct locally at $x$ with probability one (some authors call the
process \textit{transient} in this case) or survive locally with positive probability.
While in the general situation the colony cannot always survive with probability one (if
$\rho_x( 0)>0$ the process starting from $x$ might die at the first step),
a BRW with no death can survive locally at $x$ either with probability one
(strong local survival or \textit{recurrent BRW}) or with a strictly positive probability
different from one (\textit{weakly recurrent BRW}).

In this section we want to discuss how we may interpret results on
BRWs with no death in the general case. The first idea, which was introduced by Harris
in the case of a branching process (see for instance \cite{cf:Harris48}
or \cite[Chapter I.12]{cf:AthNey}), is to condition on global survival.
It is clear that a generic BRW,
such that $\bar q(x)<1$ for all $x \in X$,  
conditioned on global survival is not a BRW with no death
(it is not even a BRW). Indeed, for the conditioned process,
the behaviors of different particles of the same generation
are not independent: the probability of dying without breeding is strictly positive
but the probability that all particles in the same generation die out
without breeding is $0$. Nevertheless it is possible to associate to a generic irreducible BRW, with a fixed
starting configuration,
a BRW with no death.
Given a generic irreducible BRW $\{\eta_n\}_{n \in \N}$ consider the event $\Omega_\infty=
\{\sum_{x \in X} \eta_n(x)>0, \, \forall n \in \N\}$ and define
the process $\{\widehat \eta_n\}_{n \in \N}$ as follows:
$\widehat \eta_n(x, \omega)$ equals the number of particles in $\eta_n(x, \omega)$ with an infinite line of descent
when $\omega \in \Omega_\infty$ and it equals  $0$ when $\omega \not \in \Omega_\infty$.
It can be shown that this process, restricted to $\Omega_\infty$ is a BRW (that we call again
$\{\widehat \eta_n\}_{n \in \N}$) and its generating function is
\begin{equation}\label{eq:G-one}
\widehat G(z|x) = \frac{G(v(z)|x)-\bar q(x)}{1-\bar q(x)}
\end{equation}
where $\bar q(x)$ is the probability of global extinction starting from $x \in X$
(see Sections~\ref{subsec:discrete} and
\ref{subsec:survivalprob}),
$G$ is the generating function of the original BRW and $v:[0,1]^X \rightarrow [0,1]^X$ is
defined as $v(z|x):=\bar q(x)+z(x)(1-\bar q(x))$.
In a more compact way equation~\eqref{eq:G-one} can be written as $\widehat G=T^{-1}_{\bar q} \circ G \circ T_{\bar q}$ where
$T_w:[0,1]^X \to \{z \in [0,1]^X: w \le z\}$ is defined as $T_wz(x):=z(x)(1-w(x))+w(x)$; note that
$T_w$ is nondecreasing and, if $w(x)  <{1}$ for all $x \in X$, bijective.
More explicitly
\[
\begin{split}
 \widehat G(z|x)&=\sum_{f \in S_X} \widehat \mu_x(f)
\prod_{y \in X} z(y)^{f(y)},\\
\textrm{where }
\widehat \mu_x(f)&=
\begin{cases}
 \displaystyle \frac{
\sum_{g \in S_X:g \ge f} \mu_x(g) \prod_{y \in X} \binom{g(y)}{f(y)} \bar q(y)^{g(y)-f(y)}(1-\bar q(y))^{f(y)}}{1-\bar q(x)} &
\ \textrm{if }f \not = \mathbf{0}\\
\displaystyle 0 &\ \textrm{if }  f = \mathbf{0}.
\end{cases}
\end{split}
\]
Indeed $\{\widehat \eta_n\}_{n \in \N}$ is a Markov process and if we consider the offsprings
of a particle living at $x$ at some time $n$, then its reproduction law conditioned
on the global survival of the process and on having an infinite line of descent is equal to
the reproduction law of an initial particle at $y$ conditioned on $\mathcal{A}_x$, that is
\[
 \pr(\widehat \eta_1=f|\mathcal{A}_x)=
\begin{cases}
 \displaystyle\frac{\pr(\widehat \eta_1=f)}{1-\bar q(x)}
& f > \mathbf{0}\\
0 & f=\mathbf{0}
\end{cases}
\]
where $\mathcal{A}_x$ is the event that there is global survival (i.e.~the first particle has an infinite line of descent) starting
from one particle at $x \in X$ (clearly $\pr(\mathcal{A}_x)=1-\bar q(x)$).
Hence, if $f > \mathbf{0}$,
\[
 \pr(\widehat \eta_1=f|\mathcal{A}_x)=
\frac{1}{1-\bar q(x)} \sum_{g \in S_X: g \ge f} \mu_x(g) \prod_{y \in Y} \binom{g(y)}{f(y)} (1-\bar q(y))^{f(y)}
\bar q(y)^{g(y)-f(y)}.
\]
This implies that
\[
\begin{split}
\widehat G(z|x)&=
 \sum_{f \in S_X} \pr(\widehat \eta=f|\mathcal{A}_x) \prod_{y \in Y} z(y)^{f(y)} \\
&= \sum_{g \in S_X: g > \mathbf{0}}\frac{\mu_x(g)}{1-\bar q(x)} \sum_{f \in S_X: \mathbf{0} <f \le g}
\prod_{y \in Y} \binom{g(y)}{f(y)} (z(y)(1-\bar q(y)))^{f(y)}
\bar q(y)^{g(y)-f(y)}\\
&=\sum_{g \in S_X: g > \mathbf{0}}\frac{\mu_x(g)}{1-\bar q(x)} \left [\prod_{y \in Y}
(T_{\bar q}z(y))^{g(y)}- \prod_{y \in Y} \bar q(y)^{g(y)}
\right ]\\
&=\frac{1}{1-\bar q(x)} \left [G(T_{\bar q}z|x)-G(\bar q|x)\right ]
=\frac{1}{1-\bar q(x)} \left [G(T_{\bar q}z|x)-\bar q(x)\right ].\\
\end{split}
\]

It can be shown, following \cite{cf:AthNey}, that many results about survival
are true for $\{\eta_n\}_{n \in \N}$ if and only if
they are true for $\{\widehat \eta_n\}_{n \in \N}$. Proving this equivalence
in details goes beyond the purpose of this chapter.
Nevertheless we observe that, if $\bar q < \mathbf{1}$ then $T_{\bar q}$ is a bijective map from the set
of fixed points of $\widehat G$ to the set of fixed points of $G$. Moreover,
since $\{\widehat \eta_n\}_{n \in \N}$ is obtained
by $\{\eta_n\}_{n \in \N}$ by removing all the particles with finite progeny, which
are clearly irrelevant in view of the survival due to the fact that $\bar q(x)<1$ for all $x \in X$,
%
%
we have immediately that the probability of local survival of
$\{\widehat \eta_n\}_{n \in \N}$ in $A$ (for all $A \subseteq X$),
starting from $x$ is equal to the same probability for $\{\eta_n\}_{n \in \N}$,
that is, $q(x,A)$. In particular the probability of local survival at
$A$ starting from $x$ conditioned on $\mathcal{A}_x$
is $1-(T_{\bar q}^{-1} q(\cdot,A))(x)= (1-q(x,A))/(1-\bar q(x))$.

We call the process $\{\widehat \eta_n\}_{n \in \N}$ conditioned on $\mathcal{A}_x$ the \textit{no-death BRW associated to}
$\{\eta_n\}_{n \in \N}$ starting from $x \in X$.

\section{Survival}
\label{sec:survival}

\subsection{Probabilities of extinction}
\label{subsec:survivalprob}

Define $q_n(x,A)$ as the probability of extinction before
generation $n+1$ in $A$ starting with one particle at $x$, namely
$q_n(x,A)=\Prob(\eta_k(x)=0,\, \forall k\ge n+1,\,\forall x\in A)$. It is
clear that $\{q_n(x,A)\}_{n \in \N}$ is a nondecreasing sequence satisfying
\begin{equation}\label{eq:extprobab}
 \begin{cases}
 q_{n}(\cdot,A)=G(q_{n-1}(\cdot, A)),& \quad \forall n \ge 1\\
 q_0(x,A)=0, &\quad \forall x \in A,\\
{q}_0(x,A)=G(\mathbf{q}_0,x) &\quad  \forall x \not \in A,\\
\end{cases}
\end{equation}
hence there is a limit $q(x,A)=\lim_{n \to \infty} q_n(x, A) \in
[0,1]^X$ which is the probability of local extinction in $A$
starting with one particle at $x$. Note that equation~\eqref{eq:extprobab}
defines completely the sequence $\{q_n(\cdot,A)\}_{n \in \N}$ only when
$A=X$ (otherwise one needs the values $q_0(x,A)$ for $x \not \in A$).
For details on the last equality in equation~\ref{eq:extprobab} see Remark~\ref{rem:q0}. 
Since $G$ is continuous we have that $q(\cdot,A)=G(q(\cdot,
A))$, hence these probabilities are
fixed points of $G$
(and Proposition~\ref{pro:maximumprinciple} applies). Note that $q(\cdot, \emptyset)= \mathbf 1$,
$q(\cdot, X)=\bar q(\cdot)$ and $q(\cdot, \{y\})=q(\cdot, y)$ (see Definition~\ref{def:survival}).
%
It can be shown (see \cite[Corollary 2.2]{cf:BZ2})
that $\bar q$ is the smallest fixed point of $G(z)$ in $[0,1]^X$,
since it is $\bar q=\lim_{n \to \infty} G^{(n)}(\mathbf{0})$.
Using the same arguments, one can prove that $\bar q$ is the smallest
fixed point of $G^{(m)}$ for all $m \in \N$.

Note that $A \subseteq B$ implies $q(\cdot,A)
\ge q(\cdot,B)$. In particular, $q(\cdot,y)
\ge \bar q$ for all $y \in X$. Since for all finite $A\subseteq X$ we have $q(x,A) \ge 1-\sum_{y \in A} (1-q(x,y))$
then, for any given finite $A \subseteq X$, $q(x,A)=1$ if and only if $q(x,y)=1$ for all $y \in A$.

Moreover, given a BRW $(X,\mu)$ and $Y\subseteq X$,
consider $(Y,\nu)$ obtained by killing all particles outside $Y$; in this case
$q^X(x,A)\le q^Y(x,A)$ for all $x\in Y$, $A\subseteq Y$.

If $x \to x^\prime$ and $A \subseteq X$ then $q(x^\prime,A)<1$ implies
$q(x,A)<1$; as a consequence,
if $x \rightleftharpoons x^\prime$ and $y \rightleftharpoons y^\prime$ then
$q(x,A)<1$ if and only if $q(x^\prime,A)<1$ and
$q(x,y)=q(x,y^\prime)$.
In the irreducible case $q(x,A)<1$  for some $x \in X$ if and only if
$q(w,A)<1$  for all $w \in X$; in particular
$\bar q(x)<1$ for some $x \in X$ if and only if $\bar q(w)<1$ for all $w \in X$.
Moreover $q(x,A)<1$ for some $x \in X$ and a finite $A \subseteq X$ if and only if
$q(w,B)<1$ for all $w \in X$ and all finite $B \subseteq X$. Indeed, in the irreducible case,
one can prove that $q(x,A)=q(x,x)$ for all $x \in X$ and every finite $A \subseteq X$:
since surviving in a finite subset $A$ is equivalent to surviving in at least
one of its points, then it is enough to prove it in the case $A:=\{y\}$ for $y \in X$;
in this case the conclusion follows from a Borel-Cantelli argument.

In the irreducible case, if $\rho_x(0)>0$ for all $x \in X$, we have that $\bar q(x)=q(x,A)$ for
some $x \in X$ and a finite subset $A \subseteq X$ if and only if $\bar q(y)=q(y,B)$ for
all $y \in X$ and all finite subsets $B \subseteq X$.
Indeed, if $\bar q(x)=1$ then $q(y,B)=1$ for all $y \in X$ and $B \subseteq X$ and there is nothing to prove.
Suppose that $\bar q(x)=q(x,A)<1$ and $\bar q(y)<q(y,B)$ for some $x,y \in X$ and $A,B \subseteq X$
finite. By irreducibility $q(x,A)=q(x,x)=q(x,B)$ hence we can assume that $A=B$.
We know that there is a positive probability that the process, starting from $x$
has at least one descendant at $y$. There is also a positive probability that all the particles (except one at $y$)
die out and the progeny of the surviving particle survives globally but not locally at $A$. Thus,
there is a positive probability, starting from $x$, of surviving globally but not locally at $A$ and this is a contradiction.
Observe that if we drop the assumption $\rho_x(0)>0$ for all $x \in X$, we might actually
have $\bar q(x)=q(x,A)<1$ and $\bar q(y)<q(y,A)$ for some $x,y \in X$ and a finite $A \subseteq X$
(see Example~\ref{exm:irreduciblenotstrongeverywhere}).

\begin{rem}\label{rem:strongconditioned}
We observe that the following assertions are equivalent for every nonempty subset $A \subseteq X$.
\begin{enumerate}[(1)]
\item  $q(x, A) = \bar q(x)$, for all $x \in X$;
\item $q_0(x,A) \le \bar q(x)$, for all $x \in X$;
\item  the probability of visiting $A$ at least once starting from $x$ is larger than the probability of
global survival starting from $x$, for all $x \in X$:
\item for all $x \in X$, either $\bar q(x)=1$ or
the probability of visiting $A$ at least once starting from $x$ conditioned on global survival starting from $x$ is $1$
(strong local survival in $A$ starting from $x$);
\item  for all $x \in X$, either $\bar q(x)=1$ or
the probability of local survival in $A$ starting from $x$ conditioned on global survival starting from $x$ is $1$.
\end{enumerate}

Indeed,
since $\{q_n(\cdot,A)\}_{n \in\N}$ is non decreasing, $q_{n}(\cdot,A)=G(q_{n-1}(\cdot, A))$ and
$\bar q$ is the smallest fixed point of $G$, we have immediately that
\begin{equation}\label{eq:strongconditioned}
q(\cdot, A) = \bar q(\cdot) \Longleftrightarrow q_0(\cdot,A) \le \bar q(\cdot),
\end{equation}
that is, (1)$\Longleftrightarrow$(2).
Moreover
the event ``local survival in $A$ starting from $x$'' implies both ``global survival starting from $x$''
and ``visiting $A$ at least once starting from $x$'', hence
$q(x, A) = \bar q(x)$ if and only if
the probability of visiting $A$ infinitely many times starting from $x$ conditioned on global survival
is $1$ and (1)$\Longleftrightarrow$(5)$\Longrightarrow$(4). Trivially (2)$\Longleftrightarrow$(3)
and (4)$\Longrightarrow$(3). This proves the equivalence.

Hence if there exists $x \in X$ such that $q(x,A)>\bar q(x)$ (that is, there is a positive probability of global survival and
nonlocal survival in $A$ starting from $x$) then there exists $y \in X$ such that
$q_0(y,A)>\bar q(y)$ (that is, there is a positive probability that the colony survives globally starting from $y$ without
ever visiting $A$). Of course, $q_0(x,A)>\bar q(x)$ implies $q(x,A)>\bar q(x)$ but the converse is not true.
In particular for a BRW with no death there is strong local survival in $A$ starting from $x$ for all $x \in X$ if and only if
the probability of visiting $A$ is $1$ starting from every vertex.
This is the BRW counterpart of an analogous result in random walk theory; a vertex $x$
is transient
if and only if there exists a vertex $y$ such that with positive probability the walker never visits $x$ starting
 from $y$.

We note that, \textit{a priori}, there is not an order relation between the events
``visiting $A$ at least once starting from $x$'' and ``global survival starting from $x$''. Nevertheless
if, for all $x \in X$,  the probability of ``visiting $A$ at least once starting from $x$'' is larger or equal to
the probability of ``global survival starting from $x$'' then, using equation~\eqref{eq:strongconditioned},
we have that the probability of ``global survival starting from $x$ never visiting $A$'' is $0$.
\end{rem}
An application of the previous remark to the construction of a BRW which survives globally and
locally, but not strong locally, is given in Example~\ref{rem:nonstrongandstrong}.

\begin{rem}\label{rem:q0}
 We observe that if $d(x,A):=\min\{n \in \N, y \in A \colon m_{xy}^{(n)}>0\}$ then
 $\mathbf{q}_n(x,A)=\mathbf{q}_0(x,A)$ for all $x$ such that $d(x,A) \ge n$. Hence,
 $\mathbf{q}_1(x,A)=\mathbf{q}_0(x,A)$ for all $x \not \in A$ and according to equation~\eqref{eq:extprobab}
 we have $\mathbf{q}_0(x,A)=G(\mathbf{q}_0;x)$ for all $x \not \in A$.
\end{rem}

If $G$ has only one fixed point $z < \mathbf{1}$ then
$q(\cdot, y)=\mathbf{1}$ or $q(\cdot, y)=\bar q(\cdot)$.
More precisely, if one can prove that $q(\cdot, y)< \mathbf{1}$ then $q(\cdot,y)=q(\cdot, A)=\bar q(\cdot)$ for all $A \ni y$.
In this case, global survival starting from $x$ (i.e.~$\bar q(x)<1$) is equivalent to local survival at $y$ starting from $x$ and
it implies strong local survival at $y$ starting from $x$.
If for some $y \in X$ we have $q(\cdot,y)= \bar q$ then the global survival starting from $x$ implies
the strong local survival at $y$ starting from $x$.

If equation~\eqref{eq:particular1} holds and $\rho(n)=\frac{1}{1+\bar \rho_x} (\frac{\bar \rho_x}{1+\bar \rho_x} )^n$,
we have that the survival probability in $A$, $v_A:=\mathbf{1}-q(\cdot, A)$, satisfies the equality
$M v_A=v_A/(\mathbf{1}-v_A)$ (see equation~\eqref{eq:Gcontinuous}).
In particular in the continuous-time case we have $\lambda K v_A=v_A/(\mathbf{1}-v_A)$.

\begin{exmp}
\label{ex:branchingprocess}
In the case of a Galton--Watson branching process, the generating function is
$G(z)=\sum_{n \in \N} \mu(n) z^n$ and its smallest fixed point $\bar q \in [0,1]$
satisfies $\bar q<1$ if and only if $1< \frac{\diff}{\diff z}G(z)|_{z=1}=\sum_{n \in \N} n \mu(n)= m_{xx}$
(cfr.~Assumption~\ref{assump:1} and Remark~\ref{rem:assump1}).

Moreover denote by $H(t)=\sum_{i=1}^\infty r_i t^i$ the generating function
of the total number of descendants, that is, $r_i$ is the probability 
that the total number of descendants (including the original particle) is $i$ for all $i \in \N$. Note that clearly $r_0=0$. 
It is easy to show that $H(t)=t \cdot G(H(t))$. Clearly $G(1)=1$ and 
$H(1) \leq 1$ (since there might be a positive probability of an infinite progeny, for instance
if $G^\prime(1)>1$. More precisely, $H(1)=\sum_{i \in \N} r_i=\bar q$, hence it is the smallest fixed point of $G$.
The expected number of descendants conditioned on extinction (if $\bar q \not = 0$) is $H^\prime(1)/\bar q$ 
where $H^\prime(1)$ satisfies
$H^\prime(1)=G(H(1))+G^\prime(H(1)) H^\prime(1)=G(\bar q)+ G^\prime(\bar q) H^\prime(1)$. 
Note that, under Assumption~\ref{assump:1}, $G^\prime(\bar q)=1$ if and only if $G^\prime(1)=1$ and this
implies that $\bar q=1$.
Hence, 
\[
H^\prime(1)=
\begin{cases}
1/(1-G^\prime(\bar q)) & \text{if } G^\prime(1) \not = 1 \\
+\infty & \text{se } G^\prime(1)=1.
\end{cases}
\]
In particular, $H^\prime(1)$ is finite if and only if $G^\prime(1) \not = 1$.

If we are dealing with a continuous-time branching process with reproduction rate
$\lambda k_{xx}$ then $G(z)=1/(1+\lambda k_{xx}(1-z))$ and $\bar q= \min(1,1/\lambda k_{xx})$.
We see that, according to the general case, $\bar q <1$ if and only if $\lambda k_{xx}>1$.
\end{exmp}

\subsection{Local survival}
\label{subsec:local}

The fact that there is local survival or not, depends only on the first-moment matrix $M$.
Indeed we have the following characterization which contains \cite[Theorem 4.1]{cf:Z1}
(some hints about the proof can be found in Section~\ref{sec:proofs}). We note that the following
result still holds without the hypothesis $\sup_{x \in X} \sum_{y \in X} m_{xy}<+\infty$.

\begin{teo}\label{th:equiv1local}
Let $(X,\mu)$ be a BRW.
\begin{enumerate}[(1)]
 \item
There is
local survival starting from $x$ if and only if $M_s(x,x) 
>1$.
\item
If $\sup_{w \in X: x \to w \to y} M_s(w,w) 
>1$ then
there is local survival at $y$ starting from $x$. Moreover if the cardinality of
${w \in X: x \to w \to y}$ is finite (for instance if $X$ is finite) the converse is true.
\end{enumerate}
\end{teo}

We recall that a useful characterization $M_s(x,x)=
\left ( \max \{\lambda \in \mathbb{R}: \Phi(x,x|\lambda) \le 1\}\right )^{-1}$ was
given in Section~\ref{subsec:genfun}. 
Clearly, since $t \mapsto\Phi(x,x|t)$ is left continuous and strictly increasing,
$M_s(x,x)>1$ if and only if
$\Phi(x,x|1)>1$, which is another condition equivalent to local survival at $x$.
Note that $M_s(x,x)$ depends only on the values  $\{m_{wy}\}_{w,y \rightleftharpoons x}$.
Thus the BRW survives locally at $x$ if and only if it does so when restricted to the irreducibility
class of $x$.
If the BRW is irreducible then
$M_s(x,x) 
=M_s(w,y)$ 
for all $x,w,y \in X$; hence there is local survival at $y$ starting from $x$ if and only if
$M_s(w,w) 
>1$ for some $w \in X$ (equivalently, for all $w \in X$).

To compare with random walk theory,
the reader might recall the definition of spectral radius of a random walk on $X$ with transition matrix $P$
as $1/\limsup_{n\to\infty}\sqrt[n]{p^{(n)}{(x,x)}}$
(see the discussion after Proposition~\ref{pro:F-BRW1} and \cite[Section 2.C]{cf:Woess09}).
If $\limsup_{n\to\infty}\sqrt[n]{p^{(n)}{(x,x)}}<1$
then the random walk is
transient, i.e.~it returns to $x$ a finite number of times almost surely
(where $p^{(n)}(x,x)$ are the $n$-step return probabilities); on
the other hand if $\limsup_{n\to\infty}\sqrt[n]{p^{(n)}{(x,x)}}=1$ then the process
may be either recurrent or transient. We observe that, while for a random walk
the probability of returning to a site infinitely many times obeys a 0-1 law,
a BRW can survive locally with any probability in $[0,1]$. Of course,
survival with probability one is possible only in the no-death case.

We show that the sufficient condition $\sup_{w \in X: x \to w \to y} M_s(w,w)
>1$
stated in Theorem~\ref{th:equiv1local} is not necessary in the reducible case when the cardinality of $X$
is infinite.

\begin{exmp}\label{exm:reducible}
 Let $X:=\N \times \{0,1\}$ (see Figure~\ref{fig:stripn}), fix $p>1/2$
 and consider the BRW with the following reproduction rules:
\begin{enumerate}[(a)]
 \item every particle at $(i,0)$ has 2 children at $(i+1,0)$ and 1 child at $(i,1)$ with probability $p$
and no children with probability $1-p$ (for all $i \in \N$);
\item every particle at $(i,1)$ has 2 children at $(i-1,1)$ with probability $p$
and no children with probability $1-p$ (for all $i \ge 1$);
\item every particle at $(0,1)$ has 1 child  at $(0,1)$ with probability $p$
and no children with probability $1-p$.
\end{enumerate}
Clearly $\limsup_{n\to\infty}\sqrt[n]{m^{(n)}_{ww}}=0$ for all $x \in X$. Nevertheless,
there is local survival at $(i,1)$ for all $i \in \N$ starting from $(j,0)$ for any fixed $j \in \N$
(note that there is no local survival at any $(i,0)$ for every starting point).
\end{exmp}

\begin{figure}
    \begin{center}
    \subfigure[\tiny The graph $\N \times \{0,1\}$ of Example~\ref{exm:reducible}.]{\includegraphics[height=3cm]{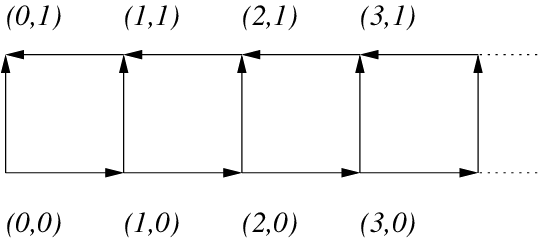}
    \label{fig:stripn}}
       \hspace{0.5cm}
     \subfigure[\tiny The graph of Example~\ref{exm:nonamenable} (the circles are loops).]{\includegraphics[height=4.5cm]{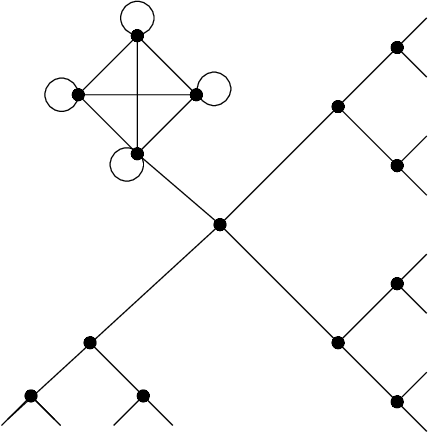}
    \label{fig:t3complete}}
    \end{center}
    \caption{}
\end{figure}


It is worth noting that if $[x]$, the irreducible
class of $x \in X$, is finite, then $M_s(x,x)$
is the Perron-Frobenius eigenvalue
of the submatrix $M^\prime:=(m_{yz})_{y,z \in [x]}$. In this case
there is local survival at $x$ if and only if $\max\{t >0 :
\exists v \not = \mathbf{0}, \, M^\prime v =t v\}>1$.

In the case of continuous-time BRWs with
rates $\{\lambda k_{xy}\}_{x,y \in X}$ one is also interested in the characterization
of the local critical parameter $\lambda_s(x)$. Moreover one may also wonder
whether at $\lambda= \lambda_s(x)$ there is survival or extinction.
We already observed that the behavior of the continuous-time BRW is equivalent  to
the behavior of
its discrete-time counterpart (that is, the BRW with independent diffusion
where $\{\mu_x\}_{x \in X}$ is
given by equations \eqref{eq:particular1} and \eqref{eq:counterpart}).
If we apply Theorem~\ref{th:equiv1local} then
we obtain the following corollary, which gives the strong critical value and states
that at the local critical value there is local extinction a.s.

\begin{cor}\textbf{\cite[Theorems 4.1 and 4.7]{cf:BZ2}}\label{cor:pemantleimproved}
Given a continuous-time BRW $(X,K)$, $\lambda_s(x)=1/K_s(x,x)=1/\limsup_{n\to\infty}\sqrt[n]{k^{(n)}_{xx}}$.
If $\lambda=\lambda_s(x)$ then there is local extinction at $x$.
\end{cor}

It is worth mentioning that, for a edge-breeding, irreducible BRW the critical value
$\lambda_s$ (not depending on $x$) was already identified in \cite[Lemma 3.1]{cf:PemStac1}, even though
the critical behavior was not known.

Another result about local survival, in the context of irreducible BRWs with no death, is the following which relies on
the existence of positive superharmonic functions; this one has
a well-known counterpart in the random walk theory (see \cite[Theorem 6.21]{cf:Woess09}).

\begin{teo}\textbf{\cite[Theorem 2.1]{cf:M08}}
\label{th:Muller08}
Let $(X,\mu)$ be an irreducible BRW such that $\rho_x(0)=0$ for all $x \in X$. There is local extinction if and only if
there exists a strictly positive function $f$ on $X$ such that $Mf \le f$.
\end{teo}

The proof of the previous theorem is inspired by the proof of
Theorem~\ref{th:MenshikovVolkov}.
Note that in the irreducible case with no death, one can obtain  Theorem~\ref{th:equiv1local}
directly from Theorem~\ref{th:Muller08}
Indeed, see \cite[Section 7.A]{cf:Woess},
$ M_s 
= \min \{t>0: \exists f:X \rightarrow (0,+\infty), \, Mf \le tf \}$
(remember that in the irreducible case $M_s(x,y)$ does not depend on $x,y \in X$).

\subsection{Global survival}
\label{subsec:global}

The classical approach to estimate the probability of extinction of a branching process
uses the fact that this probability is the minimal fixed point in $[0,1]$ of the generating
function of the law of the number of children. Theorem~\ref{th:equiv1global} extends this approach to
BRWs; indeed
in this case the vector of global extinction probability $\bar q \in [0,1]^X$ is
the smallest fixed point of the infinite dimensional generating function $G$
(see Section~\ref{subsec:survivalprob}). Global survival is equivalent to the existence
of a fixed point of $G$ strictly smaller than $\mathbf{1}$.
We remark here that given $v,w \in [0,1]^X$
by $v < w$ we mean, as usual, that $v \le w$ and $v \not = w$, that is
$v(x) \le w(x)$ for all $x \in X$ and $v(x_0)<w(x_0)$ for some $x_0 \in X$.
Theorem~\ref{th:equiv1global} gives an equivalent condition for global survival and a necessary one.

\begin{teo}\textbf{\cite[Theorem 4.1]{cf:Z1}}\label{th:equiv1global}
Let $(X,\mu)$ be a discrete-time BRW.
\begin{enumerate}[(1)]
\item There is
global survival starting from $x$ if and only if there exists $z\in [0,1]^X$, $z(x)<1$,
such that $G(z|y)\le z(y)$, for all $y \in X$
(equivalently, such that $G(z|y)= z(y)$, for all $y \in X$).
\item
If there is global survival starting from $x$, then
there exists $v\in[0,1]^X$, $v(x)>0$, such that
 $Mv \ge v$. Moreover, for all $y$, $Mv(y)=v(y)$ if and only if   
$G(\mathbf 1 -(1-t) v|y)=1-(1-t)v(y), \forall t \in [0,1]$.
\end{enumerate}
\end{teo}
It is easy to show (see for instance \cite[Section 3]{cf:BZ2}) that the first condition
implies that $\bar q \le z < \mathbf{1}$, that is, $z$ is an upper bound for the
probabilities of global extinction.
As for the second condition, one has that if  $\bar q < \mathbf{1}$ then taking
 $v=\mathbf{1}-\bar q$ (the probabilities of global survival) one obtains the inequality $Mv \ge v$.
This is the analog for a BRW of the well-known result which states that a
branching process survives if and only if the expected number of children is strictly larger than one.
Nevertheless, for a BRW, even if we prove that $Mv(y)>v(y)$ for all $y \in X$
and some $v \not = \mathbf{0}$ this does not suffice for
global survival: a counterexample is given by Example~\ref{exm:noext}.

In particular cases we can characterize global survival in terms of $M$:
the following corollary follows easily from Theorem~\ref{th:equiv1global}(1) and
equation~\eqref{eq:Gcontinuous} by taking $z=\mathbf{1}-v$.

\begin{cor}\label{cor:forgotten}
 Suppose that the generating function of $(X,\mu)$ satisfies equation~\eqref{eq:Gcontinuous}
(for instance, if $(X,\mu)$ is a BRW with independent diffusion 
where
$\rho_x(n)=\frac{1}{1+\bar \rho_x} (\frac{\bar \rho_x}{1+\bar \rho_x} )^n$) then
there is global survival starting from $x \in X$ if and only if there exists
$v \in [0,1]^X$, $v(x)>0$ such that
\[
 Mv \ge v /(\mathbf{1}-v), \qquad \text{(equivalently, }Mv = v /(\mathbf{1}-v) \text{)}
\]
(where the ratio is taken coordinatewise).
\end{cor}
Observe that the solution $v$ in the previous corollary provides a lower bound
for the probabilities of global survival $\mathbf{1}-\bar q$;
moreover, among all the solutions of either equations, the largest one is $\mathbf{1}-\bar q$.

Now we notice that in order to characterize local survival we studied the behavior
of the expected number of $n$-th generation offsprings coming back to $x$, that is
$m^{(n)}_{xx}$, see Theorem~\ref{th:equiv1local}.
For the global survival problem, we are naturally lead to investigate the behavior of
the expected number of $n$-th generation offsprings whose ancestor is a single particle at $x$,
that is $\sum_y m^{(n)}_{xy}$.
First of all, observe that it is easy to show, by using supermultiplicative arguments,
that  
$M_w(x) \ge M_s(x,x)$. 
%
%
moreover, for an irreducible BRW,
$M_w(x)$ 
does not depend on $x \in X$.

A complete description of global survival in terms of $M_w(x)$ 
is not possible
in general (we have a necessary condition); indeed Example~\ref{exm:noext} shows that
global survival does not depend on $M$ alone.
Nevertheless, a characterization of global survival by means of $M_w(x)$
holds for the
class of $\mathcal F$-BRWs.

\begin{teo}\textbf{\cite[Theorems 4.1 and 4.3]{cf:Z1}}\label{th:equiv1global3}
Let $(X,\mu)$ be a discrete-time BRW.
\begin{enumerate}[(1)]
\item If there is global survival starting from $x$, then $M_w(x) 
\ge 1$.
\item If $(X,\mu)$ is an $\mathcal{F}$-BRW then there is global survival for $(X,\mu)$ starting from $x$ if and only if
$M_w(x) 
> 1$.
\end{enumerate}
\end{teo}
We already observed that if the BRW is irreducible then $\bar q<\mathbf{1}$ implies $\bar q(x) <1$ for all
$x \in X$; Example~\ref{exm:reducible} shows that in the reducible case it might happen that
$\bar q< \mathbf{1}$ and $\bar q(x)=1$ for some $x \in X$.

In the case of continuous-time BRWs, one is interested in identifying $\lambda_w(x)$ and the
behavior of the process when $\lambda=\lambda_w(x)$.
The following result is a corollary of Theorems~\ref{th:equiv1global} and \ref{th:equiv1global3}.
It characterizes $\lambda_w(x)$ in the case of $\mathcal F$-BRWs and gives a lower bound
in the general case.

\begin{cor}\textbf{\cite[Theorems 4.3 and 4.8, Proposition 4.5]{cf:BZ2}}\label{cor:globalsurvivalcontinuous}
Consider a continuous-time BRW $(X,K)$.
\begin{enumerate}[(1)]
 \item
$\lambda_w(x) \ge 1/K_w(x)=1/\liminf_{n \in \N} \sqrt[n]{\sum_{y \in X} k^{(n)}_{xy}}$.
\item
If $(X,\mu)$ is an $\mathcal{F}$-BRW then $\lambda_w(x)=1/K_w(x)$. Moreover, if $\lambda=\lambda_w(x)$ then
there is global extinction starting from $x$.
\item
Suppose that for all $y \in X$ there exists $x \in X$ such that $x \to y$. If there is global survival starting from $x$, then
there exists $v\in[0,1]^X$, $v(x)>0$, such that $\lambda Kv > v$.
\end{enumerate}
\end{cor}
Note the difference between Theorem~\ref{th:equiv1global}(2) and Corollary~\ref{cor:globalsurvivalcontinuous}(3):
in the second case we have the strict inequality $Mv=\lambda Kv>v$.
As a consequence of the previous corollary we can compute the global critical value for two frequently used
classes of continuous-time BRWs which are locally isomorphic to a branching process: for a site-breeding BRW
$\lambda_w(x)=1/k$ for all $x \in X$ (where $k(x)=k$ for all $x \in X$),
while for an edge-breeding BRW on a regular graph of degree $d$ we have that $\lambda_w(x)=1/d$ for all $x \in X$.

\begin{rem}\label{rem:finiteclasses}
 Consider a BRW $(X,\mu)$ where $X$ is finite. Following \cite[Remark 4.4]{cf:BZ2} one can prove that
$M_w(x) 
=\max_{w:x \to w}
M_s(w,w)$; 
remember that global survival starting from $x$ is equivalent to global (hence local)
survival in some class (hence at some point $w$) since $X$ is finite.
For a continuous-time BRW, this means that $K_w(x)=\max_{w:x \to w} K_s(w,w)$ and hence
$\lambda_w(x)=\min_{w:x \to w} \lambda_s(w)$.
If the BRW is finite and irreducible then there is only one class of irreducibility and the previous
results hold without $\max$ and $\min$.
\end{rem}

Still in the case of continuous-time BRWs, we have a
characterization of $\lambda_w(x)$, which makes use of the existence of a solution of certain
systems of inequalities.

\begin{teo}\textbf{\cite[Theorem 4.2]{cf:BZ2}}
\label{th:equiv1global2}
Let $(X,K)$ be a continuous-time BRW and let $x \in X$.
\begin{enumerate}[(1)]
\item
For any fixed $\lambda>0$ there is global survival starting from $x \in X$ if and only if there exists a solution
$v \in [0,1]^X$
of the inequality $\lambda Kv \ge v/(\mathbf{1}-v)$ such that $v(x)>0$.
\item If $\lambda \le \lambda_w(x)$ and $v \in [0,1]^X$ is such that $\lambda Kv \ge v/(\mathbf{1}-v)$
then $\inf_{y:x \to y,v(y) > 0} v(y) = 0$.
\item 
\[
\lambda_w(x)
=\inf \{\lambda \in \R: \exists v \in l^\infty_+(X), v(x) >  0 \text{ such that } \lambda K v  \ge v/(\mathbf{1}-v)\}.
\]
\item For all $n \in \N$, $n \ge 1$ we have
\[
\lambda_w(x)=
\inf \{\lambda \in \R: \exists v \in l^\infty_+(X), v(x)> 0 \text{ such that } \lambda^n K^n v \ge v \}.
\]
\end{enumerate}
\end{teo}

We note that, by taking $n=1$ in Theorem~\ref{th:equiv1global2}(4), we have that
$\lambda_w(x)=\inf \{\/ \undertilde r\,\!_K(v): v \in l^\infty(X), v(x)=1\}$ where
$\undertilde r\,\!_K(v)$ is the
 \textit{lower
Collatz-Wielandt number of} $v$ (see \cite{cf:FN1}, \cite{cf:FN2} and \cite{cf:Marek1}).
In particular, according to Theorem~\ref{th:equiv1global2}(4), we have that
for a continuous-time BRW
\[
\begin{split}
 \lambda_w(x)&=\inf \{\lambda > 0 : \exists v \in l^\infty_+(X), v(x)> 0 \text{ such that } \lambda K v \ge  v \}\\
&=\inf \{\lambda > 0 : \exists v \in l^\infty_+(X), v(x)> 0 \text{ such that } \lambda K v =  v \}
\end{split}
\]
while, according to Theorem~\ref{th:Muller08}, if the BRW is irreducible,
\[
 \lambda_s:=\max \{\lambda > 0 : \exists v > 0 \text{ such that } \lambda K v \le  v \}.
\]

Finally, we know that a continuous-time BRW dies out locally at $x$ when $\lambda=\lambda_s(x)$ (see Theorem~\ref{th:equiv1local}).
Theorem~\ref{th:equiv1global2}(2) states that the vector of probabilities of survival $v$ of a generic BRW at the critical point $\lambda_w(x)$
if it is not equal to $\mathbf{0}$ it satisfies $\inf_{y \in X} v(y)=0$. This proves immediately that an irreducible $\mathcal{F}$-BRW
dies out globally when $\lambda=\lambda_w$ (which is independent of $x \in X$).
Theorem~\ref{th:equiv1global2}(2) is the most reasonable result we can expect in full generality; indeed, here is an example of an irreducible BRW which
survives globally when $\lambda=\lambda_w$.

\begin{exmp}\textbf{\cite[Example 3]{cf:BZ2}}
\label{exm:4}
Let $X:=\N$ and $K$ be defined by
$k_{0\, 1}:=2$, $k_{n\, n+1}:=(1+1/n)^2$, $k_{n\, n-1}:= 1/3^{n}$
(for all $n \ge 1$) and $0$ otherwise.
Note that the corresponding continuous-time BRW is irreducible.
In order to show that $\lambda_w=1$ we look for solutions of the inequality
 $\lambda Kv \ge v/(\mathbf{1}-v)$. The system becomes
\[
\begin{cases}
2 \lambda v(1) \ge v(0)/(1-v(0)) & \\
\lambda(v(n+1) (1+1/n)^2+v(n-1)/3^n) \ge v(n)/(1-v(n))& \text{ for all }n\ge1.\\
\end{cases}
\]
Clearly, for all $\lambda \ge 1$,
$v(0)=1/2$ and $v(n):=1/(n+1)$ (for all $n \ge 1$) is a solution; this implies, according to Theorem~\ref{th:equiv1global2}(1), that there is global survival
for $\lambda \ge 1$, thus
$\lambda_w \le 1$.
If $\lambda <1$ then there are no solutions in $l^\infty_+(X)$.
Indeed one can prove by induction that any solution must satisfy
$v(n+1)/v(n) \ge \frac1\lambda\left(\frac{n}{n+1}\right)^2\left(
1-\frac{1}{2^n}\right)$ for all $n \ge 2$. Thus
$v(n+1)/v(n)$ is eventually larger than $1+\eps$ for some $\eps>0$,
hence either $v=\mathbf 0$ or $\lim_n v(n)= +\infty$.
This implies that $\lambda_w=1$ and there is global survival if $\lambda=\lambda_w$.
\end{exmp}

Another result, which applies to edge-breeding, irreducible, continuous-time BRWs
and which deals with the relation between $K_w$ and $\lambda_w$ is the following.

\begin{teo}\textbf{\cite[Theorem 3.3]{cf:BZ}}\label{th:Tnk}
Let $(X,K)$ be an edge-breeding, irreducible, continuous-time BRW on a multigraph $X$;
let us suppose that there exists $x_0 \in X$, $Y \subseteq X$ and $n_0 \in \N$
such that
\begin{enumerate}[(1)]
\item for all $x \in X$ we have that $B^+(x,n_0) \cap Y\not = \emptyset$;
\item for all $y \in Y$ there exists an injective map $\varphi_y:X \to X$, such that $\varphi_y(x_0)=y$ and
$k_{\varphi_y(x)\varphi_y(z)} \geq k_{xz}$ for all $x,z \in X$,
\end{enumerate}
where $B^+(x,n_0)$ is the set of all vertices which can be reached from $x$ in at most
$n_0$ steps.
Then $\lambda_w=1/K_w$.
\end{teo}


The previous theorem is based on the following result (see \cite[Theorem 3.1]{cf:BZ})
which gives an interesting sufficient condition for the equality
 $\lambda_w=1/K_w$.
If the multigraph satisfies this geometrical condition:
\[
\forall \varepsilon>0\, \exists\bar n=\bar n(\varepsilon):
\sup_{n\le \bar n}
\sqrt[n]{T_x^{n}}\ge K_w-\varepsilon,\, \forall x\in X
\]
then $\lambda_w=1/K_w$.
Note that, by definition of $K_w$, for all fixed $\varepsilon>0$ and $x \in X$, there exists
$n_x$ such that $\sqrt[n_x]{T_x^{n_x}}\ge K_w-\varepsilon$.
The above condition is a request of uniformity in $x$.
An application of Theorem~\ref{th:Tnk} is given in
the following example.

\begin{exmp}\textbf{\cite[Example 3.3]{cf:BZ}}\label{exmp:0}
Given a sequence of positive natural numbers $\{m_k\}_{k\ge 1}$
we construct a non-oriented, rooted tree $\T$ (with root $o$) such that if $x\in \T$ satisfies $d(o,x)=k$
then it has $m_{k+1}$ neighbors $y$ such that $d(o,y)=k+1$ (where $d$ is the natural distance of
a non-oriented graph). We call this radial graph $T_{\{m_k\}}$-tree.
If the sequence is periodic of period $b$, then Theorem~\ref{th:Tnk} applies
with $x_0=o$, $n_0=b$, $Y:=\cup_{n\in\N} S(o,nb)$ (where $S(o,nb)$ is the sphere with center $o$ and radius
$nb$ with respect to the distance $d$) and  $\varphi_y$
(where $y \in Y$) maps isomorphically
the tree $\T$ onto the subtree branching from $y$; in this case
the global critical parameter $\lambda_w$ of the
(irreducible) edge-breeding, continuous-time BRW on the tree
equals $1/K_w$.
Note that for every periodic sequence the BRW is
not an $\cF$-BRW, hence Corollary~\ref{cor:globalsurvivalcontinuous}(2) does not apply.
%
\end{exmp}

We already observed that local survival depends only on the first moment matrix $M$.
It is clear that if we investigate the global survival in a class of BRWs where there is a
one-to-one correspondence
between first moment matrices and processes (as in the case of continuous-time BRWs), then also
the global survival depends only
on $M$. This is also true, by Corollary~\ref{cor:forgotten},
for a BRW with independent diffusion satisfying 
equation~\eqref{eq:counterpart}.
On the other hand, for a generic BRW,
according to the following example,
the global survival does not depend exclusively on $M$; in
particular, even $M_w(x) 
>1$ does not imply global survival starting from $x$.

\begin{exmp}\textbf{\cite[Example 4.4]{cf:Z1}}\label{exm:noext}
Let $X=\mathbb N$ and consider the family of BRWs $(\mathbb N,\mu)$ with
$\mu_i=p_i \delta_{n_i \ident_{\{i+1\}}+\ident_{\{i-1\}}} +(1-p_i)\delta_{\bf 0}$ (for all $i\ge 1$) and
$\mu_0=p_0 \delta_{n_0 \ident_{\{1\}}} +(1-p_0)\delta_{\bf 0}$.
Roughly speaking, each particle at $i \ge 1$ has $n_i$ children at $i+1$ and $1$ at $i-1$ with probability $p_i$ and
no children at all with probability $1-p_i$; each particle at $0$ has $n_0$ children at $1$ with probability $p_0$
and no children at all otherwise.

According to Theorem~\ref{th:equiv1global}(1) global survival starting from $0$ is equivalent
to the existence of $z\in [0,1]^{\mathbb N}$, $z(0)<1$,
such that $G(z|i)\le z(i)$, for all $i$ where
\[
G(z|i)=
 \begin{cases}
p_i z(i+1)^{n_{i}}z(i-1) +1-p_i & i \ge 1\\
p_0 z(1)^{n_{0}} +1-p_0 & i =0.
 \end{cases}
\]
The trick is to choose the sequences $\{p_i\}_{i\in \N}$ and $\{n_i\}_{i\in \N}$ such that
$p_i \to 0$ fast enough and $p_i n_i=2$ for all $i \in \N$; this way, the unique solutions
of $G(z)\le z$ is $z= \mathbf 1$.
All the details can be found in \cite[Example 4.4]{cf:Z1}.

On the other hand, if the BRW is given by
$\mu_i=1/2 \, \delta_{4 \ident_{\{i+1\}}}+p_i \delta_{\ident_{\{i-1\}}} +(1/2-p_i)\delta_{\bf 0}$ (for all $i\ge 1$) and
$\mu_0=1/2 \delta_{4 \ident_{\{1\}}} +1/2 \delta_{\bf 0}$ (where $p_i$ is the same as before) then
the first-moment matrix $M$ is the same as before, but in this case the process
survives globally
(the total number of particles dominates a branching process with $\bar \rho=2$).
Moreover, $M_w(x_0) 
\ge 2$ since at each step the expected number
of children $\bar \rho_x$ is at least $2$ for all $x \in X$.
\end{exmp}

Another legitimate question arises from Theorem~\ref{th:equiv1global3}: is it true that $\sum_{y \in X} m_{xy}<1$
for all $x \in X$ implies global extinction? According to the following example
(see also \cite[Example 1]{cf:BZ2}), the answer is negative.

\begin{exmp}\label{exm:4.5}
We start by considering a reducible BRW.
Let $X=\N$, $\{p_n\}_{n\in\N}$ be a sequence in $(0,1]$ and suppose that a particle at $n$
has one child at 
$n+1$ with
probability $p_n$ and no children with probability $1-p_n$.
The generating function of this process is
$\widetilde G(z|n)=1-p_n + p_n z(n+1) $.
The probability of extinction of this process, starting with one particle at $n$, equals
$z(n)=1-\prod_{i=n}^{\infty} p_i$ ($z$ is the smallest solution of $\widetilde G(z)=z$); hence it survives wpp,
if and only if $\sum_{i=1}^\infty (1-p_i) < +\infty$.

This process is stochastically dominated by the irreducible BRW where
each particle at  $n\ge1$ has one child at $n+1$ with
probability $p_n$, one child at $n-1$ with probability $(1-p_n)/2$
(if $n=0$ then it has one child at $0$ with probability $(1-p_0)/2$)
and no children at all with probability $(1-p_n)/2$.
The generating function $G$ can be explicitly computed
\[
G(z|n)=
\begin{cases}
\frac{1-p_n}{2} + \frac{1-p_n}{2} z(n-1)+ p_n z(n+1) & n \ge 1 \\
\frac{1-p_0}{2} + \frac{1-p_0}{2} z(0) + p_0 z(1) & n=0.\\
\end{cases}
\]
By coupling this process with the previous one (see \cite[Section 3.3]{cf:Z1})
or, simply, by applying Theorem~\ref{th:equiv1global}(1)
($z(n)=1-\prod_{i=n}^{\infty} p_i$ is a solution of $G(z) \le z$) one can prove that
$\sum_{i=1}^\infty (1-p_i) < +\infty$ implies global survival. Note that here
$\sum_{j \in \N} m_{ij} = (1+p_i)/2 <1$; clearly, $M_w(i) 
=1$.
\end{exmp}

The following example shows how to apply the results of this section
to the study of a couple of interesting branching processes.
The first one appears for instance in the proof of \cite[Proposition 3.6]{cf:Muller08-2}
(see also Corollary~\ref{cor:Mueller08}).

\begin{exmp}\label{exm:BP1-2}
Let $\rho$ be a measure on $\N$ with generating function
$\phi(z):=\sum_{n\in \N} \rho(n) z^n$ and denote by $\bar \rho=\frac{\diff}{\diff z}\phi(z)|_{z=1}$
(suppose that $\rho(0)<1$).
Consider the following Galton-Watson branching processes. BP$_1$ is the process
where each particle gives birth to $n$ children with probability $\rho(n)$ and each
newborn particle is killed (independently) with probability $1-p$. BP$_2$ is the process
where each particle is killed (independently) before breeding with probability $1-p$, otherwise it
gives birth to $n$ children with probability $\rho(n)$. We suppose that $p \in(0,1)$ to avoid trivial
situations.

In order to study these two branching processes simultaneously, consider the BRW
$\{\eta_n\}_{n \in \N}$ defined by $X:=\{1,2\}$ and
\[
 \mu_1(a,b):=
\begin{cases}
 0 & \textrm{if } a\not = 0 \\
 \rho(b) & \textrm{if } a = 0 \\
\end{cases}
\qquad
 \mu_2(a,b):=
\begin{cases}
 0 & \textrm{if } b\not = 0  \textrm{ or } a\ge 2\\
 1-p & \textrm{if } a =b= 0 \\
 p & \textrm{if } a =1,\, b= 0, \\
\end{cases}
\]
where $\mu_j(a,b)\equiv\mu_j(f:f(1)=a,f(2)=b)$.
Clearly
\[
 G(z_1,z_2)=\left ( \phi(z_2), pz_1+1-p
\right ), \qquad \forall z_1,z_2 \in [0,1].
\]
Note that $\{\eta_{2n}(1)\}_{n \in \N}$ and $\{\eta_{2n}(2)\}_{n \in \N}$ are
realizations of BP$_1$ and BP$_2$ respectively.
Indeed the 2-step generating function of $\{\eta_n\}_{n \in \N}$ is
given by $G^{(2)}(z_1,z_2|1)=\phi(pz_1+1-p)$ and $G^{(2)}(z_1,z_2|2)=p\phi(z_2)+1-p$;
$z \mapsto \phi(pz+1-p)$ is the generating function of BP$_1$ and $z \mapsto p\phi(z)+1-p$
is the generating function of BP$_2$.
Note that there is no distinction between local and global survival for the BRW since it is irreducible and
$X$ is finite. Moreover the survivals of the BRW, the BP$_1$ and the BP$_2$ are all equivalent and, in turn, they are
equivalent, for instance, to $1 < \frac{\diff}{\diff z} \phi(pz+1-p)|_{z=1}=p \bar \rho$
(see Example~\ref{ex:branchingprocess}).
The vector of extinction probabilities
satisfies $\bar q=G(\bar q)$ that is
\[
 \begin{cases}
\phi(\bar q(2))=\bar q(1) \hfill &\\
p\bar q(1) +1-p=\bar q(2). \hfill &\\
 \end{cases}
\]
Note that $\bar q(1)$ (resp.~$\bar q(2)$) is also the extinction probability of BP$_1$ (resp.~BP$_2$)
since $\bar q=(\bar q(1), \bar q(2))$ is also the smallest fixed point of
$G^{(2)}$ (see Section~\ref{subsec:survivalprob}).
Clearly $\bar q(1)<1$ if and only if $\bar q(2)<1$ and, in this case, $p(\bar q(2)-\bar q(1))= (1-p)(1-\bar q(2))>0$.
This
implies that if there is survival then the probability of survival of BP$_1$ is strictly larger
than the probability of survival of BP$_2$.
The same result can be obtained by convexity and by the fact that $\phi(1)=1$ which implies
the following order relation between the generating functions of BP$_1$ and BP$_2$:
$\phi(pz+1-p) < p\phi(z)+1-p$ whenever $z \not = 1$ and $p\in (0,1)$, whence $\bar q(1)<\bar q(2)$.
Finally if we denote by $\bar \alpha$ the probability of extinction of BP$_1$ when $p=1$, that is,
the smallest solution in $[0,1]$ of $\phi(z)=z$, from the inequality $\phi(\bar q(2))=\bar q(1)< \bar q(2)$ we have
also $\phi(\bar q(1)) <\phi(\bar q(2)) =\bar(q(1))$. Hence $\bar \alpha < \min(\bar q(1), \bar q(2))$.
\end{exmp}

\subsection{Strong local survival}
\label{subsec:stronglocal}

The main result of this section is the following proposition.

\begin{pro}\label{pro:qtransitive}
 Let $(X, \mu)$ be an irreducible and quasi-transitive BRW.
Then the existence of $x \in X$ such that there is local survival at $x$ (i.e.~$q(x,x) < 1$)
implies that there is strong local survival at $y$ starting from $w$ for
every $w,y \in X$ (i.e~$q(w,y)=\bar q(w)$).
\end{pro}

In the particular case of a quasi-transitive, irreducible BRW with no death
and with independent diffusion, 
Proposition~\ref{pro:qtransitive} was proved in~\cite[Theorem 3.7]{cf:Muller08-2}.
The proof we give in Section~\ref{sec:proofs} is of a different nature and it is a corollary of the following result
which describes some properties of fixed points of $G$ in the case of an $\mathcal{F}$-BRW.

\begin{teo}\label{th:Fgraph}
 Let $(X, \mu)$ be an $\mathcal F$-BRW.
Then, there exists at most one fixed point $z$ for $G$ such that $\sup_{x \in X} z(x)<1$,
namely $z=\bar q$.
Hence
for all $x \in X$, 
$q(\cdot,x) = \bar q(\cdot)$ or $\sup_{w \in X} q(w,x)= 1$. In particular when $(X,\mu)$ is irreducible then
it is either $q(x,x)=\bar q(x)$ for all $x \in X$ or $\sup_{x \in X} q(x,x)= 1$.
\end{teo}

The proof of this theorem, which can be found in Section~\ref{sec:proofs}, relies on
Lemmas~\ref{lem:convexity}~and~\ref{lem:generationn} which guarantee the strict convexity of
the function $G$ evaluated on a line in $[0,1]^X$.
The existence of an example of an irreducible $\mathcal F$-BRW where $\bar q(x) < q(x,x)<1$ for all $ x \in X$ is given
in Example~\ref{ex:nonstronglocalFBRW}.

In particular we can describe the case when $X$ is finite
(not necessarily irreducible).
Clearly in this case $\bar q(w)=\min_{x \in X: w \to x} q(w,x)$, hence for all $w$ such that $\bar q(w)<1$
there exists $x$ such that $q(w,x)=\bar q(w)$.
Moreover, using Theorem~\ref{th:Fgraph}, for all $x \in X$ we have that
it is either $q(\cdot, x)=\bar q(\cdot)$ or there exists $w \in X$ such that $q(w,x)=1$. If the BRW is
irreducible (and $X$ is finite) then it is 
$\bar q(w)=q(w,w)$ for all $w \in X$ or $q(w,x)=1$ for all $w,x \in X$.
Recall that, in the irreducible case,  if $\rho_x(0)>0$ for
all $x \in X$, then strong local survival is a common property of all vertices
as local and global survival are (see discussion in .
Section~\ref{subsec:survivalprob}).
This is clearly false in the reducible case but it might be false as well in the
irreducible case if we drop the assumption $\rho_x(0)>0$ for
all $x \in X$ as Example~\ref{exm:irreduciblenotstrongeverywhere} shows.

If we are dealing with a continuous-time BRW, it might happen that if $\lambda$ is small enough or large enough there is
strong local survival but in a intermediate interval for $\lambda$ there might be global and local
survival with different probabilities.
You can find this behavior in the BRW of Example~\ref{rem:nonstrongandstrong}
which is inspired by Remark~\ref{rem:strongconditioned}.
In particular this shows that, unlike local and global survival, strong local survival
is not \textit{monotonic}.

The following result is
a natural generalization of \cite[Theorem 3.1]{cf:MenshikovVolkov}. We give a sketch of
the proof in Section~\ref{sec:proofs}.

\begin{teo}
\label{th:MenshikovVolkov}
Let $(X,\mu)$ be an irreducible, globally surviving BRW. 
Then there is no strong local survival
if and only if  there exists a finite, nonempty set $A \subseteq X$ and a function $v \in [0,1]^X$ such that
$\bar q \le v$ and
\begin{equation}\label{eq:MenshikovVolkov}
 \begin{cases}
  G(v|x) \ge v(x), & \forall x \in A^\complement,\\
  (\tau_{\bar q}v)(x_0) > \max_{x \in A} (\tau_{\bar q}v)(x) & \textrm{for some } x_0 \in A^\complement,
 \end{cases}
\end{equation}
where $\tau_{\bar q}v=(v-\bar q)/(\mathbf{1}-\bar q)$ is the inverse of the map $T_{\bar q}$ defined
in Section~\ref{subsec:nodeath} (and the ratio is taken coordinatewise).
\end{teo}

It is worth mentioning at least one result for irreducible BRWs with no death.
The following proposition gives a general criterion for the strong local survival
of a BRW with no death and with independent diffusion. 

\begin{pro}\textbf{\cite[Lemma 3.4]{cf:Muller08-2}}\label{pro:Mueller08}
 Let $(X,\mu)$ be an irreducible BRW with independent diffusion 
where $\rho_x(0)=0$ for all $x \in X$. If for some $c>0$ the set $C:=\{x:q(x,x) \le 1-c\}$
is visited infinitely often by the BRW, then there is strong local survival.
\end{pro}

As a corollary one can prove Proposition~\ref{cor:Mueller08} concerning BRWs which are locally isomorphic
to branching processes. Note that, since in the previous proposition we deal with BRWs with
independent diffusion, it is possible to substitute the hypothesis that ``the BRW visits infinitely often
the set $C$'' with ``the set $C$ is recurrent for the random walk $(X,P)$''.

In view of the discussion of Section~\ref{subsec:nodeath}, a reasonable hypothesis for a generalization
of the previous result to a BRW where $\rho_x(0) \neq 0$ for some $x \in X$, could be
the fact that the BRW visits infinitely often the set $C:=\{x:(1-q(x,x))/(1-\bar q(x)) \ge c\}$
for some $c>0$.

\subsection{Pure global survival}
\label{subsec:pureweak}

The idea of \textit{pure global survival} has been introduced in continuous-time
BRW theory (and, more generally, in interacting
particle theory) to define the situation where $\lambda_s(x)>\lambda_w(x)$. In this case
for every $\lambda \in (\lambda_w(x), \lambda_s(x)]$ there is a positive probability
of global survival starting from $x$ but the colony dies out locally at $x$ almost surely.
We know that when $\lambda=\lambda_s(x)$ there is local extinction at $x$ due to
Corollary~\ref{cor:pemantleimproved}, while if
  $\lambda=\lambda_w(x)$ both global extinction or global survival starting from $x$
are possible (due to Example~\ref{exm:4} and Corollary~\ref{cor:globalsurvivalcontinuous}(2)
respectively).
Hence it is conceivable that when $\lambda=\lambda_w(x)=\lambda_s(x)$
it might happen that the process dies out locally but survives globally (we do not know
of any example though). From now on we agree with many authors by 
defining the phase of \textit{pure global survival} at $x$ when $\lambda_w(x)<\lambda_s(x)$.
A necessary condition for the existence of a pure global survival phase starting from $x$
is clearly that $K_s(x,x)<K_w(x)$; indeed,
according to Corollary~\ref{cor:globalsurvivalcontinuous}(1), if $K_s(x,x)=K_w(x)$
then there is no pure global survival
starting from $x \in X$ since
$1/K_w(x) \le \lambda_w(x) \le \lambda_s(x) =1/K_s(x,x)$.
In some cases 
this condition is also sufficient (see for instance Corollary~\ref{cor:globalsurvivalcontinuous}(2)
and Theorem~\ref{th:Tnk}).

We note that for an irreducible BRW, given $A \subseteq X$, $q(x,A)<1$ for some $x \in X$
if and only if
 $q(y,A)<1$ for all $y \in X$; thus, the existence of pure global survival does not depend
on the starting vertex.
It has been observed that the existence of a pure global survival phase is in many cases associated
with nonamenability. We start with
a characterization of nonamenability for irreducible,
non-oriented discrete-time BRWs.

Recall that for an irreducible BRW $M_s(x,y)=M_s$ 
and
$M_w(x)=M_w 
$ for all $x, y \in X$.
Analogously $\lambda_w(x)=\lambda_w$ and $\lambda_s(x)=\lambda_s$ for all $x \in X$ in the case of an irreducible
continuous-time BRW.

\begin{teo}\label{th:nonam}
 Let $(X,\mu)$ be an irreducible, non-oriented $\mathcal{F}$-BRW. Then the following
claims are equivalent:
\begin{enumerate}[(1)]
 \item the BRW is nonamenable;
\item $M_s 
<
M_w 
$.
\end{enumerate}
\end{teo}

\noindent A sketch of the proof of this result can be found in Section~\ref{sec:proofs}.
Note that an irreducible discrete-time $\mathcal{F}$-BRW is in a pure global
survival regime (i.e.~there is global survival and local extinction) if and only if
$M_s \le 1 < M_w$. Hence, the equivalent conditions in Theorem~\ref{th:nonam} are
necessary but not sufficient for global survival.

On the other hand, in the irreducible continuous-time case, the existence of a phase of
pure global survival is equivalent to $\lambda_w < \lambda_s$.
According to  Corollaries~\ref{cor:pemantleimproved} and \ref{cor:globalsurvivalcontinuous}, for
an $\mathcal{F}$-BRW,
$\lambda_s=1/K_s=\lambda/M_s$ and 
$\lambda_w=1/K_w=\lambda/M_w$, 
thus
 we have the following corollary.

\begin{cor}\textbf{\cite[Theorem 3.6]{cf:BZ}}\label{cor:nonam}
Let $(X,E(X))$ be an irreducible and non-oriented, continuous-time $\cF$-BRW.
Then $\lambda_w<\lambda_s$ if
and only if $(X,E(X))$ is nonamenable.
\end{cor}

We observe that in the irreducible case in the pure global survival phase, the colony
survives globally by clearing all finite subsets $A\subseteq X$. We might think of it
as a sort of drifting towards some sort of boundary of the graph (we do not want to
give more details on this).

\begin{exmp}\label{exmp:pureglobalsurvival}
The classical examples of amenable and nonamenable graphs are $\Z^d$ (for all $d \in \N$)
and $\mathbb{T}_d$ (with $d \ge 3$) respectively. These properties
can be verified by computing explicitly $\lambda_s$ and $\lambda_w$ for the edge-breeding
continuous-time BRWs on these graphs.

Let us consider the Euclidean lattice $X=\Z^d$ and the edge-breeding BRW on $X$.
In this case Corollaries~\ref{cor:pemantleimproved} and \ref{cor:globalsurvivalcontinuous}
apply and $\lambda_s=\lambda_w=1/2d$ since $\bar \rho=2 \lambda d$.
Indeed $\Z^d$ is an amenable graph.

If we consider an edge-breeding BRW on the homogeneous tree $\mathbb{T}_d$, where
each edge has $d \ge 3$ neighbors, the situation is different. 
Easy computations show that $\lambda_w=1/d$ (since the graph is regular) and
$\lambda_s=1/2\sqrt{d-1}$. Indeed, observe that, in this case, 
$\lambda_s=r/d$
where
$r=\max\{t \in \mathbb{R} : F(x,x|t) \le 1\}$ (as explained in Section~\ref{subsec:locisomBP}
after Proposition~\ref{pro:F-BRW1}). In this case it is easy to find the explicit
expression of the function $F(x,x|t)= (d-\sqrt{d^2-4(d-1)t^2})/(2(d-1)t)$
(see, for instance, the proof of \cite[Lemma 1.24]{cf:Woess}), whence
$r=d/2\sqrt{d-1}$. We have verified by direct computation that $\lambda_w < \lambda_s$.
\end{exmp}

Pure global survival is a fragile property of a BRW. Finite modifications,
such as for an edge-breeding BRW attaching a complete finite graph to a vertex
or removing a set of
vertices and/or edges,
 can create it or destroy it as we show in the following remark.

\begin{rem}\textbf{\cite[Remark 3.2]{cf:BZ}}\label{rem:finitegraph}
For simplicity, in this remark we consider only edge-breeding BRWs
on multigraphs, that is, continuous-time BRWs on $X$ where $k_{xy}$ is the number of edges
from $x$ to $Y$ for all $x,y \in X$.
Note that if $(Y,E(Y))$ is a submultigraph of $(X,E(X))$ then
$\lambda_w^X \le \lambda_w^Y$,
$\lambda_s^X \le \lambda_s^Y$,
$K_w^X \ge K_w^Y$, $K_s^X \ge K_s^Y$.

Suppose that there exists a finite subset $S \subseteq X$
such that $X \setminus S$ is the disjoint union of
a finite number of connected multigraphs $X_1, \ldots, X_n$
then, the existence of a pure global survival phase on
$X$, implies the existence of a pure global survival phase on some $X_i$.
Indeed
for every
$\lambda \in (\lambda_w(X),\lambda_s(X))$ the $\lambda$-BRW leaves
eventually a.s.~the subset $S$.
Hence it survives (globally but not locally) at least on one connected component; this means that,
although
$\lambda_s(X_i) \ge \lambda_s(X)$, $\lambda_w(X_i) \ge \lambda_w(X)$ for all $i=1, \ldots,n$,
there exists $i_0$ such that $\lambda_w(X_{i_0}) = \lambda_w(X)$. The existence of a pure
global survival phase on $X_{i_0}$ follows from
$\lambda_s(X_{i_0}) \geq \lambda_s(X)>\lambda_w(X)=\lambda_w(X_{i_0})$.

Moreover if there exists a subset $S$ as above such that $\lambda_w(X_i) > \lambda_w(X)$ for all $i$,
then there is no pure global survival for the BRW on $X$.
Take for instance
a graph $(X^\prime,E(X^\prime))$ and $k \in \N$ such that
$1/k < \lambda_w(X^\prime)$. Attach a \textsl{complete graph} of degree $k$ to a vertex of
$X^\prime$ by an edge (a complete graph is a finite set where every couple $(x,y)$ is an edge),
we obtain a new graph $X$ such that $\lambda_s(X)=\lambda_w(X)
\leq 1/k < \lambda_w(X^\prime)$; hence even if the BRW on $X^\prime$ has a pure global survival phase,
the BRW on $X$ has none.
Nevertheless, using the same arguments as in Example~\ref{rem:nonstrongandstrong}, if the original BRW has
a pure global survival phase, the new one has non-strong local survival.
\end{rem}

There are examples of irreducible amenable BRWs with pure global survival (see Example~\ref{exm:amenable}) and of
irreducible nonamenable BRWs with no
pure global survival (see Example~\ref{exm:nonamenable} which makes use of Remark~\ref{rem:finitegraph}).
Recall that, for an edge-breeding BRW on a graph (or a multigraph), nonamenability is equivalent
to the usual nonamenability of the graph.
\begin{exmp}\label{exm:amenable}
 Consider an irreducible, edge-breeding continuous-time BRW on the (non-oriented) graph $X$ obtained by attaching to a copy of
 $\N$ one branch $T$ of the homogeneous tree $\mathbb T_3$ (see Figure~\ref{fig:amenable}).
 The BRW is amenable by the presence of $\N$.
We claim that $\lambda_s^X=\lambda_s^{\mathbb T_3}$ and $\lambda_w^X = \lambda_w^{\mathbb T_3}$.
Indeed $T\subset X\subset \mathbb T_3$, hence
 $\lambda_s^T\ge \lambda_s^X\ge\lambda_s^{\mathbb T_3}$ and  $\lambda_w^T\ge
 \lambda_w^X \ge \lambda_w^{\mathbb T_3}$.  But by approximation, $\lambda_s^T=\lambda_s^{\mathbb T_3}$.
 Indeed $\lambda_s^T\ge \lambda_s^{\mathbb T_3}$ and does not depend on the starting vertex; moreover
 $T$ contains arbitrarily large balls isomorphic to balls of $\mathbb T_3$, hence by Theorem~\ref{th:spatial}
 their critical local parameters coincide.
 Note that by Remark~\ref{rem:finitegraph} since $\T$ is a disjoint union of three copies of $T$,
then $\lambda_w^T=\lambda_w^{\mathbb T_3}$.
 Using Example~\ref{exmp:pureglobalsurvival} we have
 $\lambda_w^X = \lambda_w^T \le \lambda_w^{\mathbb T_3}$

It is worth mentioning an alternative proof of $\lambda_w^T=\lambda_w^{\mathbb T_3}$ which makes use
of Theorem~\ref{th:equiv1global2}(4). Indeed $\lambda_w^{\mathbb T_3}=1/3$ and 
it is clear that
the function $v \in l^\infty_+(\mathbb T_3)$ defined by
\[
 v(x):=
\begin{cases}
 1 & x=o \\
 5 - 2^{2-n} & d(x,o)=n\\
\end{cases}
\]
(where $o$ is the root of $T$ and $d(x,o)$ denotes the natural distance on the graph between $x$ and $o$)
is a solution of $\frac{1}{3} K v =v$; this implies $\lambda_w^{T}\le 1/3$.
\end{exmp}

\begin{exmp}\label{exm:nonamenable}
Consider a nonamenable graph $X^\prime$ such that the corresponding edge-breeding continuous-time
BRW has a pure global survival phase (take for instance $X^\prime:=\mathbb{T}_3$ the homogeneous tree
with degree $3$). Following Remark~\ref{rem:finitegraph}, attach
to a vertex of $X^\prime$ a complete graph with degree $k>1/\lambda_w^{X^\prime}$ by an edge
(see Figure~\ref{fig:t3complete}).
It is easy to show that the resulting
graph $X$ is still nonamenable, nevertheless, according to Remark~\ref{rem:finitegraph}, there
is no pure global survival for the corresponding edge-breeding BRW.
Roughly speaking, since $\lambda_w^X \le 1/k<1/3$, then for every
$\lambda \in (\lambda_w^X,1/3)$ the process cannot survive globally in
$X^\prime:=\mathbb{T}_3$ hence it hits infinitely often with positive probability the complete
graph, hence $\lambda_s^X=\lambda_w^X$.
%
\end{exmp}

The following result gives a useful sufficient condition for the absence of pure global
survival for a continuous-time BRW $(X,K)$
which is based only on the geometry of the graph $(X,E_K)$ generated by the BRW
(where $E_K:=\{(x,y) \in X^2: k_{xy}>0\}$). It is a slight generalization of \cite[Proposition 2.1]{cf:BZ}.

\begin{teo}\label{th:subexponentialgrowth}
 Let $(X,K)$ be a continuous-time non-oriented BRW and let $x_0 \in X$.
Suppose that there exists $\kappa \in (0,1]^X$ and $\{c_n\}_{n \in \N}$  such that, for all $n \in \N$
\[
\begin{cases}
 \kappa(y)/\kappa(x_0)\le c_n & \forall y \in B(x_0,n) \\
\kappa(x)k_{xy}=\kappa(y)k_{yx} & \forall x,y \in X, \\
\end{cases}
\]
where $B(x,n)$ is the ball of center $x$ and radius $n$
w.r.~to the natural distance of the graph $(X,E_K)$.
If $\sqrt[n]{c_n} \to 1$ and $\sqrt[n]{|B(x_0,n)|} \to 1$ as $n \to 1$
then $K_s(x_0,x_0)=K_w(x_0)$ and there is no pure global survival starting from $x_0$.
\end{teo}

The condition $\sqrt[n]{|B(x,n)|} \to 1$ as $n \to 1$ is usually called \textit{subexponential growth}.
The previous result applies, for instance to BRWs  on
$\Z^d$  or $d$-dimensional combs (see \cite{cf:BZ03} for the definition).
This result extends easily to discrete-time non-oriented BRWs using
$M_w(x)$ 
and $M_s(x,x)$ 
instead of $K_w(x)$
and $K_s(x,x)$ respectively.

\begin{rem}\label{rem:exmp0}
We can apply the previous arguments to the family of rooted trees in Example~\ref{exmp:0}.
$X$ is nonamenable if and only if $T_1$ is nonamenable, that is, if and only if there exists $i$ such that
$m_i\ge 2$. In this case, according to Theorem \ref{th:nonam}, $\lambda_w^{X}< \lambda_s^X$, hence by Remark~\ref{rem:finitegraph}
(considering $X \setminus \widetilde Y$) there exists
$i$ such that $\lambda_w^{T_i}< \lambda_s^{T_i}$.
This means that for all $i= 1, \ldots, b$ we have $\lambda_w^{T_i}<  \lambda_s^{T_i}$ and there is pure
global survival on $T_i$. On the other hand, if $m_i\equiv 1$ for all $i=1,\ldots,b$ the graph has
subexponential growth, then there is no
pure global survival.
\end{rem}

Finally we construct an example of
a continuous-time BRW, where if $\lambda$ is small enough or large enough there is
strong local survival but in a intermediate interval for $\lambda$ there is global and local
survival with different probabilities. This is obtained by modifying the edge-breeding
BRW on a particular graph, namely the homogeneous
tree $\mathbb{T}_d$. The crucial property that we need here is the existence of a pure global survival
phase, thus the procedure applies to every BRW with such a phase.

\begin{exmp}\label{rem:nonstrongandstrong}
Consider the edge-breeding continuous-time BRW on the homogeneous
tree $\mathbb{T}_d$ with degree $d\ge 3$. We know from Example~\ref{exmp:pureglobalsurvival} that
if $\lambda \le 1/d$ the probabilities of survival are $0$,
if $\lambda > 1/2\sqrt{d-1}$ there is strong local survival (according to Proposition~\ref{pro:qtransitive})
and
if $\lambda \in (1/d ,1/2 \sqrt{d-1}]$ the probability
of global survival is positive and
independent from the starting point and the probability of local survival at any finite $A \subseteq X$
is $0$.

Fix $\lambda \in (1/d, 1/2 \sqrt{d-1}]$. According to Remark~\ref{rem:strongconditioned},
there exists  $x \in X$  such that
there is a positive probability of global survival starting from $x$
without visiting $A$. In this case, any modification of the rates in the subset $A$ provides a new BRW such that 
there is still a positive probability of global survival starting from $x$
without ever visiting $A$ (since, the original BRW and the new one coincide until the first hitting time on $A$).
On the other hand, if there is $y \in A$ such that $x \to y$ and we add a loop in $y$
and a rate $k_{yy}>1/\lambda$ then $\bar q(x)<q(x,y)<1$;
the first inequality holds by the discussion above on local modifications
and the second one holds since $\lambda k_{yy}>1$ implies local survival at $y$
(then irreducibility implies local survival at $y$ starting from $x$).
This means that, for this fixed value of $\lambda$, we obtained a locally and globally
(but not strong-locally) surviving BRW at $y$ starting from $x$.

Suppose now that $k_{yy}>d$;
then (see Remark~\ref{rem:finitegraph})
we have a new BRW such that $\lambda_w^\prime=\lambda_s^\prime \le 1/k_{yy}$. In this case, when $\lambda \le \lambda_w^\prime$
there is global extinction.
When $\lambda > 1/2\sqrt{d-1}$ there is strong local survival for the original
BRW (by Proposition~\ref{pro:qtransitive}) which implies strong local survival for the new one
(the probability of hitting $x$ conditioned on global survival is $1$ for both
processes and Remark~\ref{rem:strongconditioned} applies).
If $\lambda \in (\lambda_w^\prime, 1/d]$ there is local and global survival
with the same probability since in order to survive globally,
the process must visit $x$ infinitely many times (it cannot survive globally in the branches of $\mathbb{T}_d$).
If $\lambda \in (1/d, 1/2\sqrt{d-1}]$ then, according to the previous discussion,
there is non-strong local survival for the new BRW.
\end{exmp}

We show that even in the irreducible case, if $\rho_x(0)=0$ for some $x \in X$, we might have
strong local survival starting from some vertices and not from others.

\begin{exmp}\label{exm:irreduciblenotstrongeverywhere}
 Let us consider a modification of the discrete-time counterpart of the edge-breeding BRW on
$\mathbb{T}_d$ with degree $d\ge 3$ and $\lambda \in (1/d, 1/2\sqrt{d-1}]$ . Let us fix
a vertex $y$; in this modified version we add, with probability one,
one child at $y$ for every particle at $y$. In this case $\bar q(y)=q(y,A)=0$ for all $A \subseteq X$.
On the other hand according to the discussion in Example~\ref{rem:nonstrongandstrong},
there is a vertex $y$ such that $\bar q(x)<q(x,y)$.
\end{exmp}

\subsection{An application: BRW locally isomorphic to a branching process}\label{subsec:locisomBP}

This class of BRWs is very easy to study and it gives an immediate
connection with the theory of random walk.
Recall that a BRW $(X,\mu)$ is \textit{locally isomorphic to a branching process}
if and only if  the laws of the offspring number $\rho_x=\rho$ is independent
of $x \in X$ (see Definition~\ref{def:locallyisomorphic}).
In this case the branching process can be constructed as the BRW
$(\{0\}, \nu)$, where $\nu_0:=\rho$.
All the results of this section also apply to continuous-time BRWs
where $\rho_x$ is independent of $x \in X$, since 
their discrete-time
counterpart is  locally isomorphic to a branching process. In particular,
for a continuous-time BRW $(X,K)$, $\rho_x$ is uniquely determined by $k(x) \equiv \sum_{y \in X} k_{xy}$
(and by $\lambda$),
hence $(X;K)$ is locally isomorphic to a branching process if and only if $k(x)$ does not
depend on $x \in X$.

The following result characterizes local and global survival for this class.
Remember the definition of the diffusion matrix given in Section~\ref{subsec:discrete}
as $p(x,y):=m_{xy}/\bar \rho_x$; note that, for a BRW with independent diffusion, 
the diffusion matrix coincides with the transition probability matrix $P$.
In this proposition we denote by
$F$ the generating function of the hitting probabilities,
(cfr.~the function $U$ of \cite[Section 1.C]{cf:Woess09}).

\begin{pro}
 \label{pro:F-BRW1}
Let the BRW be locally isomorphic to a branching process 
and denote
by $\rho$ the common offspring law.
Then
\begin{enumerate}[(1)]
 \item there is global survival if and only if $\bar \rho>1$;
 \item there is local survival at $x$
if and only if $\bar \rho > 1/\limsup_{n \to \infty} \sqrt[n]{p^{(n)}(x,x)}$;
 \item there is local survival at $x$
if and only if either $F(x,x|\bar \rho) >1$ or $F(x,y| \bar \rho)$ diverges.
\end{enumerate}
\end{pro}

A sketch of the proof can be found in Section~\ref{sec:proofs}.
We note that there is local survival at $x$ if and only if
\[
 \bar \rho > \max\{t \in \mathbb{R}: F(x,x| t) \le 1\} \equiv r(x,x)
\]
where $r(x,x)=1/\limsup_{n \to \infty} \sqrt[n]{p^{(n)}(x,x)}$ is the spectral radius 
of the random walk $P$ 
(see \cite[Section 2.C]{cf:Woess09}).
As a corollary one derives the critical parameters
for continuous-time BRWs which are locally isomorphic to a branching process:
$\lambda_w=1/k$ and $\lambda_s(x)=r(x,x)/k$ (where $k=k(x)$ for all $x \in X$).

It is clear that, in the irreducible case, there is pure global survival (see Section~\ref{subsec:pureweak}) if and only if
$1< \bar \rho \le r$ (where $r=r(x,x)$ in this case does not depend on $x \in X$ due to the irreducibility).
This is possible if and only if $r>1$ which is equivalent to nonamenability (see Theorem~\ref{th:nonam})
since in this case
$
M_s(x,y)= \bar \rho / r$ and
$M_w(x) 
=\bar \rho$.
Notice that if there is pure global survival then $P$ defines a transient random walk but the converse is not true:
if $P$ is the simple random walk on $\Z^d$ there is no pure global survival for any $\bar \rho$ and $d$.

In general there is no strong local survival, even if the BRW has independent diffusion 
as Examples~\ref{ex:nonstronglocalFBRW} and \ref{ex:nonstronglocalFBRW2} show.
Before discussing the examples we need an easy lemma, whose proof can be found in Section~\ref{sec:proofs}.

\begin{lem}
\label{lem:test0}
 Let $\{\alpha_i\}_{i \in \N}$ and $\{k_i\}_{i\in \N}$ be such that $\alpha_i \in (-\infty, 1)$ and
$k_i \ge 0$ for all $i \in \N$. Then
\[
        \sum_{i \in \N} k_i \alpha_i < +\infty \Longleftarrow \prod_{i \in \N} (1-\alpha_i)^{k_i}>0.
       \]
Moreover if $\alpha_i \in [0, 1)$ and $k_i \ge 1$ eventually as $i \to \infty$ then
\[
        \sum_{i \in \N} k_i \alpha_i < +\infty \Longleftrightarrow \prod_{i \in \N} (1-\alpha_i)^{k_i}>0.
       \]
\end{lem}

\begin{exmp}\label{ex:nonstronglocalFBRW}
Fix $X:=\N$ and consider a BRW with the following reproduction probabilities.
Every particle has two children with probability $3/4$ and no children with probability $1/4$. Each newborn
particle is dispersed independently according to a nearest neighbor matrix $P$ on $\N$.
More precisely
\[
 p(i,j):=
\begin{cases}
 p_i & \textrm{ if } j=i+1\\
1-p_i  & \textrm{ if } j=i-1,\\
\end{cases}
\]
and $p_0=1$.
The process described above is
an irreducible $\mathcal F$-BRW for every choice of the set $\{p_i\}_{i \in \N \setminus \{0\}}$.

The generating function of the total number of children is $z \mapsto 3 z^2/4+1/4$ and its minimal
fixed point is $1/3=\bar q(x)$ (for all $x \in \N$).

Choose $p_1 < 5/9$; it is easy to show that the process confined to $\{0,1\}$ survives
(since the expected number of children at $0$ every two generations (starting from $0$)
is $ (3/2)^2 (1-p_1)>1$). By irreducibility this implies that $q(x,y)<1$ and $\bar q(x)<1$ for all $x,y \in \N$.

Choose the $p_i$s such that $\prod_{i=1}^\infty p_i^{2^i}>0$,
which, according to Lemma~\ref{lem:test0}, is equivalent to $\sum_{i=1}^\infty 2^i (1-p_i)< +\infty$.
Consider the branching process $N_n$ representing the total number of particles alive at time $n$:
for all $n$,  $N_n \le 2^n$ almost surely.
The probability, conditioned on global survival, that every particle places its children (if any) to its right,
is the conditioned expected value of $\prod_{i=1}^\infty p_i^{N_i}$.
But
$\prod_{i=1}^\infty p_i^{N_i} \ge \prod_{i=1}^\infty p_i^{2^i}>0$ almost surely.
 Hence, conditioning on global survival
there is a positive probability of non-local survival. This implies $q(\cdot, y) \not = \bar q$
for every $y \in \N$. Note that, according to Theorem~\ref{th:Fgraph}, $\sup_{x \in \N} q(x,x)=1$.
\end{exmp}

The key in the previous example is that the total number of particles alive at time $n$ is bounded.
This is not an essential assumption. The following example shows that, given any law $\rho$ of a
surviving
branching process (that is, $\bar \rho=\sum_{n\ \in \N} \rho(n)>1$), it is possible to construct
an irreducible BRW which is locally isomorphic to a branching process with no strong local survival.

\begin{exmp}\label{ex:nonstronglocalFBRW2}
Let $X=\N$ and $\rho_x:=\rho$ for all $x \in \N$; $\rho$ being the law of a surviving branching process.
We know that $\bar q(x) \equiv \bar q$ for all $x \in \N$
where $\bar q<1$ is the smallest fixed point of $z \mapsto \sum_{n\ in \N} \rho(n) z^n$.
Pick a sequence of natural numbers
$\{N_i\}_{i \in \N}$ satisfying
\begin{equation} \label{eq:defineNi}
\prod_{i \in \N} \rho([0,N_{i+1}])^{\prod_{j=0}^i N_j} >\bar q,
\end{equation}
where $N_0:=1$.
Note that the probability of the event $\mathcal{A}$ where every particle alive at time $i$ has at most $N_{i+1}$ children
for all $i \in \N$ is bounded from below by the LHS of equation~\eqref{eq:defineNi}.
Thus, from equation~\eqref{eq:defineNi}, with a probability larger than
$\prod_{i \in \N} \rho([0,N_{i+1}])^{\prod_{j=0}^i N_j} -\bar q>0$ the colony survives globally and
the total size of the population at time $n$ is not larger than $\prod_{j=0}^n N_j$
(i.e.~the intersection between $\mathcal{A}$ and global survival has positive probability).

We define a BRW with independent diffusion 
where $P$ is as follows
\[
 p(i,j)
:=
\begin{cases}
 p_i &  j=i+1,\, i \ge 0\\
1-p_i & j=i-1, \, i \ge 1\\
1-p_0 & i=j=0.\\
\end{cases}
\]
Let $p_0$ such that $(1-p_0)\bar \rho>1$; this implies local survival.
We choose the sequence $\{p_i\}_{i \in \N}$, where $p_i \in (0,1)$ in such a way that
\begin{equation} \label{eq:definepi}
\prod_{i \in \N} p_i^{\prod_{j=0}^i N_j} >0.
\end{equation}
Using equation~\eqref{eq:definepi}, if we condition on $\mathcal A$,
the probability that, every particle places its children (if any) to its right is bounded from
below by $\prod_{i \in \N} p_i^{\prod_{j=0}^i N_j}$.
This implies that there is a positive probability of global, non-local survival.

The choice of the sequences $\{N_i\}_{i \in \N}$ and $\{p_i\}_{i \in \N}$ satisfying equations~\eqref{eq:defineNi} and
\eqref{eq:definepi} respectively can be done as follows.
Choose a sequence $\{\alpha_i\}_{i \in \N}$ such that $\alpha_i \in (0,1)$ for all $i \in \N$ and
$\prod_{i \in \N} \alpha_i >1-\bar q$. Then, iteratively, if we fixed $N_0, \ldots, N_k$, since $\lim_{x \to \infty}\rho([0,x])=1$
there exists $N_{k+1} \in \N$ such that $\rho([0,N_{k+1}])>\alpha_{k+1}^{1/\prod_{j=0}^k N_j}$.
Moreover, according to Lemma~\ref{lem:test0}, equation~\eqref{eq:definepi}, is equivalent to
$\sum_{i \in \N} (1-p_i){\prod_{j=0}^i N_j}< \infty$,
hence let us take, for instance, $p_i > 1/(i \cdot {\prod_{j=0}^i N_j})$.

We note that the class constructed in this example includes
discrete-time counterparts of continuous-time BRWs where
$\rho$ can be chosen as in equation~\eqref{eq:counterpart} where $k(x)$ does not depend on $x$,
$k_{xy}:=k(x)p(x,y)$ (where $P$ is defined as before) and $\lambda>\lambda_s$ is fixed.
Finally we observe that this example extends naturally to an example
of a site-breeding BRW on a radial tree where the number of branches of a vertex
at distance $k$ from the root is at least $1/p(k,k+1)$.
\end{exmp}

The following easy theorem gives another sufficient condition for the strong local survival
of a BRW which is locally isomorphic to a branching process. The proof can be found in Section~\ref{sec:proofs}.

\begin{teo}\label{th:recurrent}
Suppose that $(X,\mu)$ is an irreducible BRW with independent diffusion
such that
$P$ is the transition matrix of a recurrent random walk
and $\rho_x=\rho$ for all $x \in X$. 
Then, global survival starting from some $x \in X$ implies strong local survival at $y$ starting from $w$
for  all $w,y \in X$.
\end{teo}

In case of a BRW with no death and with independent diffusion 
one can prove the following proposition which makes use of Proposition~\ref{pro:Mueller08}.
By $U(x,y|z)$ we mean the usual generating function of the first-return probabilities of the
random walk $P$ as defined in \cite[equation (1.26)]{cf:Woess09}; in particular
$U(x,y)$ is the probability of visiting $y$ \textit{after} starting from $x$.

\begin{pro}\textbf{\cite[Proposition 3.6]{cf:Muller08-2}}\label{cor:Mueller08}
 Let $(X,\mu)$ be an irreducible BRW with independent diffusion 
where $\rho_x=\rho$ for all $x \in X$ and $\rho(0)=0$.
If $\bar \rho > \sup_{x \in X} 1/U(x,x)$ then there is strong local survival.
\end{pro}

A result like this one could be proved without the no-death hypothesis
using the comparison described in
Section~\ref{subsec:nodeath}; in this case a reasonable hypothesis should be
$\bar \rho > \sup_{x \in X} 1/U(x,x|1-\rho(0))$.

\begin{rem}
 \label{rem:dominatingBP}
In order to extend some results, as Proposition~\ref{cor:Mueller08} for instance, from
the homogeneous case ($\rho_x=\rho$ for all $x \in X$) to the inhomogeneous case,
we have to find sufficient conditions such that the infimum of the probabilities of
survival of the branching processes with laws $\{\rho_x\}_{x \in X}$ is strictly larger than 0.
Observe that if $\rho \succeq \widehat \rho$ then for all nondecreasing, positive,
measurable function $f$ we have $\sum_{n \in \N} f(n) \rho(n) \ge
\sum_{n \in \N} f(n) \widehat \rho(n)$ which, in turn,
implies $\sum_{n \in \N} z^n \rho(n)
\le \sum_{n \in \N} z^n \widehat \rho(n)$ for all $z \in [0,1]$.
Hence if $\sum_{n \in \N} z^n \widehat \rho(n) \le z$ then $\sum_{n \in \N} z^n \rho(n)\le z$.
Thus if  $\rho_x \succeq \widehat \rho$ for all $x \in X$
then the probabilities of extinction of the branching processes associated to
the $\rho_x$s are all dominated by the probability
of extinction of the branching process associated to $\widehat \rho$.
Hence one possibility would be to assume that for all $x \in X$,
$\widehat \rho_x$ dominates some law of a surviving branching process $\widehat \rho$. This way the results
for inhomogeneous BRWs are simply corollaries of the homogeneous case. Clearly this is not a
significant improvement.

One might guess that if the expected number of offsprings $\bar \rho_x$ is sufficiently large then
 the supremum of the probabilities of
extinction of the branching processes with laws $\{\rho_x\}_{x \in X}$ is strictly smaller than 1.
But clearly, if $\sup_{x \in X} \rho_x(0)=1$ then the supremum of the probabilities of
extinction is $1$. Even bounding $\rho_x(0), \ldots, \rho_x(k)$ (for all $x \in X$) is not sufficient.
Indeed consider the set $\mathcal A$ of all probability generating functions
and a subset $\mathcal A_{a_0,a_1; \ldots, a_k,m}$ defined as
\[
 \begin{split}
  \mathcal A&:=\{\phi(z)=\sum_i \rho(i) z^i: \rho(i) \ge 0 \ \forall i, \ \sum_i \rho(i)=1\}\\
\mathcal A_{a_0,a_1, \ldots, a_k,m}&:=\{\phi\in \mathcal A:\phi(z)=\sum_i \rho(i) z^i, \  \rho(i) \le a_i \
\forall i \le k, \ \bar \rho  
\ge m\}\\
 \end{split}
\]
(where $a_i \in [0, 1]$ and $m \in \mathbb R$, $m>1$).
In this case, either there is a surviving branching process
(which does not necessarily belong to $\mathcal A_{a_0,a_1, \ldots, a_k,m}$)
whose law is stochastically dominated by all
$\rho \in \mathcal A_{a_0,a_1, \ldots, a_k,m}$
or the supremum of the probabilities of extinction is $1$.


In order to prove this claim,
let $k_0:= \max \{i=0, \ldots, k: \sum_{j=0}^k a_i < 1\}$
and 
define $\widehat \phi(z):= \sum_{i=0}^{k_0} a_i z^i+ (1-\sum_{i=0}^{k_0} a_i)z^{k_0+1}$.
Clearly the branching process corresponding to $\widehat \phi$ (which might not belong
to $\mathcal A_{a_0,a_1, \ldots, a_k,m}$) is dominated
by every branching process in $\mathcal A_{a_0,a_1, \ldots, a_k,m}$
and $\widehat \phi \le \phi$ in $[0,1]$ for every $\phi \in \mathcal A_{a_0,a_1, \ldots, a_k,m}$.
This implies that the minimal fixed point, $c \le 1$, of $\widehat \phi$ is an upper bound
of the minimal fixed point of
$\phi \in \mathcal A_{a_0,a_1, \ldots, a_k,m}$.
If $c<1$, that is, $\sum_{i=0}^{k_0} ia_i + (1-\sum_{i=0}^{k_0} a_i)(k_0+1)>1$
(i.e.~the branching process survives) there is nothing to prove.
On the other hand, if $c=1$ then consider
\[
\phi_N=\sum_{i=0}^{k_0} a_i z^i+ w_N z^{k_0+1} +(1-w_N-\sum_{i=0}^{k_0} a_i)z^N
\]
where
$w_N=(\sum_{i=0}^{k_0} i a_i - m +(1-\sum_{i=0}^{k_0} a_i)N)/(N-(k_0+1)) \in [0, 1-\sum_{i=0}^{k_0} a_i]$ eventually
as $N \to \infty$
(remember that $c=1$ if and only if $1 \ge \frac{\diff}{\diff z} {\widehat \phi}(z)|_{z=1}
\equiv \sum_{i=0}^{k_0} ia_i + (1-\sum_{i=0}^{k_0} a_i)(k_0+1)$).
Then $\frac{\diff}{\diff z} \phi_N(z)|_{z=1}=m$, $\phi_N \in \mathcal A_{a_0,a_1, \ldots, a_k,m}$
and $\lim_{N \to \infty} \phi_N \equiv \widehat \phi$.
It is straightforward to show that the minimal fixed point $c_N$ of $\phi_N$ converges
to $c=1$ as $N \to \infty$.


\end{rem}

\section{Approximation}
\label{sec:approx}

\subsection{Spatial approximation}
\label{subsec:spatial}

The first kind of approximation is based upon a result on approximation of nonnegative matrices
which is interesting in itself.
Recall the usual classification of indices of a matrix $M=(m_{xy})_{x,y \in X}$ (which is supposed to be
nonnegative throughout this section)
as described in \cite[Chapter 1]{cf:Sen}.
For any index $x$ we denote by $[x]$ its \textit{class}, that is, the set of indices
which communicate with $x$.
We define the convergence parameters $R(x,y):=M_s(x,y)
^{-1}$
and $R:= \inf_{x,y \in X} R(x,y)$.
It is straightforward to show that $M_s(x,y)=M_s(x_1,y_1)$
if
$[x]=[x_1]$ and $[y]=[y_1]$; 
this implies that for irreducible matrices, $R(x,y)$ is independent of $x,y \in X$.

Let $\{X_n\}_{n \in \N}$  be a 
sequence of subsets of $X$ 
and denote by $_nR$ the convergence parameter of  $M_n=(m_{xy})_{x,y \in X_n}$;
clearly, if the sequence $\{X_n\}_{n \in \N}$ is nondecreasing, we have that $_nR \geq {_{n+1}R}$.
The following theorem  generalizes \cite[Theorem 6.8]{cf:Sen}
(note that the submatrices $\{M_n\}_{n \in \N}$ are not necessarily irreducible); it
is the key to prove our main
result about spatial approximation (Theorem~\ref{th:spatial}).

\begin{teo}\textbf{\cite[Theorem 5.1]{cf:Z1}}\footnote{We observe that in \cite[Section 5.1]{cf:Z1}
the hypotheses that $M$ is a nonnegative matrix is missing, even though it is implicitly used.}\label{th:genseneta}
 Let $\{X_n\}_{n \in \N}$  be a general sequence of subsets of $X$ such that
$\liminf_{n \to \infty} X_n =X$ and suppose that $M=(m_{xy})_{x,y \in X}$ is a nonnegative matrix.
Then for all $x_0 \in X$ we have $_nR(x_0,x_0) \rightarrow R(x_0,x_0)$.
Moreover if $M$ is irreducible and $M_n=(m_{xy})_{x,y \in X_n}$ then $_nR \rightarrow R$ as $n \to \infty$ and, in particular,
for all $x_0 \in X$ we have $_nR(x_0,x_0) \rightarrow R$.
\end{teo}

Note that in the previous theorem the subsets $\{X_n\}_{n \in \N}$ can be chosen
arbitrarily; in particular they may be finite proper subsets. Moreover
the result extends easily to the case of a sequence of nonnegative matrices
$M_n=( {m(n)}_{xy})_{x,y \in X_n}$ where $\liminf_{n \to \infty} X_n =X$,
$0 \le {m(n)}_{xy} \le m_{xy}$ for all $x,y \in X_n$ and
$\lim_{n \to \infty} {m(n)}_{xy} =m_{xy}$ for all $x,y \in X$
(note that $m(n)_{xy}$ is eventually well defined for all $x,y \in X$ as $n \to \infty$).
The idea of the proof is essentially contained in \cite[Theorem 5.2]{cf:Z1}.



Given a sequence of BRWs $\{(X_n,\mu_n)\}_{n \in \N}$ such that $\liminf_{n \to \infty} X_n=X$,
we define $m(n)_{xy}:=\sum_{f \in S_{X_n}} f(y) \mu_{n,x}(f)$ and
the corresponding sequence of matrices $\{M_n\}_{n \in \N}$. 
Note that in the following result we are not assuming that the BRW is irreducible.

\begin{teo}\textbf{\cite[Theorem 5.2]{cf:Z1}}\label{th:spatial}
Let us 
fix a vertex $x_0 \in X$. If $\liminf_{n \to \infty} X_n=X$ and
$m(n)_{xy} \le m_{xy}$ for all $x,y \in X_n$, $n \in \N$ and
 $m(n)_{xy} \to m_{xy}$ as $n \to \infty$ then
\begin{enumerate}[(1)]
\item
$(X,\mu)$ dies out locally (resp.~globally) a.s.~starting from $x_0$ $\Longrightarrow$ $(X_n,\mu_n)$ dies out locally (resp.~globally)
a.s~starting from $x_0$  for all $n \in \N$;
\item
$(X,\mu)$ survives locally starting from $x_0$  $\Longrightarrow$ $(X_n,\mu_n)$ survives locally
starting from $x_0$ eventually as $n \to\infty$.
\end{enumerate}
\end{teo}

\noindent Clearly the discrete-time counterpart of a spatially
confined BRW in continuous-time
is obtained by spatially confining the discrete-time counterpart
of the continuous-time BRW.
Hence, Theorem~\ref{th:spatial} yields an analogous result for BRWs in continuous time.

\begin{cor}\textbf{\cite[Theorem 3.1]{cf:BZ2}}\label{cor:sen2}
Let $(X,K)$ be a continuous-time BRW and
let us consider a sequence of continuous-time BRWs $\{(X_n, {K_n})\}_{n\in \N}$
such that $\limsup_{n \to \infty} X_n = X$.
Let us suppose that
$k_{xy}(n)  \le k_{xy}$ for all $n\in\N$, $x,y\in X_n$
and $k_{xy}(n) \to k_{xy}$ as $n \to \infty$ for all $x,y\in X$.
Then $\lambda_s(X_n, {K_n})\ge \lambda_s(X,K)$ and
$\lambda_s(X_n, {K_n})
\to \lambda_s(X,K)$ as $n\to \infty$.
\end{cor}

Among all possible choices of the sequence
$\{(X_n,\mu_n)\}_{n \in \N}$ there is one which is
\textit{induced} by $(X,\mu)$ on the subsets $\{X_n\}_{n\in \N}$;
more precisely, one can take $\mu_n(g):=\sum_{f \in S_X:f|_{X_n}=g} \mu_x(f)$
for all $x \in X_n$ and $g \in S_{X_n}$.
Roughly speaking, this choice means that all reproductions outside $X_n$ are suppressed.
In this case it is simply
$m(n)_{xy}=m_{xy}$ for all $x,y \in X_n$.

Since
Theorem~\ref{th:spatial} deals with local survival, one can wonder what
can be said about global survival. First of all, if the process $(X,\mu)$ survives 
globally and locally then eventually $(X_n,\mu_n)$ survives locally and thus globally.
The question is nontrivial when $(X,\mu)$ survives globally but not locally, which we assume henceforth in this brief discussion.
In this last case,
if, for instance, $X_n$ is finite for every $n \in \N$ and the graph $(X_n, E_{\mu_n})$ is connected
then, by Theorem~\ref{th:spatial}(1),  $(X_n,\mu_n)$ dies out (locally and globally) a.s.~for all values of $n \in \N$.
On the other hand, the case where $X_n$ is finite for every $n \in \N$ and the graph $(X_n, E_{\mu_n})$ is not connected
is more complicated and can be treated as in \cite[Remark 4.4]{cf:BZ2}.
When $X_n$ is infinite for infinitely many values of $n$,
it is possible that there is no global survival for infinitely many values of $n$.
An example in the discrete-time case can be found in \cite[Remark 5.3]{cf:Z1} while
an example in the continuous-time case can be constructed using \cite[Remark 3.2]{cf:BZ}.
Taking the couple $(X_n,\mu_n)$ random, some results can be achieved as an
application of Theorem~\ref{th:spatial} (see for instance
\cite[Theorem 7.1]{cf:BZ3} or \cite[Theorem 2.4]{cf:GMPV09}).

\begin{exmp}\label{exmp:applicationspatial}
Consider the edge-breeding (continuous-time) BRW on $\Z^d$.
We saw in Example~\ref{exmp:pureglobalsurvival} that if
$\lambda>\lambda_s=1/2d$ then there is local survival.
Suppose that $d>1$ and that $\lambda \le 1/2$ and
consider the infinite cylinder $X_n:=\{x \in \Z^d : |x(i)| \le n, \, \forall i=2, \ldots, d\}$.
It is clear that there is no local survival for the BRW restricted to $X_0=\Z \subseteq \Z^d$,
nevertheless according to Corollary~\ref{cor:sen2} there exists $n_0$ such that
there is local survival on $X_n$ for all $n \ge n_0$.
This shows a difference between random walks and BRWs: the simple random walk on $X_n$ is recurrent for all $n\in \N$
(as the simple random walk on $\Z$); on the other hand, while the BRW restricted to $X_0=\Z$ dies out locally, it
survives when restricted to
$X_n$ if $n$ is sufficiently large (in some sense the BRW on $X_n$ approaches the  BRW on $\Z^d$ as $n$ tends to infinity).

Another consequence of Theorem~\ref{th:spatial} is the following
(see also \cite[Section 5]{cf:Z1}):
consider the edge-breeding continuous-time BRW on $\Z^d$
(but the argument extends easily to any translation invariant BRW in discrete
and continuous time). Let us choose a connected subset
$Y \subset \Z^d$ such that every finite
subset $A \subset \Z^d$ is a subset of a suitable translation of $Y$ in $\Z^d$.
Then the strong critical parameter $\lambda_s^\prime$ of the BRW restricted
to $Y$ is equal to $\lambda_s=1/2d$.
A possible choice is $Y:=\{y \in \Z^d :\langle y, y_0 \rangle \geq \alpha \|y\|\cdot\|y_0\| \}$
for some fixed nontrivial  $y_0 \in \Z^d$
and $\alpha <1$ (where $\langle \cdot, \cdot \rangle$ and $\|\cdot\|$ represent the
usual scalar product and norm of $\Z^d$ respectively).
\end{exmp}

\subsection{Approximation by truncated BRWs}
\label{subsec:truncated}

The family of discrete-time BRWs can be extended to the more general class of \textit{truncated BRWs} where
a maximum of $m \in \N \cup \{\infty\}$ particles per site is allowed. We denote
this process as a BRW$_m$.
The dynamics is described by the following
recursive relation
\[ 
\eta^m_{n+1}(x)=m \wedge \sum_{y \in X} \sum_{i=1}^{\eta^m_n(y)} f_{i,n,y}(x)=
m \wedge \sum_{y \in X} \sum_{j=0}^\infty \ident_{\{\eta^m_n(y)=j\}} \sum_{i=1}^{j} f_{i,n,y}(x).
\] 
Clearly the BRW$_\infty$ is the usual BRW and the BRW$_1$ is the well-known \textit{contact process}.
Note that while for the BRW the reproductions of two particles are independent, this is not true for the BRW$_m$
which is a truly interacting particle system.

There is an analogous class of continuous-time processes that we still call truncated BRWs
which are subject to the same constraint (see \cite{cf:BZ3}).

\begin{rem} \label{rem:restrained}
Even though the BRW$_m$ is the only generalization that we need here,
there is a more general class of continuous-time processes, called
\textsl{restrained} BRWs, which is worth mentioning and
which has been introduced and studied in \cite{cf:BPZ}
(an analogous construction in discrete-time is straightforward and we omit it).
Consider an infinite connected graph $X$ 
and let $\eta(x)$ be the number of individuals living at the site $x\in X$.
The lifespan of each individual is an exponential random variable of mean 1.
During its lifetime each individual tries to reproduce following a Poisson
process of intensity $\lambda$.
Every time the clock associated to the Poisson process rings, the individual tries to send an offspring to a
randomly chosen target neighboring site.
The target neighboring site is chosen using the transition matrix $P=(p(x,y))_{x,y\in X}$
of a nearest neighbor random walk on $X$.
Call the target site $y$.
The reproduction on $y$ is effective only with probability $c(\eta(y))/\lambda$, where $c:\N\to\R^+$ is a
non-increasing and nonnegative function with $c(0)=\lambda$. In this case the population
living at $y$ increases by one individual, otherwise nothing happens.\break \indent
Observe that the restrained BRW is a Markov process and
the continuous-time truncated BRW$_m$ are
special cases ($c\equiv \lambda$ if $m=\infty$ and $c=\lambda \ident_{\{0,1,\ldots, m-1\}})$ otherwise).
Results about survival and stationary measures of this process can be found in \cite{cf:BPZ}; it
is worth noting that this process may have an \textit{ecological equilibrium}, that is,
a phase of local survival where the expected number of individuals per site is bounded from above. This is
not possible for the BRW where local survival implies almost surely an unbounded population
(see the proof of \cite[theorem 4.1(1)]{cf:Z1} and \cite[Theorem 6.2]{cf:Harris63}).
\end{rem}

We observe that the discrete-time counterpart of a
continuous-time truncated BRW is not a discrete-time truncated BRW. Indeed, in order to construct
the discrete-time counterparts we lose the original time scale: on one hand, particles which are in the same generation
in the discrete-time process might have disjoint lifespan intervals in the continuous-time process and,
on the other hand, particles living at the same time in the continuous-time process might belong
to disjoint generations in the discrete-time counterpart.
Hence the results about approximation of continuous-time BRWs by means of continuous-time truncated
BRWs (see \cite{cf:BZ3}) cannot be considered as particular cases of the analogous results
for discrete-time processes (see \cite{cf:Z1}). Nevertheless, the techniques used are very
similar and the results essentially coincides.

The goal of this section is to study the approximation of a BRW $\{\eta_n\}_{n \in \N}$
by means of the sequence of truncated BRWs $\{\{\eta_n^m\}_{n \in \N}\}_{m \in \N}$.
It is clear, by stochastic domination (see \cite[Section 3.3]{cf:Z1}), that
if the BRW dies out locally (resp.~globally)
a.s.~then any truncated BRW dies out locally (resp.~globally).
We are going to prove here a result similar to Theorem~\ref{th:spatial}
as $m$ tends to infinity.
For discrete-time BRWs this has been done in \cite{cf:Z1} while the results
for continuous-time BRWs can be found in
\cite{cf:BZ3}.

From now on we make some assumptions on $(X,\mu)$.
First, we assume that $X$ is countable. Indeed, the finite case is uninteresting
since the the truncated BRW $\{\eta_n^m\}_{n \in \N}$ do not survive
for any $m \in \N$ in discrete and continuous time (by standard Markov chain arguments,
being $\mathbf{0}$ an absorbing state).
Second, we require that the graph $(X,E_{\mu})$ has finite geometry, that is,
$\sup_{x \in X} \mathrm{deg}(x)<+\infty$ and that
the matrix $M$ is irreducible- We denote its
convergence parameter by $R_\mu$.
We observe that, using this notation,
according to Theorem~\ref{th:equiv1local}, local survival
is equivalent to $R_\mu<1$.
For a continuous-time BRW $(X,K)$ we denote the convergence parameter of the matrix
$K$ by $R_K$ and we observe that according to Corollary~\ref{cor:pemantleimproved}, local survival
is equivalent to $1/R_K<\lambda$.

Finally, we suppose that
\begin{equation}\label{eq:suprho}
 \sup_{x \in X} \rho_x([n,+\infty)) \to 0, \qquad \text{as } n \to +\infty.
\end{equation}
This assumption allows to use the measure $\rho$ defined as
\[
\rho(n)=\sup_{x \in X} \rho_x([n,+\infty))-\sup_{x \in X} \rho_x([n+1,+\infty)),
\]
to stochastically dominate all the laws $\{\rho_x\}_{x \in X}$.
Indeed equation~\eqref{eq:suprho} is equivalent to the existence of a measure
which dominates the $\rho_x$s.
The measure $\rho$ has finite second (hence first) moment if and only if
$\int_0^\infty \sup_{x \in X} \rho_x([\sqrt{t},+\infty)) \diff t<+\infty$
(that we assume henceforth).
For a continuous-time BRW $(X,K)$ we simply assume that $\sup_{x \in X} k(x) <+\infty$
and this implies immediately the stochastic domination by means of a continuous-time branching process
with parameter $\sup_{x \in X} k(x)$ (see \cite{cf:BZ3} for details).
\smallskip

We start with the result on the approximation of the local behavior of a BRW by means of
the sequence of truncated BRWs $\{\{\eta_n^m\}_{n \in \N}\}_{m \in \N}$.

\begin{teo}\textbf{\cite[Theorem 6.3]{cf:Z1}} \label{th:mainlocal}$ $\\
Suppose that at least one of the following conditions holds
\begin{enumerate}[(1)]
\item  $(X,\mu)$ is 
 quasi transitive and irreducible;
\item $(X,\mu)$ is irreducible and there exists
$\gamma$
bijection on $X$ such that
\begin{enumerate}[(a)]
\item
$\mu$ is $\gamma$-invariant;
\item for some $x_0\in X$ we have $x_0=\gamma^nx_0$ if and only if $n=0$.
\end{enumerate}
\end{enumerate}
If 
if $\{\eta_n\}_{n \in \N}$ survives locally (starting from $x_0$) then $\{\eta^m_n\}_{n \in \N}$
survives locally (starting from $x_0$) eventually as $m\to +\infty$.
\end{teo}

In the continuous-time case there is an analogous result which
gives an approximation of $\lambda_s$ by $\lambda_s^m$, where $\lambda_s^m$ is
the local critical parameter of the truncated BRW with (at most) $m$ particles per site.

\begin{teo}\textbf{\cite[Theorem 5.1]{cf:BZ3}} \label{th:main}$ $\\
Let $(X,K)$ be a continuous-time BRW and suppose that at least one of the following conditions holds
\begin{enumerate}[(1)]
\item  $(X,K)$ is 
quasi transitive;
\item $(X,K)$ is irreducible and there exists
$\gamma$
bijection on $X$ such that
\begin{enumerate}[(a)]
\item
$\mu$ is $\gamma$-invariant;
\item for some $x_0\in X$ we have $x_0=\gamma^nx_0$ if and only if $n=0$.
\end{enumerate}
\end{enumerate}
Then
\[
\lim_{m\to+\infty}\lambda^m_s=\lambda_s
\ge\lim_{m\to+\infty}\lambda^m_w\ge\lambda_w.
\]
Moreover if $\lambda_s=\lambda_w$ then
$\lambda^m_w\downarrow_{m\to+\infty} \lambda_w$.
\end{teo}
\medskip

Let us consider now the global behavior; as before, we start
with a discrete-time process.
We take $(X,\mu)$ with $X=\Z \times Y$ (for some set $Y$) and we denote by $g:X \to \Z$ the usual
projection from $X$ onto $\Z$, namely $g(n,y):=n$.
In the following we use the same notation as in Section~\ref{subsec:FBRWs}.
We suppose that $\nu=\mu\circ g^{-1}$ is translation invariant 
(that is, $\gamma$-invariant according to Definition~\ref{def:quasitransitive} for every
translation operator $\gamma$ on $\Z$)
and
we denote the common distribution and the expected number of offsprings
of the BRW
by $\rho$ and $\bar \rho=\sum_{y \in X} m_{xy}$; observe that they.
do not depend on $x \in X$ or $i \in \Z$
since $\nu$ is translation invariant.

\begin{teo}\textbf{\cite[Theorem 6.5]{cf:Z1}} \label{th:zdriftglobal}
Let $X=\Z \times Y$ and suppose that the BRW $(X, \mu)$ is locally isomorphic to $(\Z, \nu)$ where  $\nu$ is
translation invariant.
If $m_{xy}=0$ whenever $|g(x)-g(y)|>1$ then
\begin{enumerate}[(1)]
\item
the BRW survives globally starting from $x$ if and only if $\bar \rho=\sum_{y \in \Z} m_{xy} > 1$;
\item
if the BRW survives globally (starting from $x$) then
then the BRW$_m$ survives globally (starting from $x$) for every sufficiently
large $m$.
\end{enumerate}
\end{teo}
Note that the hypotheses that we made in the previous theorem implies that the BRW $(\Z,\nu)$ is
``nearest neighbor'' in the sense that reproductions are possible only in the same site
or towards neighboring sites (in the usual graph $\Z$).
Theorem~\ref{th:zdriftglobal} applies to translation invariant BRWs on two particular graphs:
$\Z^d$ and the homogeneous tree $\mathbb{T}_{r}$ with degree $r$.
In particular the application to $\mathbb{T}_{r}$ is possible since the product
$\Z \times Y$ is meant as a set product and not a graph product; indeed
the set of vertices of $\mathbb{T}_{r}$
can be seen as $\Z^2$ and the projection $g$ as the horocyclic map
(see \cite[Section 12.13]{cf:Woess}).

\begin{cor}\textbf{\cite[Corollary 6.6]{cf:Z1}} \label{cor:zd}
If the BRW $(\Z^d,\mu)$ is translation invariant and there exists a projection $g$ on one of the coordinates
such that $m_{xy}=0$ whenever $|g(x)-g(y)|>1$, then
\begin{enumerate}[(1)]
\item
the BRW survives globally (starting from $x$) if and only if $\bar \rho=\sum_{y \in \Z} m_{xy} > 1$;
\item
if the BRW survives globally (starting from $x$)
then the BRW$_m$ survives globally (starting from $x$) for every sufficiently
large $m$.
\end{enumerate}
\end{cor}

\begin{cor}\textbf{\cite[Corollary 6.7]{cf:Z1}} \label{th:tree}
Let $\mathbb{T}_r$ be a homogeneous tree and suppose that the BRW $(\mathbb{T}_r,\mu)$ is
$\gamma$-invariant for every automorphism $\gamma$ of
 $\mathbb{T}_r$. 
If
$\mu_x(f)\not = 0$ implies $\mathrm{supp}(f) \subseteq B(x,1)$ (where $B(x,1)$
is the usual ball of radius 1 and center $x$ of the graph $\mathbb{T}_r$)
then
\begin{enumerate}[(1)]
\item
the BRW survives globally (starting from $x$) if and only if $\bar \rho=\sum_{y \in \Z} m_{xy} > 1$;
\item
if the BRW survives globally (starting from $x$) then
the BRW$_m$ survives globally (starting from $x$) for every sufficiently
large $m$.
\end{enumerate}
\end{cor}

Let us consider now the continuous-time case; there are analogous results  which
give an approximation of $\lambda_w$ by $\lambda_w^m$, where $\lambda_w^m$ is
the global critical parameter of the truncated BRW with (at most) $m$ particles per site.
From now on we deal with a site-breeding BRW; thus, $k(x)=k$ for all $x \in X$, that is,
we set $k_{xy}=k p(x,y)$ where $P$ is a transition matrix of a random walk.
We
stress that in this case $\lambda_w=1/k$.
We are concerned with the question whether
$\lambda_w^m\downarrow\lambda_w=1/k$ or not. Under the hypotheses of
Theorem~\ref{th:main}, this is the case when the BRW has no pure
global survival phase (i.e.~$r=1$ where $r$ is the spectral radius of the random walk $P$).
The interesting case is $r>1$. Most natural examples are drifting random walks on $\Z^d$
and the simple random walk on homogeneous trees: as for discrete-time processes,
in both cases we show that $\lambda^m_w\stackrel{m\to\infty}{\lra}\lambda_w$.

\begin{teo}\textbf{\cite[Theorem 6.1]{cf:BZ3}}\label{th:zdrift}
Let $P$ be a transition matrix of a nearest-neighbor, translation invariant random walk on $\Z$.
Then $\lim_{m\to+\infty}\lambda^m_w=1/k=\lambda_w$.
\end{teo}

This result immediately extends to the case of a class of
more general spaces (including multidimensional lattices $\Z^d$) in
the following way.

\begin{cor}\textbf{\cite[Corollary 6.1]{cf:BZ3}}\label{cor:zdcontinuous}
Let us consider the BRW $(Y\times\Z, K_\alpha)$
where $K_\alpha=\alpha ({\mathbb I}^Y\times P) +(1-\alpha) (Q\times {\mathbb
I}^\Z)$ and $Q$ and $P$ are transition matrices of a random walk on $Y$ and of a translation invariant random walk on
$\Z^d$ (as in
Theorem~\ref{th:zdrift}) respectively.
Then $\lim_{k\to+\infty}\lambda^k_w=1/k=\lambda_w$.
\end{cor}

For a homogeneous tree the following result holds.

\begin{teo}\textbf{\cite[Theorem 6.2]{cf:BZ3}}\label{th:treecontinuous}
If $X=\mathbb{T}_d$ is the homogeneous tree of
degree $d$ and $P$ is the simple random walk on $X$
then $\lim_{m\to+\infty}\lambda^m_w=1/k=\lambda_w$.
\end{teo}

Observe that Theorem~\ref{th:treecontinuous} can be immediately extended to the edge-breeding
BRW on $\mathbb{T}_d$ (on regular graphs, the edge-breeding BRW is a particular site-breeding BRW).
In the edge-breeding case we have that
$\lim_{m\to+\infty}\lambda^m_w=1/d=\lambda_w$; on the other hand, according to Theorem~\ref{th:main},
$\lim_{m\to+\infty}\lambda^m_s=1/2\sqrt{d-1}=\lambda_s$. Thus $\lambda_w^m<\lambda_s^m$ eventually as
$m \to \infty$ (pure global survival phase for truncated BRWs).
In particular in \cite{cf:Pem} it was shown that $\lambda_w^1< \lambda_s^1$, hence we conjecture that
$\lambda_w^m<\lambda_s^m$ for all $m \in \N$.

Discussing the details of the proofs goes beyond the purpose of this chapter. We just observe
that they rely on a comparison between the processes and a suitable oriented percolation.
Such a strategy has been introduced in \cite{cf:BD88} and widely used since then. The difference
in our case is that the percolation is not even one-dependent and this brings some additional
difficulties from a technical point of view (see \cite[Section 4]{cf:BZ3} and \cite[Section 6.1]{cf:Z1}).
Some applications and a slight generalization of these results of approximation can be found in
\cite[Theorem 3.4]{cf:BBZ} and \cite[Theorem 1]{cf:BLZ}.

\section{Proofs}
\label{sec:proofs}

Here we sketch the proofs of the new results.

\begin{proof}[Proof of Proposition~\ref{pro:maximumprinciple}]
Without loss of generality we can suppose that $\bar q(x)<1$ for all $x \in X$. Indeed,
given $x_0$ such that $\bar q(x_0)=1$ then for all $x \in \mathcal{N}_{x_0}$ we have
$\bar q(x)=1$. Since we defined $\widehat z(x):=1$ whenever $\bar q(x)=1$ we can remove
these vertices obtaining a new set $X^\prime \subseteq X$.
Consider the restricted BRW on $X^\prime$ (obtained by killing all the particle going outside $X^\prime$).
The generating function $G^\prime$ of the new BRW satisfies $G^\prime((z|_{X^\prime})|x) \ge G(z|x)$ for all
$x \in X^\prime$, hence $G(z) \ge z$ implies $G^\prime(z|_{X^\prime}) \ge z|_{X^\prime}$.
Moreover $\widehat z$ satisfies the conclusions of the proposition if and only if $\widehat{z|_{X^\prime}}\equiv
 \widehat z|_{X^\prime}$
does. Thus, it is enough to prove the result for the BRW restricted to $X^\prime$.

We use the notation of Section~\ref{subsec:nodeath}.
Note that $\widehat z:=T_{\bar q}^{-1}(z)$, 
thus $G(z) \ge z$  is equivalent to $\widehat G(\widehat z) \ge \widehat z$.
Hence it is enough to prove the proposition when $\mu_x(\mathbf{0})=0$ for all $x \in X$
which implies $\bar q=\mathbf{0}$ and $\widehat z= z$.
Suppose that $\mathcal{N}_x$ is nonempty, $z(y) \le z(x)$ for all $y \in  \mathcal{N}_x$
and $z(y_0)< z(x)$ for some $y_0 \in  \mathcal{N}_x$.
Then, using the fact that $z \le \mathbf{1}$ and that $\prod_{y \in X} z(y)^{f(y)} \le z(x)$ if $\mathcal{H}(f) \ge 1$,
we have that
$z(x) \le G(z|x) \le \sum_{f \in S_X:f(y_0)=0} \mu_x(f) z(x) +\sum_{f \in S_X:f(y_0)>0} \mu_x(f) z(y_0)<  z(x)$
which is a contradiction.
As for the second part, since $z(y) \le 1= z(x)$ for all $y \in X$ then we have $z(y)=1$ for all $y \in X$.
Finally, by induction we obtain the result for the set $\{y \in X:x \to y\}$.
\end{proof}

\begin{proof}[Proof of Theorem~\ref{th:equiv1local}]
 The first part of the theorem is \cite[Theorem 4.1]{cf:Z1}. The sufficient condition
in the second part follows easily from the first part. Clearly, it is equivalent to study the BRW restricted to
$Y:=\{w:x \to w \to y\}$ which is finite. In this case
$q(w,w)=1$ for all $w$ in a finite irreducible class implies a.s.~extinction in the class; if the number
of classes is finite then $q(x,y)=1$.
\end{proof}

Before proving
Proposition~\ref{pro:qtransitive} and Theorem~\ref{th:Fgraph}
we need two technical lemmas.

\begin{lem}\label{lem:convexity}
 Let $(X,\mu)$ be a BRW and fix $z,v \in [0,1]^X$ such that $z+\varepsilon v \in [0,1]^X$ for some
$\varepsilon>0$.
Then the function $t \mapsto G(z+vt|x)$ is strictly convex if and only if
\begin{equation}\label{eq:supp}
 \exists f : \mu_x(f)>0 ,\, \sum_{y \in \mathrm{supp}(v)} f(y) \ge 2, \,
\mathrm{supp}(z) \cup \mathrm{supp}(v)\supseteq \mathrm{supp}(f).
\end{equation}
\end{lem}


\begin{proof}[Proof of Lemma~\ref{lem:convexity}]
Let us evaluate the function $G$ on the line $t \mapsto z+tv$
where $t \in [0, T)$ and $T:=\sup\{s>0:z+sv \in [0,1]^X\}$.

\[ 
\begin{split}
G(z+tv|x)&=
\sum_{f \in S_X} \mu_x(f) \prod_{y \in X} \sum_{i=0}^{f(y)} \binom{f(y)}{i} z(y)^{f(y)-i}v(y)^i t^i\\
&=
\sum_{f \in S_X} \mu_x(f)  \sum_{g \in S_X:g \le f} \prod_{y \in X}\binom{f(y)}{g(y)} z(y)^{f(y)-g(y)}v(y)^{g(y)} t^{g(y)}\\
&=
\sum_{f \in S_X} \mu_x(f)  \sum_{g \in S_X:g \le f} t^{\mathcal H(g)} \prod_{y \in X}\binom{f(y)}{g(y)} z(y)^{f(y)-g(y)}v(y)^{g(y)}  \\
&=
\sum_{f \in S_X} \mu_x(f)  \sum_{i=0}^{\infty} \sum_{g \in S_X:\mathcal H (g)=i, g \le f}
t^{i} \prod_{y \in X}\binom{f(y)}{g(y)} z(y)^{f(y)-g(y)}v(y)^{g(y)}  \\
&=
\sum_{i=0}^\infty t^{i} \left ( \sum_{f,g \in S_X:\mathcal H (g)=i, g \le f}
\mu_x(f)  \prod_{y \in X}\binom{f(y)}{g(y)} z(y)^{f(y)-g(y)}v(y)^{g(y)} \right ) \\
\end{split}
\] 
The strict convexity of a power series in $t$ with nonnegative coefficients is equivalent to the strict positivity
of at least one coefficient corresponding to $t^i$ with $i \ge 2$.
Hence it is easy to show that each of the following assertions is equivalent to the next one and
that they are all equivalent to the strict
convexity of $t \mapsto G(z+vt|x)$
\begin{enumerate}
 \item $\exists f,g : \mathcal H(g) \ge 2, \, f \ge g, \, \mu_x(f)>0 : \mathrm{supp}(v) \supseteq \mathrm{supp}(g), \,
\mathrm{supp}(z) \supseteq \mathrm{supp}(f-g)$;
 \item $ \exists f,g : \mathcal H(g) \ge 2, \, f \ge g, \, \mu_x(f)>0 : g=f \ident_{\mathrm{supp}(v)}, \,
\mathrm{supp}(z) \supseteq \mathrm{supp}(f) \setminus \mathrm{supp}(v)$;
 \item $\exists f : \mu_x(f)>0 : \sum_{y \in \mathrm{supp}(v)} f(y) \ge 2, \,
\mathrm{supp}(z) \supseteq \mathrm{supp}(f) \setminus \mathrm{supp}(v)$;
 \item $\exists f : \mu_x(f)>0 : \sum_{y \in \mathrm{supp}(v)} f(y) \ge 2, \,
\mathrm{supp}(z) \cup \mathrm{supp}(v)\supseteq \mathrm{supp}(f)$;
 \item $ \exists f : \mu_x(f)>0 : \sum_{y \in \mathrm{supp}(v)} f(y) \ge 2, \,
\mathrm{supp}(z+v)\supseteq \mathrm{supp}(f)$.
\end{enumerate}
\end{proof}

\begin{lem}\label{lem:generationn}
 Let $(X, \mu)$ be a BRW and fix $x_0 \in X$. Suppose that for some
$\bar x$ in the same irreducible class of $x_0$ and $f \in S_X$ we have that $\mu_{\bar x}(f)>0$,
$\sum_{w:w \rightleftharpoons x_0} f(w) \ge 2$.
We can fix $\bar n \in\N$ such that
if the process starts with one particle at $x_0 \in X$ then we have at least 2 particles at
$x_0$ in the generation $\bar n$ wpp.
\end{lem}

\begin{proof}[Proof of Lemma~\ref{lem:generationn}]
Consider a path $x_0, x_1, \ldots, x_m=\bar x$ and let $f \in S_X$ be
such that $\mu_{\bar x}(f)>0$ and $\sum_{w:w \rightleftharpoons x_0} f(w) \ge 2$.
We can have two cases.

\smallskip

\noindent \textbf{(a)}.~There exists $x_{m+1} \in X$ such that $x_{m+1} \rightleftharpoons x_0$ and $f(x_{m+1}) \ge 2$; in this case consider the closed path
$x_0, x_{1}, x_{2}, \ldots, x_m, x_{m+1}, \ldots, x_n=x_0$ and take $\bar n:=n$. Since any particle at $x_i$
has at least one child at $x_{i+1}$ wpp and a particle at $\bar x$ has at least 2 children at $x_{m+1}$ wpp,
then any particle at $x_0$ has at least 2 descendants
at $x_0$ in the $\bar n$th generation.
Indeed, denote by $f_i \in S_X$ such that $\mu_{x_i}(f_i)>0$, $f_i(x_{i+1}) \ge 1$ for all $i =0, \ldots \bar n-1$
($f_m$ being $f$), then the probability  that a particle at $x_0$ has at least 2 particle at $x_0$
in the $\bar n$th generation is bounded from below by $\prod_{i=0}^m \mu_i(f_i) \prod_{j=m+1}^{\bar n-1} \mu_j(f_i)^2$.
\smallskip

\noindent \textbf{(b)}.~There exists a couple of different vertices
$x_{m+1}, y_{m+1}$ 
such that $x_{m+1}, y_{m+1}\rightleftharpoons x_0$ and
$f(x_{m+1}), f(y_{m+1}) \ge 1$; in
this case consider the paths $x_0, x_1, \ldots x_m, x_{m+1},
\ldots, x_{n_1}=x_0$ and $x_0, x_1, \ldots x_m, y_{m+1}, \ldots,
y_{n_2}=x_0$ and take $\bar n:= GCD(n_1,n_2)$ (the conclusion is
similar as before).
\end{proof}

\begin{proof}[Proof of Theorem~\ref{th:Fgraph}]
For every $z$ fixed point of $G$, we know that $z \ge \bar q$ and $z \le \mathbf 1_X$; this implies
that if $\sup_{x \in X} z(x)<1$ for some fixed point then necessarily $\sup_{x \in X} \bar q(x)<1$.
Hence, if $\bar q=\mathbf 1$ there is nothing to prove. Otherwise,
we show that if $G(z)=z$ and $z \not = \bar q$ then $\sup_{w \in X} z(w)=1$.
Suppose that  the BRW is locally isomorphic to $(Y,\nu)$ through the map $g$
and define $h(y):=\sup_{w \in g^{-1}(y)} z(w)$.
Clearly $h \in [0,1]^Y$ and $h \circ g \ge z$ which implies that
$G_Y(h) \ge h$.
Indeed
\[
\begin{split}
G_Y(h|y)&= \sup_{x \in g^{-1}(y)}G_Y(h|g(x))=\sup_{x \in g^{-1}(y)}G(h \circ g|x) \\
&\ge
\sup_{x \in g^{-1}(y)}G(z|x)= \sup_{x \in g^{-1}(y)} z(x) = h(y).
\end{split}
\]
If $Y$ finite then we can choose $\widetilde y \in Y$ which minimizes
\[
 t(y):= \frac{1-\bar q^Y(y)}{h(y)-\bar q^Y(y)}
\]
(where $t(y):=+\infty$ if $h(y)=\bar q^Y(y)$);
note that $t(y) \ge 1$ for all $y \in Y$ and $t(\widetilde y)<+\infty$.
By applying the maximum principle (Proposition~\ref{pro:maximumprinciple})
to the function $1/t(y)$ (where $y$ is ranging in the set $\{w:\bar q^Y(w)<1\}$)
we have that it is constant on $\{y:\widetilde y \to y\}$.
Since $\bar q^Y (\widetilde y)<1$ and $Y$ is finite,
then there exists $y_0$ such that $\widetilde y \to y_0$ and
there is local survival at $y_0$ starting from $y_0$.
Since $(Y,\nu)$ satisfies Assumption~\ref{assump:1}
then there exists $\bar y \rightleftharpoons y_0$ such that
a particle living at $\bar y$ wpp has at least 2 children in the
irreducible class of $y_0$.
Then by taking $y_0$ instead of $x_0$ in
Lemma~\ref{lem:generationn} we have that we can
find $\bar n \in \N$ such that the function
\[
\phi(t):=G_Y^{(\bar n)}(\bar q^Y+t (h-\bar q^Y)|y_0)-\bar q^Y(y_0)-t (h(y_0)-\bar q^Y(y_0))
\]
is strictly convex by Lemma~\ref{lem:convexity}.
Indeed $G_Y^{(\bar n)}$ is the generating function of the BRW constructed by considering
the $n$-th generations of the original BRW where $\bar n | n$ and, under our hypotheses, it satisfies
equation~\eqref{eq:supp}. 

Note that $\phi$ is well defined in $[0,t(y_0)]$ since
\[
r_t(y):=\bar q^Y(y)+t (h(y)-\bar q^Y(y)) \le \bar q^Y(y)+t(y_0) (h(y)-\bar q^Y(y))\le 1
\]
hence $r_t \in [0,1]^Y$ for all $t \in [0,t(y_0)]$.

Clearly every fixed point of $G_Y$ is a fixed point of $G_Y^{(\bar n)}$;
in particular, $G^{(\bar n)}(z)=z$ and $G^{(\bar n)}_Y(\bar q^Y)=\bar q^Y$, whence
$\phi(0)=0$ and $\phi(1)=G^{(\bar n)}_Y(h|y_0)-h(y_0)$.
Now, using equation~\eqref{eq:G-FBRWs}, $G^{(\bar n)}_Y(h) \ge h$
and this, in turn, implies $\phi(1) \ge 0$.
Since $\phi$ is strictly convex we have that $\phi(t)>0$ for all $t \in (1, t(y_0)]$.
If $t(y_0)>1$ then $0 < \phi(t(y_0))=G^{(\bar n)}_Y(r_{t(y_0)}|y_0)-1$ but this is a contradiction
since $r_{t(y_0)} \in [0,1]^Y$ and $G^{(\bar n)}_Y(r_{t(y_0)}) \in [0,1]^Y$.
In the end $t(y_0)=1$, thus $1=h(y_0)=\sup_{w \in X} z(w)$.
\end{proof}
Note that, from the previous proof, 
if the BRW on $Y$ is irreducible then by the maximum principle 
we have
that $(h-\bar q^Y)/(\mathbf{1}-\bar q^Y)$ is a constant function, thus 
$h(y)=\sup_{w \in g^{-1}(y)} z(w)=1$ for all $y \in Y$.

\begin{proof}[Proof of Proposition~\ref{pro:qtransitive}]
Since  $(X, E_\mu)$ is irreducible we have that $q(x,y)=q(x,x)$ for all $x,y \in X$ and
if $\bar q<\mathbf 1$ (resp.~$q(\cdot,y)<\mathbf 1$)
then $\bar q(x) <1$ (resp.~$q(x,y)<1$) for all $x \in X$.
Moreover, quasi transitivity implies that if $q(\cdot,y)<\mathbf 1$ then
$\sup_{x \in X} q(x,y) <1$.
Thus, according to Theorem~\ref{th:Fgraph}, $q(\cdot,y) \not = \mathbf 1$
implies $q(\cdot,y)=\bar q$.
\end{proof}

\begin{proof}[Proof of Theorem~\ref{th:MenshikovVolkov}]
According to Section~\ref{subsec:nodeath}, there is
strong local survival at $y$ starting from $x$ for the BRW
$\{\eta_n\}_{n \in \N}$
if and only if there is a.s.~local survival at $y$ for the associated BRW with no death,
that is, $\{\widehat \eta_n\}_{n \in \N}$
conditioned on $\mathcal{A}_x$ (global survival starting from $x$).
Moreover $v$ satisfies equation~\eqref{eq:MenshikovVolkov}
if and only if $v_1:=T_{\bar q}^{-1}v$ satisfies
\[
 \begin{cases}
  \widehat G(v_1|x) \ge v_1(x), & \forall x \in A^\complement,\\
  v_1(x_0) > \max_{x \in A} v_1(x) & \textrm{for some } x_0 \in A^\complement,
 \end{cases}
\] 
which is equation~\eqref{eq:MenshikovVolkov} in the case of the associated BRW with no death.
Hence it is enough to prove the result for the case $\rho_x(0)=0$ for all $x \in X$.

For completeness, we sketch the proof of \cite[Theorem 3.1]{cf:MenshikovVolkov}.
Suppose that there exist a function $v$ and a set $A$ as in the statement of
the theorem.
Recall the definition of $G$ given by equation~\eqref{eq:genfun}, define
$\widetilde Q_n:=\prod_{x \in X} v(x)^{\eta_n(x)}$ and $\sigma:=\min\{n \in \N: \sum_{x \in A} \eta_n(x)>0\}$,
where $\{\eta_n\}_{n \in \N}$ is a
realization of the BRW. As usual $\min \emptyset :=+\infty$. Let $Q_n:=\widetilde Q_{n \wedge \sigma}$.
If $\bar v(x):=\mathbb E [ \widetilde Q_{n+1} | \eta_n=\delta_x]$ then it is easy to show that
$\bar v
= G(v)$. 
Using the same arguments as in  \cite[Theorem 3.1]{cf:MenshikovVolkov}
we can show that $\{Q_n\}_{n\ \in \N}$ is a nonnegative supermartingale, hence there exists
$Q_\infty:=\lim_{n \to \infty} Q_n$ in $L^1$ and almost surely. Clearly $\mathbb E[Q_\infty] \ge \mathbb E[Q_0]$.
If $\eta_0:=\delta_{x_0}$ where $x_0 \not \in A$ satisfies the hypotheses of Theorem~\ref{th:MenshikovVolkov}
and if there were strong local survival then at least one particle would hit $A$ a.s., thus
$Q_\infty \le \max_{x \in A} v(x) < v(x_0)$ which is a contradiction. This yields the first part
of the proof.

Assume now that there is no strong local survival. Fix $\bar x \in X$ and $A:=\{\bar x\}$. Define
$v(x):=q_0(x,A)$, the probability of never hitting $A$ starting from $x$. Since the BRW is irreducible,
then there is no strong local survival if and only if $v(x)>0$ for some $x$. Clearly $v(\bar x)=0<v(x_0)$
for some $x_0 \not \in A$ and
\[
 v(x)= \sum_{g \in S_X: g(\bar x)=0} \mu_x(g) \prod_{y \in X} v(y)^{g(y)} \le G(v|x),
\quad \forall x \not = \bar x
\]
and the theorem is proved.
\end{proof}

\begin{proof}[Proof of Theorem~\ref{th:nonam}]
 The proof is essentially the same as in \cite[Section 3.3]{cf:BZ}. We just sketch
the main steps. 
Let $Y$ be the finite set onto which $X$ can be mapped by definition of $\mathcal F$-BRW.
Instead of the operators $N$ and $\widetilde N$ we use
$Mf(x):=\sum_{w \in X} m_{xw}f(w)$ (for all $x \in X$)
and  $\widetilde Mf(y):=\sum_{w \in Y} \widetilde m_{yw}f(w)$ (for all $y \in Y$)
where $m_{xw}=m^X_{xw}$ and $\widetilde m_{yw}=m^Y_{yw}$. These are well-defined,
bounded, linear operators from $l^2(X)$ and $l^2(Y)$ into itself respectively.
One can prove that $\|M\|=\rho(M)=M_s$ where $\rho(M)$ is the spectral radius of the operator
(this can be done as in \cite[Lemma 3.3]{cf:BZ} and \cite[Lemma 2.2]{cf:Stacey03}).
Similarly $\|\widetilde M\|=\rho(\widetilde M)=\widetilde M_s=\widetilde M_w=M_w$
(here we use the finiteness of $Y$).

Moreover, it can be shown that
for nonamenable BRWs
there exists $c>0$ such that, for all $f\in l^2(X)$,
\[
\| f\|_{D(2)}\ge c\| f\|_2,
\]
where the Dirichlet norm is defined as
\[
\| f\|_{D(2)} = \left( \sum_{x,y\in X} m_{xy}|f(x)-f(y)|^2
\right)^{1/2}.
\]
The proof of this inequality is analogous to the one of
\cite[Theorem 2.6]{cf:Stacey03}, the only difference being
the presence of $m_{xy}$ which can be easily dealt with.

The rest of the proof is the same as \cite[Theorem 3.6]{cf:BZ} using the term
$\sum_{x \in S, y \in S^\complement} m_{xy}$ instead of $|\partial_E S|$ and using the
graph $G^2$ induced by
$M^2$. 
\end{proof}

\begin{proof}[Proof of Theorem~\ref{th:subexponentialgrowth}]
Note that $\kappa(x)k^{(n)}_{xy}=\kappa(y)k^{(n)}_{yx}$ for all $x,y \in X$,
$n \in \N$.
Moreover, by the Cauchy-Schwartz inequality, for all $n \in \N$,
\[
K_s(x_0,x_0)^{2n}\ge k^{(2n)}_{x_0x_0}=\sum_{y \in X}k^{(n)}_{x_0y}k^{(n)}_{yx_0}=
\sum_{y \in B(x_0,n)} (k^{(n)}_{x_0y})^2\frac{\kappa(x_0)}{\kappa(y)}\ge
\frac{\left(\sum_yk^{(n)}_{x_0y}\right)^2}{c_n|B(x_0,n)|}.
\]
Hence
\[
K_w(x_0)=\liminf_n\sqrt[n]{\sum_yk^{(n)}_{x_0y}}=
\liminf_n\sqrt[2n]{\frac{\left(\sum_yk^{(n)}_{x_0y}\right)^2}{c_n|B(x_0,n)|}} \le K_s(x_0,x_0).
\]
\end{proof}

\begin{proof}[Proof of Proposition~\ref{pro:F-BRW1}]
\textit{(1)} and \textit{(2)} follow easily from Theorems~\ref{th:equiv1local}(1) and \ref{th:equiv1global3}(2).
As for \textit{(3)},
 we note that $m^{(n)}(x,y)= \bar \rho^n p^{(n)}(x,y)$ and that the generating
function $\Phi$ defined in Section~\ref{subsec:genfun} satisfies $\Phi(x,y|t)=F(x,y|t \bar \rho)$.
Thus, $\Phi(x,x|1)>1$, which is equivalent to local survival at $x$, is equivalent to
$F(x,x|\bar \rho) >1$.
\end{proof}

\begin{proof}[Proof of Lemma~\ref{lem:test0}]
Clearly $\prod_{i \in \N} (1-\alpha_i)^{k_i}>0$ if and only if $\sum_{i \in \N} k_i \log(1-\alpha_i)>-\infty$.
Observe that $\log(1-x) \le -x$ for all $x < 1$ hence
\[ 
  \sum_{i \in \N} k_i \alpha_i \le -\sum_{i \in \N} k_i \log(1-\alpha_i)<\infty.
\] 

If $\alpha_i \in [0, 1)$ and $k_i \ge 1$ eventually as $i \to \infty$ then
there is no loss of generality by assuming that $\alpha_i \in [0, 1)$ and $k_i \ge 1$ for all $i$.
In this case, since $k_i \ge 1$ both sides imply $\alpha_i \to 0$. Thus $\log(1-\alpha_i) \sim -\alpha_i$ and
\[
 \sum_{i \in \N} k_i \log(1-\alpha_i) > -\infty \Longleftrightarrow \sum_{i \in \N} k_i \alpha_i < \infty.
\]
\end{proof}

\begin{proof}[Proof of Theorem~\ref{th:recurrent}]
Since $\mu$ satisfies equation~\eqref{eq:particular1} 
then
$G(z| x) = \sum_{n=0}^\infty \rho_x(n) (Pz(x))^n$ where
$Pz(x)=\sum_{y \in X} p(x,y)z(y)$.
On the other hand, $\rho_x=\rho$ for all $x \in X$,
thus  $\bar q(x)=\bar q$ for all $x \in X$ where $\bar q$ is
the smallest fixed point in
$[0,1]$ of $t \mapsto G(t \mathbf{1}| x) \equiv \sum_{n=0}^\infty \rho(n) t^n=:F(t)$.
Clearly, any fixed point $z$ of $G$ must satisfy the inequality $z(x) \ge \bar q(x)=\bar q$.
Since $F(t) < t$ for all $t \in (\bar q,1)$ then
\[
z(x)=G(z|x)= F(Pz(x)) \le Pz(x),
\]
hence $z$ is a bounded subharmonic function. It is well known that the existence
of non-constant subharmonic functions which are bounded from above is equivalent to transience,
thus, in the recurrent case we have necessarily $z=t \mathbf 1$ which implies that $t=F(t)$
and $t \in \{\bar q, 1\}$.
Suppose that $\bar q<1$, since the random walk is recurrent,
then $q_0(\cdot, A) \le \bar q \mathbf{1}$ (for all $A \subseteq X$),
hence by Remark~\ref{rem:strongconditioned} $q(\cdot, A) = \bar q \mathbf{1}$
which is equivalent to strong local survival in $A$.
\end{proof}


\begin{thebibliography}{99}



\bibitem{cf:AthNey}
K.B.~Athreya,  P.E.~Ney,
Branching processes,
Die Grundlehren der mathematischen Wissenschaften, \textbf{196},
Springer-Verlag, 1972.


\bibitem{cf:BBZ}
D.~Bertacchi, L.~Belhadji, F.~Zucca,
A self-regulating and patch subdivided population,
Adv.~Appl.~Probab.~\textbf{42} n.3 (2010), 899--912.

\bibitem{cf:BLZ}
D.~Bertacchi, N.~Lanchier, F.~Zucca,
Contact and voter processes on the infinite
percolation cluster as models of host-symbiont interactions,
Ann.~Appl.~Probab.~\textbf{21} n.~4 (2011), 1215--1252.

\bibitem{cf:BPZ}
D.~Bertacchi, G.~Posta, F.~Zucca,
        Ecological equilibrium for restrained random walks,
        Ann.~Appl.~Probab.~\textbf{17} n.~4 (2007), 1117--1137.

\bibitem{cf:BZ03}
D.~Bertacchi, F.Zucca,
Uniform asymptotic estimates of transition probabilities on combs,
J. Aust. Math. Soc.~\textbf{75} n.~3 (2003), 325--353.

\bibitem{cf:BZ}
D.~Bertacchi, F.~Zucca,
Critical behaviors and critical values of branching random walks
on multigraphs, J.~Appl.~Probab.~\textbf{45} (2008), 481--497.

\bibitem{cf:BZ2}
D.~Bertacchi, F.~Zucca,
Characterization of the critical values of branching random walks on
weighted graphs through infinite-type branching processes,
J.~Stat.~Phys.~\textbf{134} n.~1 (2009), 53--65.

\bibitem{cf:BZ3}
D.~Bertacchi, F.~Zucca,
Approximating critical parameters of branching random walks,
J.~Appl.~Probab.~\textbf{46} (2009), 463--478.


\bibitem{cf:BD88}
M.~Bramson, R.~Durrett,
A simple proof of the stability criterion of Gray
and Griffeath,
  Probab.~Theory Related Fields \textbf{80}  (1988),  no. 2, 293--298.

\bibitem{cf:Big1977}
J.D.~Biggins,
Martingale convergence in the branching random walk,
J.~Appl.~Probab.~\textbf{14}  n.~1 (1977), 25--37.

\bibitem{cf:Big1978}
J.D.~Biggins,
The asymptotic shape of the branching random walk,
Adv.~Appl.~Probab.~\textbf{10} n.~1 (1978), 62--84.


\bibitem{cf:BigKypr97}
J.D.~Biggins, A.E.~Kyprianou,
Seneta-Heyde norming in the branching random walk,
Ann.~Probab.  \textbf{25} n.~1 (1997), 337--360.

\bibitem{cf:BigRah05}
J.D.~Biggins, A.~Rahimzadeh Sani,
Convergence results on multitype, multivariate branching random walks,
Adv.~Appl.~Probab.  \textbf{37} n.~3 (2005), 681--705.


\bibitem{cf:CMP98}
F.~Comets, M.V.~Menshikov, S.Yu.~Popov,
One-dimensional branching random walk in random environment: A classification,
Markov Process.~Related Fields~\textbf{4} (1998), 465--477.

\bibitem{cf:FN1}
K.~H.~F\"orster, B.~Nagy,
On the Collatz-Wielandt numbers and the local spectral radius of a nonnegative operator,
Proceedings of the Fourth Haifa Matrix Theory Conference (Haifa, 1988),
Linear Algebra Appl.~\textbf{120} (1989), 193--205.

\bibitem{cf:FN2}
K.~H.~F\"orster, B.~Nagy,
Local spectral radii and Collatz-Wielandt numbers of monic operator polynomials with nonnegative coefficients,
Linear Algebra Appl.~\textbf{268} (1998), 41--57.

\bibitem{cf:GW1875}
F.~Galton, H.W.~Watson, On the probability of the extinction of
families, Journal of the Anthropological Institute of Great Britain and Ireland \textbf{4} (1875), 138--144.

\bibitem{cf:GMPV09}
N.~Gantert, S.~M\"uller, S.Yu.~Popov, M.~Vachkovskaia,
Survival of branching random walks in random environment,
to appear on J.~Theor.~Probab., arXiv:0811.1748v3.

\bibitem{cf:Harris48}
T.E.~Harris,
Branching processes,
Ann.~Math.~Statistics \textbf{19} (1948). 474--494.

\bibitem{cf:Harris63}
T.E.~Harris, The theory of branching processes, Springer-Verlag, Berlin, 1963.

\bibitem{cf:DHMP99}
F.~den Hollander, M.V.~Menshikov, S.Yu.~Popov,
A note on transience versus recurrence for a branching random walk in random environment,
J.~Stat.~Phys.~\textbf{95}  (1999), 587--614.

\bibitem{cf:HuLalley}
I.~Hueter, S.P.~Lalley, {Anisotropic branching random walks
on homogeneous trees},  Probab.~Theory Related Fields  \textbf{116},
(2000),  n.1, 57--88.


\bibitem{cf:KestenSidorav}
H.~Kesten, V.~Sidoravicius,
Branching random walk with catalysts,
Electron.~J.~Probab.~\textbf{8} (2003), no. 5.

\bibitem{cf:Ligg1}
T.M.~Liggett,
{Branching random walks and contact processes on homogeneous trees},
Probab.~Theory Related Fields  \textbf{106},  (1996),  n.4, 495--519.

\bibitem{cf:Ligg2}
T.M.~Liggett,
{Branching random walks on finite trees},
Perplexing problems in probability,  315--330, Progr.~Probab.,
\textbf{44}, Birkh\"auser Boston, Boston, MA, 1999.



\bibitem{cf:MachadoMenshikovPopov}
F.P.~Machado,M.~V.~Menshikov, S.Yu.~Popov,
Recurrence and transience of multitype
branching random walks,
Stoch.~Proc.~Appl.~\textbf{91} (2001), 21--37.

\bibitem{cf:MP00}
F.P.~Machado, S.Yu.~Popov,
One-dimensional branching random walk in a Markovian random environment,
J.~Appl.~Probab.~\textbf{37} (2000), 1157--1163.

\bibitem{cf:MP03}
F.P.~Machado, S.Yu.~Popov,
Branching random walk in random environment on trees,
Stoch.~Proc.~Appl.~\textbf{106} (2003), 95--106.

\bibitem{cf:MadrasSchi}
N.~Madras, R.~Schinazi, {Branching random walks on trees},
Stoch.~Proc.~Appl.~\textbf{42},  (1992),  n.2, 255--267.

\bibitem{cf:Marek1}
I.~Marek,
Collatz-Wielandt numbers in general partially ordered spaces,
Linear Algebra Appl.~\textbf{173}, (1992), 165--180.

\bibitem{cf:MenshikovVolkov}
M.~V.~Menshikov, S.~E.~Volkov, Branching Markov chains: Qualitative characteristics,
Markov Proc.~and~rel.~Fields.~\textbf{3} (1997), 225–241.

\bibitem{cf:MountSchi}
T.~Mountford, R.~Schinazi, A note on branching random walks on
finite sets, J.~Appl.~Probab.~\textbf{42} (2005),  287--294.

\bibitem{cf:M08}
S.~M\"uller,
A criterion for transience of multidimensional branching random walk
in random environment,
Electron.~J.~Probab.~\textbf{13} (2008).


\bibitem{cf:Muller08-2}
S.~M\"uller, Recurrence for branching Markov chains,
Electron.~Commun.~Probab.~\textbf{13} (2008), 576--605.


\bibitem{cf:Pem}
R.~Pemantle, {The contact process on trees}, Ann.~Prob.
\textbf{20}, (1992), 2089--2116.

\bibitem{cf:PemStac1}
R.~Pemantle, A.M.~Stacey, {The branching random walk and
contact process on Galton--Watson and nonhomogeneous trees},
Ann.~Prob.~\textbf{29}, (2001),
 n.4, 1563--1590.


\bibitem{cf:Sen}
        E.~Seneta, {Non-negative matrices and Markov chains},
        Springer Series in Statistics, Springer, New York, 2006.

\bibitem{cf:Stacey03}
A.M.~Stacey, {Branching random walks on quasi-transitive graphs},
Combin.~Probab.~Comput.~\textbf{12}, (2003), n.3
 345--358.

\bibitem{cf:Woess}
        W.~Woess, {Random walks on infinite graphs and groups},
        Cambridge Tracts in Mathematics, {\textbf 138},
    Cambridge Univ.~Press, 2000.

\bibitem{cf:Woess09}
        W.~Woess,
Denumerable Markov chains,
Generating functions, boundary theory, random walks on trees.
EMS Textbooks in Mathematics,
European Mathematical Society (EMS), 2009.

\bibitem{cf:Zaehle}
I.~Z\"ahle, Renormalizations of branching random walks in equilibrium,
Electron.~J.~Probab.~\textbf{7} (2002), no. 7.

\bibitem{cf:Z1}
F.~Zucca,
Survival, extinction and approximation of discrete-time branching random walks,
J.~Stat.~Phys., \textbf{142} n.4 (2011), 53-65.


\end{thebibliography}
\end{document}